\documentclass[a4paper,reqno]{amsart}

\usepackage[T1]{fontenc}
\usepackage{amsmath,amsthm,amssymb,mathtools}
\usepackage[alphabetic]{amsrefs}
\usepackage{booktabs}
\usepackage{caption}
\usepackage{url}

\usepackage[shortlabels]{enumitem}
    \setenumerate[0]{label={\rm (\roman*)}, leftmargin=*, itemsep=.5ex}

\usepackage{tikz}
    \usetikzlibrary{calc}
\usepackage[edge-length=1cm]{dynkin-diagrams}

\usepackage{xcolor}
\usepackage[textsize=footnotesize,textwidth=20ex,colorinlistoftodos,backgroundcolor=lightgray]{todonotes}

\newcommand{\one}{\mathbf{1}}

\usepackage[hidelinks]{hyperref}
\usepackage[capitalise]{cleveref}

\theoremstyle{plain}
\newtheorem{theorem}{Theorem}[section]
\newtheorem*{theorem*}{Theorem}
\newtheorem{lemma}[theorem]{Lemma}
\newtheorem{proposition}[theorem]{Proposition}

\theoremstyle{definition}
\newtheorem{remark}[theorem]{Remark}
\newtheorem{definition}[theorem]{Definition}
\newtheorem{example}[theorem]{Example}

\newcommand{\fus}{\circledast}

\newcommand{\pprod}{\odot}
\newcommand{\ax}{e}
\newcommand{\idx}{\mathfrak{I}}
\newcommand{\idxe}{\mathfrak{i}}

\newcommand{\pa}{\mathbin{\diamond}}
\newcommand{\funit}{\mathbf{1}}
\newcommand{\sspec}{\ddag}
\newcommand{\spec}{\dag}

\newcommand{\mL}{\mathcal{L}}
\newcommand{\mA}{\mathcal{A}}
\newcommand{\mV}{\mathcal{V}}
\newcommand{\mB}{\mathcal{B}}
\newcommand{\mF}{\mathcal{F}}
\newcommand{\mH}{\mathcal{H}}
\newcommand{\mI}{\mathcal{I}}

\newcommand{\fJ}{\mathfrak{J}}
\newcommand{\fS}{\mathfrak{S}}
\newcommand{\fZ}{\mathfrak{Z}}
\newcommand{\fA}{\mathfrak{A}}
\newcommand{\fV}{\mathfrak{V}}
\newcommand{\fH}{\mathfrak{H}}
\newcommand{\sll}{\mathfrak{sl}}

\newcommand{\Z}{\mathbb{Z}}
\newcommand{\C}{\mathbb{C}}

\newcommand{\dash}{\nobreakdash-\hspace{0pt}}
\newcommand{\llangle}{\langle\!\langle}
\newcommand{\rrangle}{\rangle\!\rangle}
\newcommand{\conj}[2]{\prescript{#1}{}{#2}}

\DeclareMathOperator{\Aut}{Aut}
\DeclareMathOperator{\GL}{GL}
\DeclareMathOperator{\PSL}{PSL}
\DeclareMathOperator{\SL}{SL}
\DeclareMathOperator{\id}{id}
\DeclareMathOperator{\Char}{char}
\DeclareMathOperator{\Res}{Res}
\DeclareMathOperator{\ad}{ad}
\DeclareMathOperator{\ch}{ch}
\DeclareMathOperator{\Hom}{Hom}

\DeclareMathOperator{\tr}{tr}
\DeclareMathOperator{\End}{End}

\DeclareMathOperator{\Ind}{Ind}

\DeclareMathOperator{\Int}{Int}
\DeclareMathOperator{\HInt}{\widehat{\Int}}
\DeclareMathOperator{\Irr}{Irr}

\DeclarePairedDelimiter{\abs}{\lvert}{\rvert}
\makeatletter
\let\oldabs\abs%
\def\abs{\@ifstar{\oldabs}{\oldabs*}}
\makeatother

\allowdisplaybreaks%

\title[Frobenius algebras for simply laced Chevalley groups]{Non-associative Frobenius algebras for simply laced Chevalley groups}

\author[T. De Medts]{Tom De Medts}
\email{tom.demedts@ugent.be}

\author[M. Van Couwenberghe]{Michiel Van Couwenberghe}
\thanks{Michiel Van Couwenberghe is a Ph.\@~D.\@~fellow of the Research Foundation -- Flanders (FWO)} 
\email{michiel.vancouwenberghe@ugent.be}

\address{
    \parbox{\linewidth-\parindent}{
    Ghent University \\
    Department of Mathematics: Algebra and Geometry \\
    \mbox{Krijgslaan} 281--S25 \\
	9000 Gent \\
    Belgium}}

\keywords{non-associative algebras, Frobenius algebras, Chevalley groups, Lie algebras, E8, representation theory, axial algebras, decomposition algebras}
\makeatletter 
\@namedef{subjclassname@2020}{\textup{2020} Mathematics Subject Classification}
\makeatother
\subjclass[2020]{20F29, 17A36, 17D99, 17B10, 20G05}

\date{May 5, 2020}

\begin{document}

\begin{abstract}
    We provide an explicit construction for a class of commutative, non-associative algebras for each of the simple Chevalley groups of simply laced type. Moreover, we equip these algebra with an associating bilinear form, which turns them into Frobenius algebras. This class includes a 3876\dash{}dimensional algebra on which the Chevalley group of type $E_8$ acts by automorphisms. We also prove that these algebras admit the structure of (axial) decomposition algebras.
\end{abstract}

\maketitle

\section*{Introduction}

In 1982, Robert R. Griess constructed the largest sporadic simple group as the automorphism group of a commutative non-associative algebra, the \textit{Griess algebra}~\cite{Gri82}. In 1985, John Conway gave a different construction of the Griess algebra and observed that it has a peculiar feature~\cite{Con85}. It is generated by idempotents whose action by multiplication gives rise to a decomposition of the algebra that obeys a certain \textit{fusion law} with respect to the multiplication. Alexander A. Ivanov axiomatized this property in his definition of \textit{Majorana algebras}~\cite{Iva09}. Only recently, the definition has been further generalized to the definition of an \textit{axial algebra}~\cite{HRS15} by Jonathan I. Hall, Felix Rehren and Sergey Shpectorov. A further generalization, called (axial) \textit{decomposition algebras}, was given by Tom De Medts, Simon F. Peacock, S. Shpectorov and Michiel Van Couwenberghe~\cite{DMPSVC19}. The subject has received a lot of attention recently and there is an ongoing search for axial decomposition algebras admitting a given (simple) group as automorphism group.

Completely unrelated to this, Skip Garibaldi and Robert M. Guralnick observed%
\footnote{As Arjeh Cohen pointed out to us, this fact was already known to experts including Tonny Springer, but we could not find an earlier reference.}
that there exists a 3875-dimensional, commutative, non-associative, Frobenius algebra on which the simple algebraic group of type $E_8$ acts by automorphisms~\cite{GG15}. To the best of our knowledge, however, no explicit construction of this algebra was known.

Our paper aims to connect both worlds and shed light on the structure of this 3875\dash{}dimensional algebra. On the one hand, we will give an explicit construction of this algebra. On the other hand, we will be able to give it the structure of an axial decomposition algebra. Furthermore, this algebra fits into a larger class of algebras: the construction can be applied to any simple group of Lie type of type $ADE$ and each of these algebras will have the structure of a decomposition algebra.

\subsection*{Organization of the paper.}

We start in \cref{sec:S2} by defining a commutative product on the symmetric square of a simple Lie algebra of simply laced type. This definition arises very naturally from the definition of the Lie bracket and will be the starting point for the construction of our algebra.

In \cref{sec:construction}, we will give an explicit, albeit impractical, construction of our algebra. The underlying module of the algebra will be a subrepresentation of the symmetric square of the Lie algebra. We will define a product on it by embedding it into the symmetric square, using the product from \cref{sec:S2} and projecting it back onto our subrepresentation.
We have postponed a technical character computation to \cref{app:charV}.

In \cref{sec:0sub,sec:extending}, we provide an explicit multiplication rule for our algebra. We start in \cref{sec:0sub} by constructing a very small subalgebra which we use in \cref{sec:extending} to define the full algebra.

The algebra turns out to be unital, a fact that we prove in \cref{sec:unit}.

A very brief introduction into the realm of (axial) decomposition algebras is given in \cref{sec:decompositionalgebras}. \cref{sec:0dec,sec:dec} explain how to make our algebra into a decomposition algebra. Once again, we do this for the small subalgebra first and then extend our results to the full algebra.

Since the algebra for $E_8$ has been our main algebra of interest, we give it some more attention in \cref{sec:E8}. We prove that the algebra is, in fact, an axial decomposition algebra. In particular, we prove the following, cf. \cref{thm:E8}.

\begin{theorem*}
    There exist a one-parameter family of non-associative, commutative, unital 3876\dash{}dimensional algebras $(A,\pprod)$ on which the complex Chevalley group of type $E_8$ acts by automorphisms. Each of these contains a set $\Omega$ of idempotents. For each idempotent $e \in \Omega$, there exists a decomposition $A = \bigoplus_{1 \leq i \leq 6} A_i^e$ of $A$ as a vector space. Moreover, $a \pprod e = \lambda_i a$ when $a \in A_i^e$ for $\lambda_1=1$, $\lambda_2=0$, $\lambda_3=\frac{4}{3}c_1-\frac16$, $\lambda_4=\lambda_5=\frac12$ and $\lambda_6=c_1$ where $c_1 \in \C$ depends on the parameter. The linear map defined by
    \[
        {\tau_e}(a) \coloneqq \begin{cases}
            a & \text{if $a \in \bigoplus_{1 \leq i \leq 4} A_i^e$}, \\
            -a & \text{if $a \in A_5^e \oplus A_6^e$},
        \end{cases}
    \]
    defines an automorphism of $(A,\pprod)$. These automorphism $\tau_e$ for $e \in \Omega$ generate the complex Chevalley group of type $E_8$.
\end{theorem*}

\subsection*{Acknowledgment}

The authors are grateful to Sergey Shpectorov from the University of Birmingham for his suggestions during one of his visits at Ghent University. He provided the main idea for the construction in \cref{sec:construction}.

\section{A product on the symmetric square}%
\label{sec:S2}

Throughout, we will use the word ``algebra'' to mean a vector space equipped with a bilinear product.
We do not assume this product to be associative nor our algebra to be unital.
However, most of our algebras come equipped with a special bilinear form which turns them into Frobenius algebras.

\begin{definition}
    \begin{enumerate}
        \item A \textit{Frobenius algebra} is a triple $(V,\theta,\eta)$ where $V$ is a vector space over a field $k$, \[\theta \colon V \times V \to V \colon (v,w) \mapsto v\ast w\] is a bilinear product and \[\eta \colon V \times V \to k \colon (v,w) \mapsto \langle v ,w \rangle\] is a non-degenerate symmetric bilinear form such that
            \begin{equation}\label{eq:frob}
                v_1 \ast v_2 = v_2 \ast v_1 \quad \text{and} \quad \langle v_1 \ast v_2 ,v_3 \rangle = \langle v_1 , v_2 \ast v_3 \rangle
            \end{equation}
            for all $v_1,v_2,v_3 \in V$. The bilinear form $\eta$ is called the \emph{Frobenius form} for the algebra.
        \item Let $G$ be a group or a Lie algebra. We say that a Frobenius algebra $(V,\theta,\eta)$ is a Frobenius algebra \textit{for $G$} if $V$ is a linear $G$\dash{}representation and both $\theta$ and $\eta$ are $G$\dash{}equivariant, i.e.\ morphisms of $G$\dash{}representations.
        \item Let $(V_i,\theta_i,\eta_i)$ be Frobenius algebras for $i \in \{1,2\}$. We say that a linear map $\varphi: V_1 \to V_2$ is a \textit{morphism} of Frobenius algebras if $\theta_2(\varphi(v),\varphi(w)) = \varphi(\theta_1(v,w))$ and $\eta_2(\varphi(v),\varphi(w)) = \eta_1(v,w)$ for all $v,w \in V_1$. If both Frobenius algebras are Frobenius algebras for a group or a Lie algebra $G$, then we say that $\varphi$ is a morphism of Frobenius algebras for $G$ if, in addition, $\varphi$ is $G$-equivariant.
    \end{enumerate}
\end{definition}

\begin{example}%
    \label{ex:frob}
    \begin{enumerate}
        \item\label{ex:frob:jordan} Let $V$ be a finite-dimensional vector space over a field $k$ with characteristic $\Char(k) \neq 2$.
            Then we can equip the vector space $\End(V)$ also with the Jordan product defined by
            \[
                f \bullet g \coloneqq \frac{1}{2}(fg + gf)
            \]
            for all $f,g \in \End(V)$.
            This defines a Jordan algebra $(\End(V),\bullet)$.
            Let $\tr \colon \End(V) \to k$ denote the trace map.
            Then the bilinear form
            \[
                B \colon \End(V) \times \End(V) \colon fg \mapsto \tr(fg),
            \]
            is a Frobenius form for this algebra, i.e.\@ $(\End(V),\bullet,B)$ is a Frobenius algebra.
            This follows from the well-known identity $\tr(fgh) = \tr(ghf)$ for all $f,g,h \in \End(V)$.
            The proof of the non-degeneracy of this form is an easy exercise, see for example~\cite{Lam99}*{Example~16.57, p.~443}.
        \item\label{ex:frob:symm} Suppose that, in addition, $V$ itself is equipped with a non-degenerate bilinear form $\kappa$ and that $\Char(k) \neq 2$.
            Then we call an operator $f \in \End(V)$ \textit{symmetric} if $\kappa(f(a),b) = \kappa(a,f(b))$ for all $a,b \in V$ and \textit{antisymmetric} if $\kappa(f(a),b)=-\kappa(a,f(b))$ for all $a,b \in V$.
            Let $S$ (resp.\@ $A$) be the subspace of $\End(V)$ consisting of all symmetric (resp.\@ antisymmetric) operators; then $\End(V) = S \oplus A$ as vector spaces.
            Then $S$ is a subalgebra of the Jordan algebra $(\End(V),\bullet)$.
            Moreover $B(S,A) = 0$.
            Hence the restriction of $B$ to $S$ is non-degenerate.
            Therefore $(S,\bullet,B)$ is a Frobenius subalgebra of $(\End(V),\bullet,B)$.
    \end{enumerate}
\end{example}

We introduce some terminology and notation about Lie algebras that we will use throughout this paper.
The relevant definitions can be found in~\cite{Hum72}.

\begin{definition}%
    \label{def:stage}
    \begin{enumerate}
        \item Let $\mL$ be a complex simple Lie algebra of \textit{simply laced type}, i.e., of type $A_n$, $D_n$ or $E_n$.
            To avoid some technicalities that appear when working with low rank, we assume that $n \geq 3$, $n \geq 4$ or $n \in \{6,7,8\}$ when $\mL$ is of type $A_n$, $D_n$ or $E_n$ respectively.
            Consider a Cartan subalgebra $\mH$ of $\mL$ and the set of roots $\Phi \subseteq \mH^*$ relative to $\mH$.
            For each root $\alpha \in \Phi$, denote its coroot by $h_\alpha \in \mH$.
            Denote the weight lattice by $\Lambda$.
            Let $\Delta$ be a base for $\Phi$ and denote the set of positive roots with respect to~$\Delta$ by $\Phi^+$.
            To each root $\alpha \in \Phi$ we associate the reflection
            \[
                s_\alpha \colon \mH \to \mH \colon h \mapsto h - \alpha(h) h_\alpha
            \]
            across the root $\alpha$.
            The group $W$ generated by these reflections is called the Weyl group of $\Phi$.
        \item Let $\{ h_\alpha \mid \alpha \in \Delta \} \cup \{e_\alpha \mid \alpha \in \Phi\}$ be a Chevalley basis for $\mL$ with respect to $\mH$ and $\Delta$~\cite{Hum72}*{\S~25}. For $\alpha , \beta \in \Phi$ for which $\alpha + \beta \in \Phi$, define $c_{\alpha,\beta} \in \C$ such that $[e_\alpha,e_\beta] = c_{\alpha,\beta} e_{\alpha+\beta}$. Then
            \begin{align*}
                [h_\alpha,h_\beta] &= 0, \\
                [h_\alpha,e_\beta] &= \beta(h_\alpha) e_\beta, \\
                [e_{\alpha},e_{-\alpha}] &= h_\alpha \\
                [e_\alpha,e_\beta] &= c_{\alpha,\beta}e_{\alpha+\beta} & \text{if $\alpha + \beta \in \Phi$}, \\
                [e_\alpha,e_\beta] &= 0 &\text{if $\alpha + \beta \notin \Phi$},
            \end{align*}
            for all $\alpha,\beta \in \Phi$.
        \item For $\ell \in \mL$ let
            \[
                \ad_{\ell} : \mL \to \mL : l \mapsto [\ell,l].
            \]
            The requirement that $\mL$ is simple and of simply laced type is equivalent to the fact that $W$ acts transitively on $\Phi$.
            Therefore $t \coloneqq \tr(\ad_{h_\alpha} \ad_{h_\alpha})$ does not depend on the choice of $\alpha \in \Phi$.
            Define
                \[
                    \kappa(\ell_1,\ell_2) = 2t^{-1} \tr(\ad_{\ell_1}\ad_{\ell_2})
                \]
                for all $\ell_1, \ell_2 \in \mL$.
            Then $\kappa$ is a rescaling of the Killing form of $\mL$ such that $\kappa(h_\alpha,h_\alpha) = 2$ for all $\alpha \in \Phi$; in particular, $\kappa$ is non-degenerate.
            We will simply refer to $\kappa$ as the Killing form.
            This allows us to identify $\mH^*$ with $\mH$.
            Then $\alpha \in \mH^*$ corresponds to $h_\alpha$ under this identification.
            In particular, $\alpha(h_\beta) = \kappa(\alpha,h_\beta) = \kappa(\alpha,\beta)$ and thus, since $\Phi$ is simply laced,
                \[
                    \kappa(\alpha,\beta) =
                    \begin{cases}
                        -2 & \text{if $\alpha = -\beta$}, \\
                        -1 & \text{if $\alpha + \beta \in \Phi$}, \\
                        1 & \text{if $\alpha - \beta \in \Phi$}, \\
                        2 & \text{if $\alpha = \beta$}, \\
                        0 & \text{otherwise.}
                    \end{cases}
                \]
            Also note that \[\kappa(e_{-\alpha},e_{\alpha}) = \frac{1}{2}\kappa(e_{-\alpha},[e_\alpha,h_\alpha]) = \frac{1}{2} \kappa([e_{-\alpha},e_\alpha],h_\alpha) = \frac{1}{2}\kappa(h_\alpha,h_\alpha)=1.\]
        \item\label{def:stage:c} The structure constants $c_{\alpha,\beta}$ for $\alpha, \beta \in \Phi$ with $\alpha + \beta \in \Phi$ satisfy the following identities (see~\cite{Car72}*{Theorem~4.1.2}):
            \begin{enumerate}
                \item\label{prop:chevstrct:anti} $c_{\alpha,\beta} = -c_{\beta,\alpha}$,
                \item\label{prop:chevstrct:opp} $c_{\alpha,\beta} = -c_{-\alpha,-\beta}$,
                \item\label{prop:chevstrct:sum} $\frac{c_{\alpha,\beta}}{\kappa( \gamma, \gamma)} = \frac{c_{\beta,\gamma}}{\kappa( \alpha , \alpha ) } = \frac{c_{\gamma,\alpha}}{\kappa( \beta,\beta )}$ for all $\alpha,\beta,\gamma \in \Phi$ such that $\alpha + \beta + \gamma =0$.
            \end{enumerate}
            Since we assume that $\Phi$ is simply laced, we also have $c_{\alpha,\beta} = \pm 1$.
        \item\label{def:stage:S2} Let $S^2(\mL)$ be the symmetric square of $\mL$ considered as a representation for $\mL$.
            This means that $S^2(\mL)$ is the quotient of the $\mL$\dash{}representation $\mL \otimes \mL$ by the subrepresentation $\langle \ell_1 \otimes \ell_2 - \ell_2 \otimes \ell_1 \mid \ell_1,\ell_2 \in \mL\rangle$.
            We denote the image of $\ell_1 \otimes \ell_2$ under the natural projection onto $S^2(\mL)$ be $\ell_1\ell_2$.
            We can also view $S^2(\mL)$ as a subrepresentation of $\mL \otimes \mL$ by considering the section defined by
            \[
                \sigma \colon S^2(\mL) \to \mL \otimes \mL \colon \ell_1\ell_2 \mapsto \frac{1}{2}(\ell_1 \otimes \ell_2 + \ell_2 \otimes \ell_1)
            \]
            Denote the action of $\mL$ on $S^2(\mL)$ by $\cdot\;$:
            \[
                \ell \cdot \ell_1\ell_2 = [\ell,\ell_1]\ell_2 + \ell_1[\ell,\ell_2]
            \]
            for all $\ell,\ell_1,\ell_2 \in \mL$.
    \end{enumerate}
\end{definition}

We define a product and bilinear form on the symmetric square $S^2(\mL)$ starting from the Lie bracket defined on $\mL$ and the Killing form $\kappa$.
Recall that a product $\bullet$ on an $\mL$-module $A$ is called \textit{$\mL$-equivariant} if
\[ \ell \cdot (a \bullet b) = (\ell \cdot a) \bullet b + a \bullet (\ell \cdot b) \]
for all $a,b \in A$ and all $\ell \in \mL$, and that a bilinear form $B \colon A \times A \to \C$ is called \textit{$\mL$-equivariant} if
\[ B(\ell \cdot a , b) + B(a, \ell \cdot b ) = 0 \]
for all $a,b \in A$ and all $\ell \in \mL$.
Notice that the Killing form $\kappa \colon \mL \times \mL \to \C$ is $\mL$-equivariant (with respect to the adjoint action).

\begin{definition}%
    \label{def:prodS2}
    The non-degeneracy of the Killing form $\kappa$ allows us to identify $\mL$ with its dual $\mL^*$.
    Hence we can identify $\mL \otimes \mL$ with $\mL \otimes \mL^* \cong \Hom(\mL,\mL)$ via the isomorphism defined by
    \begin{equation*}
        \zeta^\prime \colon \mL \otimes \mL \to \Hom(\mL,\mL) \colon \ell_1 \otimes \ell_2 \mapsto \left[ \ell' \mapsto \kappa(\ell_2,\ell')\ell_1 \right].
    \end{equation*}
    Consider the $\mL$\dash{}equivariant section from \cref{def:stage}~\labelcref{def:stage:S2}:
    \[
        \sigma \colon S^2(\mL) \to \mL \otimes \mL \colon \ell_1\ell_2 \mapsto \frac{1}{2}(\ell_1 \otimes \ell_2 + \ell_2 \otimes \ell_1).
    \]
    Let $\zeta \coloneqq \zeta' \circ \sigma$.
    Then $\zeta$ is injective and its image consists of the symmetric operators.
    These are the operators $f \in \Hom(\mL,\mL)$ for which $\kappa(f(a),b) = \kappa(a,f(b))$ for all $a,b \in \mL$.
    We use the construction from \cref{ex:frob}~\labelcref{ex:frob:symm} to turn $S^2(\mL)$ into a Frobenius algebra.
    Under the above correspondence, the maps $\bullet$ and $B$ are defined by
    \begin{align*}
        \bullet \colon\; &S^2(\mL) \times S^2(\mL) \to S^2(\mL) \\ &(\ell_1\ell_2,\ell_3\ell_4) \mapsto \frac14 \left( \kappa(\ell_1,\ell_3)\ell_2\ell_4 + \kappa(\ell_1,\ell_4)\ell_2\ell_3 + \kappa(\ell_2,\ell_3)\ell_1\ell_4 + \kappa(\ell_2,\ell_4)\ell_1\ell_3 \right), \\
        B \colon\; &S^2(\mL) \times S^2(\mL) \to \C \\ &(\ell_1\ell_2,\ell_3\ell_4) \mapsto \frac12 \left( \kappa(\ell_1,\ell_3)\kappa(\ell_2,\ell_4) + \kappa(\ell_1,\ell_4)\kappa(\ell_2,\ell_3) \right).
    \end{align*}
\end{definition}

\begin{proposition}%
    \label{prop:prodS2}
    Consider the bilinear maps $\bullet$ and $B$ from \cref{def:prodS2}. Then $(S^2(\mL),\bullet,B)$ is a Frobenius algebra for $\mL$.
\end{proposition}
\begin{proof}
    The $\mL$\dash{}equivariance follows immediately from the $\mL$\dash{}equivariance of $\kappa$.
    Now, this follows immediately from the construction and \cref{ex:frob}~\labelcref{ex:frob:symm}.
\end{proof}

The following fact will be used throughout the following sections.
Recall that, by definition, an element $a \in S^2(\mL)$ is a \textit{weight vector} with \textit{weight} $w \in \mH^*$ if $h \cdot a= w(h) a$ for all $h \in \mH$.

\begin{lemma}%
    \label{lem:Bweights}
    Let $a,b \in S^2(\mL)$ be weight vectors (with respect to the Cartan subalgebra $\mH$ of $\mL$) with respective weights $w_a,w_b \in \mH^*$. Then
    \begin{enumerate}
        \item $a \bullet b$ is a weight vector with weight $w_a + w_b$,
        \item $B(a,b) = 0$ unless $w_a + w_b = 0$.
    \end{enumerate}
\end{lemma}
\begin{proof}
    \begin{enumerate}
        \item For all $h \in \mH$ we have \[h \cdot (a \bullet b) = (h \cdot a) \bullet b + a \bullet (h \cdot b) = (w_a + w_b)(h) (a \bullet b)\]
            because the product $\bullet$ is $\mL$\dash{}equivariant.
        \item Let $h \in \mH$ such that $(w_a+w_b)(h) \neq 0$.
            Because $B$ is $\mL$\dash{}equivariant, we have $w_a(h)B(a,b) = B(h \cdot a, b) = - B(a,h \cdot b) =- w_b(h)B(a,b)$.
            This implies that $B(a,b)=0$.
        \qedhere
    \end{enumerate}
\end{proof}

\section{Constructing the algebra}%
\label{sec:construction}

We will use the algebra from \cref{prop:prodS2} to build a Frobenius algebra of smaller dimension for $\mL$.
The highest occurring weight in $S^2(\mL)$, as an $\mL$\dash{}representation, is the double of a root.
Its weight space is one-dimensional.
We will explicitly determine a generating set of the subrepresentation $\mV$ generated by this weight space in \cref{prop:gensetV} below.
Next, we will define an algebra product on the complement $\mA$ of $\mV$ in $S^2(\mL)$ with respect to $B$.
The algebra product on $\mA$ will be the composition of the algebra product from \cref{prop:prodS2} and the projection onto~$\mA$.
We are grateful to Sergey Shpectorov for providing the central idea of this construction.

\begin{definition}
    Let $\mV$ denote the subrepresentation of $S^2(\mL)$ generated by $e_\omega e_\omega$, where $\omega$ is the highest root with respect to the base $\Delta$.
\end{definition}

It will be fairly straightforward to find elements that lie in $\mV$. However, in order to determine whether they span $\mV$ as a vector space, we will first have to determine the multiplicity of each weight in $\mV$, a task requiring some work. We will use the terminology of (formal) characters to describe these multiplicities; see~\cite{Hum72}*{\S~22.5}.

\begin{definition}
    Let $\Lambda$ be the weight lattice of $\Phi$ and consider the group ring $\Z[\Lambda]$. To avoid confusion, we denote the basis element of $\Z[\Lambda]$ corresponding to a weight $\lambda \in \Lambda$ by $e^\lambda$ (so in particular, $e^0=1$ and $e^\lambda e^\mu = e^{\lambda+\mu}$ for all $\lambda,\mu \in \Lambda$). Let $\mV$ be a representation for $\mL$. For each weight $\lambda \in \Lambda$, we denote its weight-$\lambda$-space by $\mV_\lambda$:
    \[
        \mV_\lambda \coloneqq \{ v \in \mV \mid h \cdot v = \lambda(h) v \text{ for all } h \in \mH \}.
    \]
    Write $m_\lambda \coloneqq \dim(\mV_\lambda)$.
    Let $\Pi$ be the set of weights $\lambda \in \Lambda$ for which $m_\lambda \neq 0$.
    Then we define the \textit{formal character} $\ch_\mV$  of $\mV$ as
    \[
        \ch_\mV = \sum_{\lambda \in \Pi} m_\lambda e^\lambda \in \Z[\Lambda].
    \]
\end{definition}

We introduce some notation to describe the weights and multiplicities of $S^2(\mL)$ and $\mV$.

\begin{definition}%
    \label{def:charV2}
    \begin{enumerate}
        \item For $-2 \leq i \leq 2$ let $\Lambda_i \coloneqq \{\alpha + \beta \mid \alpha,\beta \in \Phi, \kappa( \alpha , \beta ) = i \}$.
            This means that $\Lambda_i$ contains those weight vectors that can be represented as the sum of two roots $\alpha, \beta \in \Phi$ such that $\kappa(\alpha,\beta)=i$.
        \item For $\lambda \in \bigcup_{i=-2}^2 \Lambda_i$, let $N_\lambda \coloneqq \{\{\alpha,\beta\} \mid \alpha,\beta \in \Phi , \alpha + \beta = \lambda\}$ and $n_\lambda = \abs{N_\lambda}$.
            Simply put, $n_\lambda$ is the number of ways to write $\lambda$ as the sum of two roots.
    \end{enumerate}
\end{definition}

The elements of $\Lambda_i$ for $-2 \leq i \leq 2$ are the weights of $S^2(\mL)$ as an $\mL$\dash{}representation, cf. \cref{prop:charS}.
We prove a few easy statements about these weights.

\begin{lemma}%
    \label{lem:weights}
    \begin{enumerate}
        \item\label{lem:weights:disj} For each $\lambda \in \Lambda_i$, we have $\kappa(\lambda,\lambda) = 4+2i$. In particular, the sets $\Lambda_i$ are disjoint.
        \item\label{lem:weights:-2and-1} $\Lambda_{-2} = \{0\}$ and $\Lambda_{-1}=\Phi$.
        \item\label{lem:weights:kalphabeta} If $\alpha + \beta \in \Lambda_i$ for $\alpha, \beta \in \Phi$, then $\kappa(\alpha,\beta) = i$.
        \item\label{lem:weights:lambdaalpha} Let $\alpha \in \Phi$ and $\lambda \in \Lambda_i$. Then $\alpha + \beta = \lambda$ for some $\beta \in \Phi$ if and only if $\kappa(\alpha,\lambda)=2+i$.
    \end{enumerate}
\end{lemma}
\begin{proof}
    \begin{enumerate}
        \item For $\lambda \in \Lambda_i$, we can write $\lambda = \alpha + \beta$ where $\alpha, \beta \in \Phi$ and $\kappa(\alpha,\beta) = i$. Thus $\kappa(\lambda,\lambda) = \kappa(\alpha + \beta , \alpha + \beta) = 4+2i$.
        \item We have $\kappa(\alpha,\beta) = -2$ for $\alpha,\beta \in \Phi$ if and only if $\alpha = -\beta$. Therefore $\Lambda_{-2} = \{0\}$. Also $\Lambda_{-1} = \Phi$ because $\alpha + \beta \in \Phi$ for $\alpha , \beta \in \Phi$ if and only if $\kappa(\alpha,\beta) = -1$.
        \item By~\labelcref{lem:weights:disj} we know that $4+2i=\kappa(\alpha+\beta,\alpha+\beta)=4+2\kappa(\alpha,\beta)$ from which the assertion follows.
        \item This is obvious by~\labelcref{lem:weights:-2and-1} for $i \in \{-2,-1\}$. Suppose that $i\in\{0,1,2\}$. If $\lambda = \alpha + \beta$ for some $\beta \in \Phi$, then $\kappa(\alpha,\beta)=0$ and $\kappa(\lambda,\alpha)=2$ by~\labelcref{lem:weights:disj}. Conversely, suppose that $\kappa(\lambda,\alpha)=2$. Write $\lambda = \alpha'+\beta'$ for some $\alpha',\beta' \in \Phi$. Because $\Phi$ is simply laced we have either $\alpha \in \{\alpha',\beta'\}$ or $i=0$ and $\kappa(\alpha',\alpha)=\kappa(\beta',\beta)=1$. In the first case the condition is obviously satisfied. In the second case we have that $\alpha' - \alpha$ is a root and $\kappa(\alpha'-\alpha,\beta)=-1$. So $\lambda-\alpha=(\alpha'-\alpha)+\beta'$ is a root. \qedhere
    \end{enumerate}
\end{proof}

\begin{proposition}%
    \label{prop:charS}
    The character of $S^2(\mL)$ is given by
    \[
        \ch_{S^2(\mL)} = \left(\frac{n(n+1)}{2} + n_0\right) + \sum_{\lambda \in \Lambda_{-1}} (n_\lambda + n)e^\lambda +\sum_{\lambda \in \Lambda_0} n_\lambda e^\lambda + \sum_{\lambda \in \Lambda_1 \cup \Lambda_2} e^\lambda.
    \]
\end{proposition}
\begin{proof}
    If $\sum_\lambda m_\lambda e^\lambda$ is the character of a representation, then its symmetric square has character $\frac{1}{2} \sum_{\lambda,\mu} m_\lambda m_\mu e^{\lambda+\mu} + \frac12 \sum_{\lambda} m_\lambda e^{2\lambda}$; see~\cite{FH91}*{Exercise~23.39}.
    Since the character of $\mL$ as $\mL$\dash{}representation is given by
    \[
        n + \sum_{\alpha \in \Phi} e^\alpha,
    \]
    the statement follows from \cref{def:charV2}.
    It is also possible to verify this more explicitly. The Chevalley basis of $\mL$ is a basis of weight vectors of $\mL$ with respect to $\mH$. Now, if $b_1,\dots,b_n$ is a basis of weight vectors of $\mL$ as $\mL$-representation, then $b_i b_j$ for $i \leq j$ is a basis of weight vectors for $S^2(\mL)$ from which the character can be computed.
\end{proof}

We are now ready to specify the formal character of $\mV$.

\begin{proposition}%
    \label{prop:charV}
    The character of $\mV$ is given by
    \[
        \ch_\mV = n_0 + \sum_{\lambda \in \Lambda_{-1}} (n_\lambda + 1) e^{\lambda} + \sum_{\lambda \in \Lambda_0} (n_\lambda - 1)e^\lambda + \sum_{\lambda \in \Lambda_1 \cup \Lambda_2} e^\lambda.
    \]
\end{proposition}
\begin{proof}
    The character can be computed using Freudenthal's formula~\cite{Hum72}*{\S~22.3}. We refer to \cref{prop:charVapp} for the details.
\end{proof}

Next, we compute certain elements of $\mV$ and we use the character of $\mV$ to verify that these elements, in fact, span $\mV$ as a vector space.

\begin{proposition}%
    \label{prop:gensetV}
    Let
    \begin{align*}
        \Gamma_0 &\coloneqq \{2e_\alpha e_{-\alpha} - h_\alpha h_\alpha \mid \alpha \in \Phi\}, \\
        \Gamma_1 &\coloneqq \{e_\alpha h_\alpha \mid \alpha \in \Phi\}, \\
        \Gamma_2 &\coloneqq \{2e_\alpha e_\beta + c_{\alpha,\beta}e_{\alpha + \beta}(h_\beta - h_\alpha) \mid \alpha,\beta \in \Phi , \kappa( \alpha , \beta) = -1 \}, \\
        \Gamma_3 &\coloneqq \{ e_{\alpha}e_{\beta} + \frac{c_{\alpha,-\gamma}}{c_{\beta,-\delta}}e_\gamma e_\delta \mid \alpha,\beta,\gamma,\delta \in \Phi, \kappa( \alpha , \beta ) = 0 , \{\gamma,\delta\} \in N_{\alpha+\beta} \setminus \{\{\alpha,\beta\}\}\}, \\
        \Gamma_4 &\coloneqq \{e_{\alpha}e_{\beta} \mid \alpha,\beta \in \Phi, \kappa( \alpha , \beta ) = 1 \}, \\
        \Gamma_5 &\coloneqq \{e_{\alpha}e_{\alpha} \mid \alpha \in \Phi\}.
    \end{align*}
    Then $\Gamma_0 \cup \Gamma_1 \cup \Gamma_2 \cup \Gamma_3 \cup \Gamma_4 \cup \Gamma_5$ spans $\mV$ as a vector space.
\end{proposition}
\begin{proof}
    The Weyl group $W$ acts transitively on $\Phi$ and $e_\omega e_\omega \in \mV$. Therefore $\Gamma_5 \subseteq \mV$.
    For any $\alpha \in \Phi$, we have $e_{-\alpha} \cdot e_{\alpha} e_{\alpha} = -2 e_\alpha h_\alpha$, thus $\Gamma_1 \subseteq \mV$.
    Hence $e_{\alpha} \cdot e_{-\alpha}h_{-\alpha} = 2e_{\alpha} e_{-\alpha} - h_{\alpha}h_\alpha \in \mV$ for any $\alpha \in \Phi$, which shows that $\Gamma_0 \subseteq \mV$.
    Now let $\alpha,\beta \in \Phi$. Suppose that $\kappa( \alpha,\beta ) = 1$. Then $\alpha - \beta \in \Phi$ and $e_{\alpha-\beta} \cdot e_\beta e_\beta = 2c_{\alpha-\beta,\beta} e_{\alpha}e_{\beta} \in \mV$. Therefore $\Gamma_4 \subseteq \mV$.
    Suppose next that $\kappa(\alpha,\beta) = -1$ such that $\alpha + \beta$ is a root. Then $e_\beta \cdot e_\alpha h_\alpha  + e_\alpha \cdot e_\beta h_\beta = 2 e_\alpha e_\beta + c_{\alpha,\beta} e_{\alpha + \beta} (h_\beta-h_\alpha) \in \mV$. Hence $\Gamma_2 \subseteq \mV$.

    Finally, let $\alpha, \beta, \gamma, \delta \in \Phi$ such that $\kappa( \alpha , \beta ) = 0$ and $\{\gamma,\delta\} \in N_{\alpha + \beta} \setminus \{\{\alpha,\beta\}\}$. Then $\alpha , \beta , \gamma , \delta$ generate a root subsystem of type $A_3$, $\kappa( \gamma , \delta ) = 0$ and $\kappa( \alpha , \gamma ) = \kappa( \alpha , \delta ) = \kappa( \beta , \gamma ) = \kappa( \beta , \delta ) = 1$. Now $e_\delta \cdot (2e_{\gamma-\alpha}e_\alpha + c_{\gamma-\alpha,\alpha}e_{\gamma}(h_{\alpha} - h_{\gamma-\alpha})) = 2c_{\delta,\gamma-\alpha}e_\alpha e_\beta - 2 c_{\gamma-\alpha,\alpha}e_\gamma e_\delta$. Using the identities from \cref{def:stage}~\labelcref{def:stage:c}, we see that $c_{\delta,\gamma-\alpha} = -c_{\beta,-\delta}$ and $c_{\gamma-\alpha,\alpha} = c_{\alpha,-\gamma}$. This amounts to $\Gamma_3 \subseteq \mV$.

    In order to prove that $\Gamma \coloneqq \Gamma_0 \cup \Gamma_1 \cup \Gamma_2 \cup \Gamma_3 \cup \Gamma_4 \cup \Gamma_5$ spans $\mV$, it suffices to check that the elements of weight $\lambda$ in $\Gamma$ span $\mV_\lambda$, the weight-$\lambda$-space of $\mV$.
    The dimension $\dim(\mV_\lambda)$ of $\mV_\lambda$ can be derived from \cref{prop:charV}.

    The elements of $\Gamma$ of weight $0$ are precisely those contained in $\Gamma_0$.
    Obviously $\dim(\langle \Gamma_0 \rangle) = \frac{\abs{\Phi}}{2} = n_0$ and therefore $\langle \Gamma_0 \rangle = \mV_0$.

    Let $\alpha \in \Phi$.
    The elements of $\Gamma$ of weight $\alpha$ are $e_\alpha h_\alpha$ and the elements $2e_\beta e_\gamma + c_{\beta,\gamma}(h_\gamma - h_\beta)$ where $\{\beta,\gamma\} \in N_\alpha$.
    Since these elements are linearly independent, they span a subspace of dimension $n_\alpha + 1$, which is the dimension of $\mV_\alpha$ by \cref{prop:charV}.

    For $\lambda \in \Lambda_0$, the elements of $\Gamma$ of weight $\lambda$ are those of the form $e_{\alpha}e_{\beta} + \frac{c_{\alpha,-\gamma}}{c_{\beta,-\delta}}e_\gamma e_\delta$ where $\{\alpha,\beta\}$ and $\{\gamma,\delta\}$ are two different elements of $N_\lambda$.
    Hence they span a subspace of dimension at least $n_\lambda-1$ and at most $n_\lambda$.
    Since they are all contained in $\mV$, they span $\mV_\lambda$, a subspace of dimension $n_\lambda-1$ by \cref{prop:charV}.

    Finally, let $\lambda \in \Lambda_1 \cup \Lambda_2$.
    Then all elements of $S^2(\mL)$ of weight $\lambda$ are contained in $\langle \Gamma_4 \cup \Gamma_5 \rangle$.
    Therefore $\mV_\lambda \leq \langle \Gamma \rangle$.
\end{proof}

As we observed in the previous proof, the elements of $\Gamma_3$ of weight $\lambda \in \Lambda_0$ are linearly dependent.
We introduce some notation to describe this linear dependence; this will be useful later.

\begin{definition}%
    \label{def:f}
    \begin{enumerate}
        \item Recall from \cref{def:charV2} that
        \[ \Lambda_0 = \{ \alpha + \beta \mid \alpha, \beta \in \Phi, \kappa(\alpha,\beta) = 0 \} . \]
            For each $\lambda \in \Lambda_0$, fix elements $\alpha_\lambda,\beta_\lambda \in \Phi$ such that $\alpha_\lambda + \beta_\lambda = \lambda$.
            Note that it immediately follows that $\kappa(\alpha_\lambda,\beta_\lambda)=0$ from \cref{lem:weights}~\labelcref{lem:weights:kalphabeta}.
        \item For all $\alpha,\beta \in \Phi$ such that $\kappa(\alpha,\beta)=0$ (and therefore $\alpha+\beta \in \Lambda_0$) we write
            \[
                f_{\alpha,\beta} \coloneqq \begin{cases}
                    1 & \text{if $\{\alpha,\beta\} = \{\alpha_\lambda,\beta_\lambda\}$}, \\
                    -\frac{c_{\alpha,-\alpha_\lambda}}{c_{\beta,-\beta_\lambda}} & \text{otherwise}.
                \end{cases}
            \]
            Notice that $f_{\alpha,\beta} \in \{ \pm 1 \}$.
    \end{enumerate}
\end{definition}

\begin{proposition}%
    \label{prop:f}
    Let $\alpha,\beta \in \Phi$ such that $\kappa(\alpha,\beta)=0$ and let $\lambda \coloneqq \alpha + \beta \in \Lambda_0$.
    Then
    \begin{enumerate}
        \item\label{prop:f:comm} $f_{\alpha,\beta}=f_{\beta,\alpha}$ and,
        \item\label{prop:f:-l} $f_{-\alpha,-\beta}=f_{\alpha,\beta}f_{-\alpha_\lambda,-\beta_\lambda}$.
        \item Let $\Gamma_3$ be as in \cref{prop:gensetV}. The subspace of $\mV$ spanned by $\Gamma_3$ is equal to the subspace of $\mV$ spanned by
            \[
                \{e_{\alpha_\lambda}e_{\beta_\lambda} - f_{\alpha,\beta}e_\alpha e_\beta \mid \lambda \in \Lambda_0; \alpha,\beta \in \Phi; \alpha+\beta=\lambda\}.
            \]
    \end{enumerate}
\end{proposition}
\begin{proof}
    The elements of $\Gamma_i$ for $0 \leq i \leq 5$ are all weight vectors.
    Those with weight $\lambda$ are contained in $\Gamma_3$.
    Thus the weight-$\lambda$-space of $\mV$ is contained in the span of $\Gamma_3$.
    By \cref{prop:charV}, this weight space has dimension $n_\lambda-1$.
    Thus the elements $e_\alpha e_\beta + \frac{c_{\alpha,-\gamma}}{c_{\beta,-\delta}} e_\gamma e_\delta$ for $\alpha,\beta,\gamma,\delta \in \Phi$ with $\alpha + \beta = \gamma + \delta = \lambda$ and $\{\alpha,\beta\} \neq \{\gamma,\delta\}$ must be linearly dependent.
    In particular $e_\alpha e_\beta \notin \mV$ for all $\alpha,\beta \in \Phi$ where $\kappa(\alpha,\beta) = 0$.
    \begin{enumerate}
        \item This is obvious if $\{\alpha,\beta\}=\{\alpha_\lambda,\beta_\lambda\}$. Assume that $\{\alpha,\beta\} \neq \{\alpha_\lambda,\beta_\lambda\}$. Since $e_\alpha e_\beta = e_\beta e_\alpha$ it follows from the argument above that the elements $e_\alpha e_\beta + \frac{c_{\alpha,-\alpha_\lambda}}{c_{\beta,-\beta_\lambda}}e_{\alpha_\lambda}e_{\beta_\lambda}$ and $e_\beta e_\alpha + \frac{c_{\beta,-\alpha_\lambda}}{c_{\alpha,-\beta_\lambda}}e_{\alpha_\lambda}e_{\beta_\lambda}$ must be linearly dependent. Therefore $\frac{c_{\alpha,-\alpha_\lambda}}{c_{\beta,-\beta_\lambda}} = \frac{c_{\beta,-\alpha_\lambda}}{c_{\alpha,-\beta_\lambda}}$ and thus $f_{\alpha,\beta}=f_{\beta,\alpha}$.
        \item Since 
            \begin{align*}
                &e_{-\alpha}e_{-\beta}+\frac{c_{-\alpha,\alpha_\lambda}}{c_{-\beta,\beta_\lambda}}e_{-\alpha_\lambda}e_{-\beta_\lambda} \in \Gamma_3, \\
                &e_{-\alpha_\lambda}e_{-\beta_\lambda} - f_{-\alpha_\lambda,-\beta_\lambda}e_{\alpha_{-\lambda}}e_{\beta_{-\lambda}} \in \Gamma_3, \\
                &e_{-\alpha}e_{-\beta} - f_{-\alpha,-\beta} e_{\alpha_{-\lambda}}e_{\beta_{-\lambda}} \in \Gamma_3, \\
            \end{align*}
            these elements must be linearly dependent.
            Thus 
            \[
                f_{-\alpha,-\beta}=-\frac{c_{-\alpha,\alpha_\lambda}}{c_{-\beta,\beta_\lambda}}f_{-\alpha_\lambda,-\beta_\lambda}.
            \]
            The assertion follows because $c_{-\alpha,\alpha_\lambda}=-c_{\alpha,-\alpha_\lambda}$ and $c_{-\beta,\beta_\lambda}=-c_{\beta,-\beta_\lambda}$ (see \cref{def:stage}~\labelcref{def:stage:c}) and therefore \[-\frac{c_{-\alpha,\alpha_\lambda}}{c_{-\beta,\beta_\lambda}}=-\frac{c_{\alpha,-\alpha_\lambda}}{c_{\beta,-\beta_\lambda}}=f_{\alpha,\beta}.\]
        \item This follows immediately because the elements of $\Gamma_3$ of weight $\lambda$ span a subspace of dimension $n_\lambda-1$. \qedhere
    \end{enumerate}
\end{proof}

Next, we want to take a complement of $\mV$ in $S^2(\mL)$ with respect to the bilinear form $B$. In order for this complement to be well-defined, we need $B$ to be non-degenerate on $\mV$.

\begin{proposition}%
    \label{prop:nondeg}
    The restriction of $B$ to $\mV \times \mV$ is non-degenerate.
\end{proposition}
\begin{proof}
    Since $B$ is $\mL$-equivariant, the radical $\{v \in \mV \mid B(v,w) \text{ for all $w \in \mV$}\}$ of $B \restriction_{\mV \times \mV}$ is a subrepresentation of~$\mV$.
    However, $\mV$ is irreducible as it is a highest weight representation.
    Since $B \restriction_{\mV \times \mV}$ is non-zero (e.g., $B(e_\alpha e_\alpha,e_{-\alpha}e_{-\alpha})=1$), we conclude that the radical of $B \restriction_{\mV \times \mV}$ is trivial.
\end{proof}

The previous proposition allows us to define an orthogonal complement of $\mV$ with respect to the bilinear form $B$.
This will be the underlying representation of our algebra.

\begin{definition}%
    \label{def:A}
    \begin{enumerate}
        \item Let $\mA$ be the orthogonal complement of $\mV$ in $S^2(\mL)$ with respect to the $\mL$-equivariant bilinear form $B$:
            \[\mA \coloneqq     \{ v \in S^2(\mL) \mid B(v,w) = 0 \text{ for all $w \in \mV$}\}.\]
        \item Denote the orthogonal projection of $S^2(\mL)$ onto $\mA$ by $\pi$.  For each $v \in S^2(\mL)$, we will also denote $\pi(v)$ by $\overline{v}$.
    \end{enumerate}
\end{definition}

The character of $\mA$ follows easily from the characters of $S^2(\mL)$ and $\mV$.

\begin{proposition}%
    \label{prop:charA}
    The character of $\mA$ as a representation for $\mL$ is given by
    \[
        \ch_{\mA} = \frac{n(n+1)}{2} + \sum_{\alpha \in \Phi} (n-1)e^\alpha + \sum_{\lambda \in \Lambda_0} e^\lambda
    \]
    where $n$ is the rank of $\Phi$.
\end{proposition}
\begin{proof}
    Since $S^2(\mL) = \mA \oplus \mV$, we have $\ch_{\mA} = \ch_{S^2(\mL)} - \ch_{\mV}$. The characters of $S^2(\mL)$ and $\mV$ follow from \cref{prop:charS,prop:charV}.
\end{proof}

Using \cref{prop:gensetV} we can explicitly describe the weight spaces of $\mA$.

\begin{proposition}%
    \label{prop:Aunique}
    The weights of $\mA$ are $0$, the roots $\alpha \in \Phi = \Lambda_{-1}$ and the sums of orthogonal roots $\lambda \in \Lambda_0$. Any weight vector can be uniquely written as
    \begin{enumerate}
        \item $\overline{a}$ for $a \in S^2(\mH) \leq S^2(\mL)$ if the weight vector has weight 0;
        \item $\overline{e_\alpha h}$ for $h \in \alpha^\perp \coloneqq \{ h \in \mH \mid \kappa(\alpha,h) = 0\}$ if the weight vector has weight $\alpha \in \Phi$ (also note that $\overline{e_\alpha h_\alpha}=0$);
        \item $c\overline{e_{\alpha_\lambda}e_{\beta_\lambda}}$ for $c \in \C$ when the weight vector has weight $\lambda \in \Lambda_0$.
    \end{enumerate}
\end{proposition}
\begin{proof}
    Recall that $\mV$ is the orthogonal complement of $\mA$ in $S^2(\mL)$ with respect to the bilinear form $B$.
    The statement follows from the description of the generating set of $\mV$ from \cref{prop:gensetV,prop:f}.
\end{proof}

The projection $\pi$ from \cref{def:A} can be computed explicitly.
In fact, in what follows, we will only need formula~\eqref{eq:overlineealphaebeta}, but for completeness, we also provide formulas~\eqref{eq:A2} and~\eqref{eq:A3}.

\begin{lemma}%
    \label{lem:A}
    Let $\lambda \in \Lambda_0$. Then
    \begin{equation}\label{eq:overlineealphaebeta}
        \overline{e_{\alpha_\lambda} e_{\beta_\lambda}} = \frac{1}{n_\lambda} \left( \sum_{\alpha+\beta=\lambda} f_{\alpha,\beta} e_\alpha e_\beta \right),
    \end{equation}
    where the sum runs over all sets $\{\alpha,\beta\}$ where $\alpha,\beta \in \Phi$ such that $\alpha+\beta=\lambda$, or equivalently, over all elements $\{\alpha,\beta\} \in N_\lambda$. Also
    \begin{align}
        \overline{e_{\alpha_\lambda} h_{\beta_\lambda}} &= \frac{1}{n_{\lambda}} \left( e_{\alpha_\lambda} h_{\beta_\lambda} + \sum (c_{\beta,-\alpha_\lambda}e_{\alpha-\beta_\lambda}e_{\beta} + c_{\alpha,-\alpha_\lambda}e_{\beta-\beta_\lambda}e_{\alpha}) \right), \label{eq:A2}\\
        \overline{h_{\alpha_\lambda} h_{\beta_\lambda}} &= \frac{1}{n_{\lambda}} \left( h_{\alpha_\lambda} h_{\beta_\lambda} + \sum (e_\alpha e_{-\alpha} + e_\beta e_{-\beta} - e_{\alpha_\lambda - \alpha} e_{\alpha-\alpha_\lambda} - e_{\alpha_\lambda - \beta} e_{\beta-\alpha_\lambda}) \right), \label{eq:A3}
    \end{align}
    where each sum runs over all $\{\alpha,\beta\} \in N_\lambda$ with $\{\alpha,\beta\} \neq \{\alpha_\lambda,\beta_\lambda\}$.
\end{lemma}
\begin{proof}
    It is immediately verified that
    \[
        B\left(v,\frac{1}{n_\lambda} \biggl( \sum_{\alpha+\beta=\lambda} f_{\alpha,\beta} e_\alpha e_\beta \biggr)\right) = 0
    \]
    for all $v \in \Gamma_0 \cup \dots \cup \Gamma_5$. Since
    \[
        \frac{1}{n_\lambda} \left( \sum_{\alpha+\beta=\lambda} f_{\alpha,\beta} e_\alpha e_\beta \right) - e_{\alpha_\lambda}e_{\beta_\lambda}
        = -\frac{1}{n_\lambda} \left(\sum_{\alpha+\beta=\lambda} (e_{\alpha_\lambda}e_{\beta_\lambda} - f_{\alpha,\beta} e_\alpha e_\beta) \right) \in \mV,
    \]
    we have~\eqref{eq:overlineealphaebeta}.

    Recall the definition of $f_{\alpha,\beta}$ from \cref{def:f} and remember that $f_{\alpha,\beta}=f_{\beta,\alpha}$ by \cref{prop:f}.
    Using these, we have, since the projection $\pi$ is $\mL$-equivariant,
    \begin{align*}
        \overline{e_{\alpha_\lambda} h_{\beta_\lambda}} &= e_{-\beta_\lambda} \cdot (- \overline{e_{\alpha_\lambda} e_{\beta_\lambda}}) \\
                                                        &= \frac{1}{n_{\lambda}} \left( e_{\alpha_\lambda} h_{\beta_\lambda} - \sum (f_{\beta,\alpha} c_{-\beta_\lambda,\alpha} e_{\alpha-\beta_\lambda}e_{\beta} + f_{\alpha,\beta} c_{-\beta_\lambda,\beta} e_{\beta-\beta_\lambda} e_{\alpha}) \right) \\
                                                        &= \frac{1}{n_{\lambda}} \left( e_{\alpha_\lambda} h_{\beta_\lambda} + \sum (c_{\beta,-\alpha_\lambda}e_{\alpha-\beta_\lambda}e_{\beta} + c_{\alpha,-\alpha_\lambda}e_{\beta-\beta_\lambda}e_{\alpha}) \right)
    \end{align*}
    where each sum runs over the sets $\{\alpha,\beta\}$ with $\alpha,\beta \in \Phi$, $\alpha + \beta = \lambda$ and $\{\alpha,\beta\} \neq \{\alpha_\lambda,\beta_\lambda\}$.
    Similarly, we have
    \begin{align*}
        \overline{h_{\alpha_\lambda} h_{\beta_\lambda}} &= e_{-\alpha_\lambda} \cdot (- \overline{e_{\alpha_\lambda} h_{\beta_\lambda}}) \\
                                                        &= \frac{1}{n_\lambda} \left( h_{\alpha_\lambda} h_{\beta_\lambda} - \sum (e_\alpha e_{-\alpha} + e_{\beta} e_{-\beta} - e_{\alpha_\lambda-\alpha}e_{\alpha-\alpha_\lambda} - e_{\alpha_\lambda-\beta}e_{\beta-\alpha_\lambda}) \right)
    \end{align*}
    where we have used that \[c_{\beta,-\alpha_\lambda}c_{-\alpha_\lambda,\beta}=c_{\alpha,-\alpha_\lambda}c_{-\alpha_\lambda,\alpha}=-1,\] and \[c_{\beta,-\alpha_\lambda} c_{-\alpha_\lambda,\alpha-\beta_\lambda}=c_{\alpha,-\alpha_\lambda}c_{-\alpha_\lambda,\beta-\beta_\lambda}=1,\] from \cref{def:stage}~\labelcref{def:stage:c}.
\end{proof}

We finish this section by defining a suitable product $\ast$ and bilinear form $\mB$ for $\mA$ such that $(\mA,\ast,\mB)$ is a Frobenius algebra for $\mL$.

\begin{proposition}%
    \label{prop:defast}
    Consider the linear maps
    \begin{align*}
        \ast :\; &\mA \times \mA  \to \mA : (v,w) \mapsto \overline{v \bullet w}, \\
        \mB :\; &\mA \times \mA \to \C:(v,w) \mapsto B(v,w).
    \end{align*}
    Then $(\mA,\ast,\mB)$ is a Frobenius algebra for $\mL$.
\end{proposition}
\begin{proof}
    The maps $\ast$ and $\mB$ are $\mL$-equivariant as a composition of $\mL$-equivariant maps.
    The Frobenius property~\eqref{eq:frob} follows from \cref{prop:prodS2} and because $\mB(\overline{v},w)=B(v,w)$ if $w \in \mA$.
\end{proof}

\section{The zero weight subalgebra}%
\label{sec:0sub}

Consider the zero weight space $\mA_0$ of $\mA$ with respect to a fixed Cartan subalgebra $\mH$ of $\mL$.
Since the product $\ast$ and bilinear form $\mB$ are $\mL$-equivariant, $\mA_0$ is a Frobenius subalgebra of $\mA$.
In this section, we describe this subalgebra explicitly.
In order to keep a clear distinction with the construction of the previous chapter, we will denote vector spaces occuring in this new construction by gothic letters.
First, we will use the isomorphism $\zeta \colon S^2(\mL) \to \Hom(\mL,\mL)$ from \cref{def:prodS2} to describe the zero weight space of $S^2(\mL)$ as a space of homomorphisms.

\begin{definition}%
    \label{def:JZS}
    Recall the notation $e_\alpha$ and $h_\alpha$ for $\alpha \in \Phi$ from \cref{def:stage}.
    Then the zero weight subspace of $S^2(\mL)$ is spanned by the elements $h_\alpha h_\beta$ and $e_\alpha e_{-\alpha}$ for $\alpha,\beta \in \Phi$.
    \begin{enumerate}
        \item Let $\fJ$ be the subspace of $\Hom(\mL,\mL)$ spanned by the endomorphisms $j_\alpha \coloneqq \zeta(h_\alpha h_\alpha)$ for $\alpha \in \Phi$. Explicitly, the endomorphism $j_\alpha$ is defined by
            \[
                j_\alpha \colon \mL \to \mL \colon \ell \mapsto \kappa(\ell,h_\alpha)h_\alpha.
            \]
            Since $h_{-\alpha} = -h_{\alpha}$, we have $j_\alpha = j_{-\alpha}$.
            Note that $j_\alpha(e_\beta)=0$ and ${j_\alpha}(\mL) \subseteq \mH$ for all $\alpha,\beta \in \Phi$. Therefore we can, and will, view $\fJ$ as a subspace of $\Hom(\mH,\mH)$.
        \item Let $\fZ$ be the subspace of $\Hom(\mL,\mL)$ spanned by the endomorphisms $z_\alpha \coloneqq \zeta(e_\alpha e_{-\alpha})$ for $\alpha \in \Phi$. We have
            \[
                z_\alpha \colon \mL \to \mL \colon \ell \mapsto \frac{1}{2}\left(\kappa(\ell,e_\alpha)e_{-\alpha} + \kappa(\ell,e_{-\alpha})e_\alpha\right).
            \]
            Also $z_\alpha = z_{-\alpha}$.
        \item Define $\fS_0 \coloneqq \fJ + \fZ$. Then $\fS_0 = \fJ \oplus \fZ$ as vector spaces.
        \item Consider the Jordan product $\bullet$ and bilinear form $B$ on $\Hom(\mL,\mL)$ and on $\Hom(\mH,\mH)$ as defined in \cref{ex:frob}~\labelcref{ex:frob:jordan}.
            This turns these vector spaces into Frobenius algebras.
    \end{enumerate}
\end{definition}

We will prove that $\fS_0$ is a Frobenius subalgebra of $\Hom(\mL,\mL)$.
In fact, we have $\fS_0 = \fJ \oplus \fZ$ as Frobenius algebras.
In particular, $\fJ$ is a subalgebra of $\Hom(\mH,\mH)$.
This subalgebras has already been studied by T.~De Medts and F.~Rehren in~\cite{DR17} in a different context.

\begin{proposition}%
    \label{prop:DR}
    The subspace $\fJ$ is a Frobenius subalgebra of $\Hom(\mH,\mH)$.
    More precisely,
    \begin{align*}
        j_\alpha \bullet j_\beta &= \begin{cases}
            2j_\alpha & \text{if $\alpha = \pm \beta$}, \\
            0 & \text{if $\kappa(\alpha,\beta) = 0$}, \\
            \frac{1}{2}(j_\alpha + j_\beta - j_{{s_\beta}(\alpha)}) & \text{if $\kappa(\alpha,\beta) = \pm 1$},
        \end{cases} \\
        \intertext{and}
        B(j_\alpha,j_\beta) &= {\kappa(\alpha,\beta)}^2,
    \end{align*}
    for all $\alpha, \beta \in \Phi$. It has dimension $\tfrac{n(n+1)}{2}$ and hence consists of all endomorphisms $f : \mH \to \mH$ for which $\kappa(f(a),b)=\kappa(a,f(b))$ for all $a,b \in \mH$.
    The Frobenius algebra $\fJ$ is isomorphic to the Frobenius algebra from \cref{ex:frob}~\labelcref{ex:frob:symm} for $V = \mH$.
\end{proposition}
\begin{proof}
    The multiplication follows from~\cite{DR17}*{Lemma~3.2}.
    However, this can also be calculated using the explicit description of these homomorphisms from \cref{def:JZS}.
    The dimension follows from~\cite{DR17}*{Lemma~3.3}.
    The endomorphisms $j_\alpha$ satisfy the condition that $\kappa({j_\alpha}(a),b)=\kappa(a,{j_\alpha}(b))$ for all $a,b \in \mH$.
    Since the subspace of all such homomorphisms has dimension $n(n+1)/2$, this subspace must be equal to $\fJ$.
    So, in fact, $\fJ$ is precisely the Frobenius algebra from \cref{ex:frob}~\labelcref{ex:frob:symm} for $V=\mH$.
\end{proof}

Also $\fZ$ and $\fS_0$ are Frobenius subalgebras of $\Hom(\mL,\mL)$.

\begin{proposition}
    The subspace $\fZ$ is a Frobenius subalgebra of $\Hom(\mL,\mL)$.
    We have
    \begin{align*}
        z_\alpha \bullet z_\beta &= \begin{cases}
            \frac12 z_\alpha & \text{if $z_\alpha = z_\beta$}, \\
            0 & \text{otherwise},
        \end{cases} \\
        \intertext{and}
        B(z_\alpha,z_\beta) &= \begin{cases}
            \frac12 & \text{if $z_\alpha = z_\beta$}, \\
            0 & \text{otherwise},
        \end{cases}
    \end{align*}
    for all $\alpha,\beta \in \Phi$.
    The subspace $\fZ$ has dimension $\tfrac{\abs{\Phi}}{2}$.
\end{proposition}
\begin{proof}
    Notice that $\kappa(e_\alpha,e_{-\alpha})=1$ and $\kappa(e_\alpha,e_\beta) = 0$ for all $\beta \neq -\alpha$.
    The assertion now follows from an explicit calculation.
\end{proof}

\begin{proposition}%
    \label{prop:Z}
    The subspace $\fS_0$ is a Frobenius subalgebra of $\Hom(\mL,\mL)$.
    In fact $\fS_0 = \fJ \oplus \fZ$ as Frobenius algebras.
    This means that $\fS_0 = \fJ \oplus \fZ$ as vector spaces and
    \[
        a \bullet b = 0 \quad \text{and} \quad B(a,b) = 0
    \]
    for all $a \in \fJ$ and all $b \in \fZ$.
\end{proposition}
\begin{proof}
    Since the composition $ab$ of the $\mL$\dash{}endomorphisms $a$ and $b$ is zero, we also have $a \bullet b = \tfrac12 (ab + ba) = 0$ and $B(a,b) = \tr(ab) = 0$.
\end{proof}

Note that the zero weight space ${S^2(\mL)}_0$ of $S^2(\mL)$ is a Frobenius subalgebra of $S^2(\mL)$ because $\bullet$ and $B$ are $\mL$-equivariant.
It is isomorphic to $\fS_0$ as a Frobenius algebra.

\begin{proposition}%
    \label{prop:S}
    Let $\zeta$ be as in \cref{def:prodS2} and let ${S^2(\mL)}_0$ be the zero weight subspace of $S^2(\mL)$ with respect to $\mH$.
    Then $\zeta$ induces an isomorphism
    \[
        \zeta_S \colon {S^2(\mL)}_0 \to \fS_0 \colon a \mapsto \zeta(a)
    \]
    of Frobenius algebras.
\end{proposition}
\begin{proof}
    Of course, $\fS_0$ must be contained in the image of ${S^2(\mL)}_0$ under $\zeta$.
    However, since both have the same dimension, $\fS_0$ must actually be equal to this image.
    Notice that this implies that $\zeta(h_\alpha h_\beta) \in \fS_0$ for all $\alpha,\beta \in \Phi$.
    Now it follows immediately from the construction of the product and bilinear form on $S^2(\mL)$ (see \cref{def:prodS2}) that this is an isomorphism of Frobenius algebras.
\end{proof}

Next, we describe the zero weight space $\mA_0$ of $\mA$.
Recall that $\mA$ is defined as the orthogonal complement of the $\mL$-invariant subspace $\mV$ with respect to $B$.
Since $B$ is $\mL$-equivariant, the zero weight space $\mA_0$ of $\mA$ is the orthogonal complement of the zero weight space $\mV_0$ of $\mV$ in ${S^2(\mL)}_0$.
By \cref{prop:gensetV}, we know that the space $\mV_0$ is spanned by the elements $2e_\alpha e_{-\alpha} - h_\alpha h_\alpha$.
Therefore, we introduce the following definition.

\begin{definition}%
    \label{def:A0}
    \begin{enumerate}
        \item For each $\alpha \in \Phi$, let $v_\alpha \coloneqq \zeta(2e_\alpha e_{-\alpha} - h_\alpha h_\alpha) = 2z_\alpha - j_\alpha$ and let $\fV$ be the subspace of $\fS_0$ spanned by these $v_\alpha$.
            Then the restriction of $B$ to $\fV \times \fV$ is non-degenerate since $B(v_\alpha,v_\alpha)=6$.
            Let $\fA_0$ be the orthogonal complement of $\fV$ in $\fS_0$ with respect to $B$ and let $\pi \colon \fS_0 \to \fA_0$ be the orthogonal projection.
        \item Define the following product and bilinear form on $\fA_0$:
            \[
                a \pa b \coloneqq \pi(a \bullet b) \quad \text{and} \quad B_A(a,b) \coloneqq B(a,b)
            \]
            for all $a,b \in \fA_0$.
    \end{enumerate}
\end{definition}

\begin{proposition}%
    \label{prop:A0}
    \begin{enumerate}
        \item\label{prop:A0:frob} The triple $(\fA_0,\pa,B_A)$ is a Frobenius algebra.
        \item\label{prop:A0:iso} The isomorphism $\zeta$ from \cref{def:prodS2} induces an isomorphism
            \[
                \mA_0 \to \fA_0 \colon a \mapsto \zeta(a)
            \]
            of Frobenius algebras.
    \end{enumerate}
\end{proposition}
\begin{proof}
    This is obvious from \cref{def:A0,prop:S} since the subspace $\mV_0$ corresponds to $\langle v_\alpha \mid \alpha \in \Phi \rangle$ under $\zeta$.
\end{proof}

\begin{remark}%
    \label{rem:0subsum}
    Before we continue, let us give a summary of the notation and the obtained results on the connection between the different algebras.
    The goal of this section is to get a better understanding of the zero weight space $\mA_0$ with respect to $\mH$ of the $\mL$\dash{}module from \cref{def:A}.
    Recall that $\mA$ is the complement (with respect to $B$) of the $\mL$\dash{}module $\mV$ inside the symmetric square $S^2(\mL)$ of the adjoint module of $\mL$.
    By \cref{lem:Bweights}, the zero weight space $\mA_0$ is the complement of the zero weight space $\mV_0$ of $\mV$ inside the zero weight space ${S^2(\mL)}_0$ of ${S^2(\mL)}$.
    So our first step is to obtain a better understanding of ${S^2(\mL)}_0$.
    We can decompose ${S^2(\mL)}_0$ as
    \begin{equation*}
        \langle h_\alpha h_\beta \mid \alpha,\beta \in \Phi \rangle \oplus \langle e_{\alpha}e_{-\alpha} \mid \alpha \in \Phi \rangle.
    \end{equation*}
    The first component of this decomposition can be identified with the symmetric square $S^2(\mH)$ of the Cartan subalgebra $\mH$.
    Now we consider the monomorphism $\zeta \colon S^2(\mL) \to \Hom(\mL,\mL)$.
    Under this monomorphism the product $\bullet$ and bilinear form $B$ of $S^2(\mL)$ correspond, by definition, to the Jordan product and trace form on $\Hom(\mL,\mL)$.
    From the results above, we have that $\zeta$ induces an isomorphism between the following structures:
    \begin{align*}
        \zeta &\colon (S^2(\mH),\bullet,B) \xrightarrow{\sim} (\fJ,\bullet,B), \\
        \zeta &\colon (\langle e_\alpha e_{-\alpha} \mid \alpha \in \Phi \rangle, \bullet,B) \xrightarrow{\sim} (\fZ,\bullet,B), \\
        \zeta &\colon ({S^2(\mL)}_0,\bullet,B) \xrightarrow{\sim} (\fS_0,\bullet,B), \\
        \zeta &\colon \mV_0 \xrightarrow{\sim} \fV, \\
        \zeta &\colon (\mA_0,\ast,\mB) \xrightarrow{\sim} (\fA_0,\pa,B_A).
    \end{align*}
\end{remark}

It turns out that $\fJ$ is isomorphic to $\fA_0$ as vector spaces.
It will be convenient in the next section to identify both, as the elements of $\fJ$ can be viewed as endomorphisms of $\mH$.

\begin{proposition}%
    \label{prop:JA}
    The restriction $\pi_J$ of $\pi$ to $\fJ$ is an isomorphism of vector spaces.
\end{proposition}
\begin{proof}
    Note that for each $\alpha \in \Phi$, we have $\pi(j_\alpha) = 2\pi(z_\alpha)$.
    Thus, since $\pi$ is surjective, its restriction to $\fJ$ is surjective as well.
    Since $\fJ$ and $\fA_0$ have the same dimension, the restriction of $\pi$ to $\fJ$ must be an isomorphism onto $\fA_0$.
\end{proof}

\begin{definition}%
    \label{def:Jast}
    In the next section we will identify $\fA_0$ with $\fJ$ using the isomorphism $\pi_J$ from \cref{prop:JA}.
    In particular, we can transfer the product $\pa$ and the bilinear form $B_A$ to $\fJ$:
    \begin{align*}
        a \pa b &\coloneqq {\pi_J^{-1}}({\pi_J}(a) \pa {\pi_J}(b)), \\
        B_A(a,b) &\coloneqq B_A({\pi_J}(a),{\pi_J}(b)),
    \end{align*}
    for all $a,b \in \fJ$.
    From \cref{prop:A0} it follows that $(\fJ,\pa,B_A)$ is isomorphic to $(\mA_0,\ast,\mB)$ as Frobenius algebras.
    This will be the starting point of the next section.
\end{definition}

\begin{remark}
    In the spirit of \cref{rem:0subsum} we now have that ${\pi_J}^{-1} \circ \zeta$ induces an isomorphism between $(\mA_0,\ast,\mB)$ and $(\fJ,\pa,B_A)$.
\end{remark}

The Weyl group of $\mL$ acts naturally on the zero weight space of $S^2(\mL)$. Since this zero weight space is isomorphic to $S(\Phi)$ by \cref{prop:S}, also $S(\Phi)$ carries the structure of a representation of the Weyl group of $\mL$.

\begin{definition}%
    \label{def:Waction}
    Consider the natural action of the Weyl group of $\mL$ on the zero weight space ${S^2(\mL)}_0$ of $S^2(\mL)$.
    Due to \cref{prop:S}, we can transfer this action to $\fS_0$:
    \[
        w \cdot s \coloneqq \zeta_S^{-1}(w \cdot \zeta_S(s))
    \]
    for all $s \in \fS_0$ and $w \in W$.
    Notice that the product $\bullet$ and bilinear form $B$ are $W$-equivariant.
    Therefore $(\fS_0,\bullet,B)$ is a Frobenius algebra for $W$.
    It is readily verified that $w \cdot j_\alpha = j_{w \cdot \alpha}$, $w \cdot z_\alpha = z_{w \cdot \alpha}$ and $w \cdot v_\alpha=v_{w \cdot \alpha}$ for all $\alpha \in \Phi$ and $w \in W$.
\end{definition}

In the remainder of this section, we prove that the projection $\pi$ is $W$-equivariant.
Therefore its image, $\fA_0$, is $W$-invariant.
This fact will be used in \cref{sec:0dec} to give $\fA_0$ the structure of an axial decomposition algebra.
This also allows to compute the projection $\pi$ efficiently, as we will illustrate below.
(This is, however, not essential for the rest of our results.)

\begin{definition}
    \begin{enumerate}
        \item Consider the transitive action of $W$ on the set $X = \{ j_\alpha \mid \alpha \in \Phi^+ \}$. Let $O_0, O_1, \dots, O_d$ be the orbits of $W$ on $X \times X$, where $O_0$ is the diagonal $O_0 \coloneqq \{(x,x) \mid x \in X\}$.
            Define the following \textit{intersection parameters} for $0 \leq i,j,k \leq d$:
            \[
                p_{ij}^k \coloneqq \abs{\{y \in X \mid (x,y) \in O_i \text{ and } (y,z) \in O_j\}}
            \]
            where $(x,z)$ is any element of $O_k$.
            Note that this does not depend on the choice of $(x,z)$.
            (In fact, $(X,{\{O_i\}}_{0 \leq i \leq d})$ is an \textit{association scheme}; see, e.g.,~\cite{BI84}*{\S 2.2}.)
        \item For each $\alpha \in \Phi$, we can write $\pi(j_\alpha)$ uniquely as
            \[
                j_\alpha + \sum_{\beta \in \Phi^+} \mu_{(j_\alpha,j_\beta)} z_\beta
            \]
            for certain constants $\mu_{(j_\alpha,j_\beta)} \in \C$.
    \end{enumerate}
\end{definition}

\begin{proposition}%
    \label{prop:piW}
    The projection $\pi: \fS_0 \to \fA_0$ is $W$-equivariant.
    In particular $\mu_x = \mu_y$ for all $x,y \in O_i$, $0 \leq i \leq d$.
\end{proposition}
\begin{proof}
    Because $w \cdot v_\alpha = v_{w \cdot \alpha}$ for all $\alpha \in \Phi$ and all $w \in W$, the subspace $\fV$ of $\fS_0$ is $W$-invariant.
    Since the bilinear form $B$ is $W$-equivariant, the orthogonal complement $\fA_0$ of $\fV$ as well as the orthogonal projection $\pi \colon \fS_0 \to \fA_0$ with respect to $B$ is $W$-equivariant.
    Thus on the one hand we have
    \[
        \pi(w \cdot j_\alpha) = j_{w \cdot \alpha} + \sum_{\beta \in \Phi^+} \mu_{(j_{w \cdot \alpha},j_\beta)} v_\beta
    \]
    while on the other hand
    \[
        \pi(w \cdot j_\alpha) = w \cdot \pi(j_\alpha) = j_{w \cdot \alpha} + \sum_{\beta \in \Phi^+} \mu_{(j_{\alpha},j_{\beta})} v_{w \cdot \beta}.
    \]
    Since the elements of $\{j_{w \cdot \alpha}\} \cup \{v_\beta \mid \beta \in \Phi^+\}$ are linearly independent, we have $\mu_{(j_{\alpha},j_{\beta})}=\mu_{(j_{w \cdot \alpha},j_{w \cdot \beta})}$ for all $w \in \Phi$.
\end{proof}

\begin{definition}
    Let $1 \leq i \leq d$.
    By \cref{prop:piW} we can define $\mu_i \coloneqq \mu_x$ for any $x \in O_i$.
    Since the bilinear form $B$ is $W$-equivariant, we can also write $b_i \coloneqq B(j_\alpha,j_\beta)$ for any $(j_\alpha,j_\beta) \in O_i$.
\end{definition}

The following proposition allows us the compute the constants $\mu_i$ and hence the projection $\pi$ by solving a system of $d+1$ linear equations.

\begin{proposition}
    For all $0 \leq k \leq d$ we have
    \[
        \sum_{0\leq i,j \leq d} p_{ij}^k b_j \mu_i = b_k-2\mu_k.
    \]
    Moreover, these equations uniquely determine the constants $\mu_i$.
\end{proposition}
\begin{proof}
    We have that
    \[
        \pi(j_\alpha) = j_\alpha + \sum_{0 \leq i \leq d} \sum_{\substack{\beta \in \Phi^+ \\ (j_\alpha,j_\beta) \in O_i}} \mu_i v_\beta
    \]
    if and only if
    \[
        B(\pi(j_\alpha),v_\gamma) = B(j_\alpha,v_\gamma) + \sum_{0 \leq i \leq d} \sum_{\substack{\beta \in \Phi^+ \\ (j_\alpha,j_\beta) \in O_i}} \mu_i B(v_\beta,v_\gamma) = 0
    \]
    for all $\gamma \in \Phi^+$.
    If $(j_\alpha,j_\gamma) \in O_k$ then we have, by \cref{prop:Z,prop:S},
    \begin{align*}
        B(\pi(j_\alpha),v_\gamma) &= -b_k + \sum_{0 \leq i,j \leq d} \sum_{\substack{(j_\alpha,j_\beta) \in O_i \\ (j_\beta,j_\gamma) \in O_j}} \mu_i B(j_\beta,j_\gamma) + 4 \mu_k B(z_\gamma,z_\gamma), \\
        &= -b_k + \sum_{0 \leq i,j \leq d} p_{ij}^k \mu_i b_j + 2 \mu_k.
    \end{align*}
    This proves the statement.
\end{proof}

\section{Extending the product}~\label{sec:extending}

The goal of this section is to explicitly describe the algebra $\mA$ from \cref{sec:construction}.
More precisely, we will write the product of any two elements of $\mA$ in terms of the product on the zero weight subalgebra $\mA_0$ studied in \cref{sec:0sub}.
It suffices to describe the product of any two weight vectors of $\mA$.
These weight vectors are described in \cref{prop:Aunique}.

We will use the action of the Lie algebra $\mL$ on $\mA$ to accomplish this goal.
Therefore it will be essential to get a good description of this action.
Since $\mL$ is generated by the elements $e_\alpha$ for $\alpha \in \Phi$, it suffices to describe the action of $e_\alpha$ on each of the weight-$\lambda$-spaces.
This action will of course depend on the $W$-orbit of $(\alpha,\lambda)$.
Inevitably, we need to distinguish between each of those orbits which makes the following proposition look daunting at first sight.
However, in each of the cases, the action is very natural.

\begin{proposition}%
    \label{prop:actionL}
    Let $\alpha \in \Phi$. Recall \cref{def:stage} and the linear homomorphism $\zeta$ from \cref{def:prodS2}.
    The linear action of $e_\alpha$ on $\mA$ is uniquely determined as follows.
    \begin{align*}
        e_\alpha \cdot \overline{h_1h_2} &= -\overline{e_\alpha(\kappa(\alpha,h_1)h_2 + \kappa(\alpha,h_2)h_1)}, \\
                                         &= -2\overline{e_\alpha({\zeta(h_1h_2)(\alpha)})} \\
        e_\alpha \cdot \overline{e_{\beta} h} &= \overline{h_\alpha h} &\text{if $\beta=-\alpha$}, \\
        e_\alpha \cdot \overline{e_\beta h} &= c_{\alpha,\beta}\overline{e_{\alpha+\beta}(h+\kappa(h,\alpha)h_\beta)} &\text{if $\kappa(\alpha,\beta)=-1$}, \\
        e_\alpha \cdot \overline{e_\beta h} &= -\kappa(\alpha,h)\overline{e_\alpha e_\beta} &\text{if $\kappa(\alpha,\beta)=0$}, \\
        e_\alpha \cdot \overline{e_\beta h} &= 0 &\text{if $\kappa(\alpha,\beta) \geq 1$}, \\
        e_\alpha \cdot \overline{e_{\alpha_\lambda}e_{\beta_\lambda}} &= f_{\lambda+\alpha,-\alpha} \overline{e_{\lambda+\alpha}h_\alpha} & \text{if $\kappa(\alpha,\lambda) = -2$, i.e.\ $\lambda + \alpha \in \Phi$}, \\
        e_\alpha \cdot \overline{e_{\alpha_\lambda}e_{\beta_\lambda}} &= c_{\alpha,\alpha_\lambda}\overline{e_{\alpha+\alpha_\lambda}e_{\beta_\lambda}} & \text{if $\kappa(\alpha,\alpha_\lambda) = -1$ and $\kappa(\alpha,\beta_\lambda)=0$}, \\
        e_\alpha \cdot \overline{e_{\alpha_\lambda}e_{\beta_\lambda}} &= c_{\alpha,\beta_\lambda}\overline{e_{\alpha_\lambda}e_{\alpha+\beta_\lambda}} & \text{if $\kappa(\alpha,\alpha_\lambda) = 0$ and $\kappa(\alpha,\beta_\lambda)=-1$}, \\
        e_\alpha \cdot \overline{e_{\alpha_\lambda}e_{\beta_\lambda}} &= 0 &\text{if $\kappa(\alpha,\lambda) \geq 0$},
    \end{align*}
    where $h_1,h_2 \in \mH$, $h \in \beta^\perp$, $\alpha,\beta \in \Phi$ and $\lambda \in \Lambda_0$.
\end{proposition}
\begin{proof}
    First of all, note that this enumeration exhausts all possible weight vectors of $\mA$.
    Indeed, because the root system $\Phi$ is simply laced, we have for any root $\beta \in \Phi$ that $\kappa(\alpha,\beta) \in \{-2,-1,0,1,2\}$ and for any weight $\lambda \in \Lambda_0$ that $\kappa(\alpha,\lambda) \in \{-2,-1,0,1,2\}$.
    Moreover, if $\kappa(\alpha,\beta)=-2$ then $\alpha=-\beta$ and if $\kappa(\alpha,\lambda)=-2$, then $\lambda+\alpha\in \Phi$ by \cref{lem:weights}~\ref{lem:weights:lambdaalpha}.
    The form of these weight vectors follows from \cref{prop:Aunique}.

    The statements follow from explicit calculations using the rules from \cref{def:stage} and the description of the generating set for $\mV$ from \cref{prop:gensetV}.
    We will do these calculations for the case when $\alpha,\beta \in \Phi$ such that $\kappa(\alpha,\beta) = -1$.
    The other cases are proven analogously.
    Let $h \in \mH$.
    By \cref{prop:gensetV} it follows that $\overline{e_\beta h_\beta} = 0$.
    Thus $\overline{e_\beta h} = \overline{e_\beta(h + \kappa(h,\alpha)h_\beta)}$.
    Now we have
    \begin{align*}
        e_\alpha \cdot \overline{e_\beta h} &= e_\alpha \cdot \overline{e_\beta(h + \kappa(h,\alpha)h_\beta)}, \\
                                            &= \overline{[e_\alpha,e_\beta](h + \kappa(h,\alpha)h_\beta)} + \overline{e_\beta[e_\alpha,h + \kappa(h,\alpha)h_\beta]}, \\
                                            &= c_{\alpha,\beta} \overline{e_{\alpha+\beta}(h + \kappa(h,\alpha)h_\beta)} + 0,
    \end{align*}
    because the projection $S^2(\mL) \to \mA \colon v \mapsto \overline{v}$ is $\mL$-equivariant.
\end{proof}

The next step is to write the product of any zero weight vector and an arbitrary weight vector in terms of products between zero weight vectors.
The following lemma will be crucial.

\begin{lemma}%
    \label{lem:Aaut}
    Let $\tau$ be an automorphism of the Lie algebra $\mL$. Then $\tau$ induces an automorphism of the Frobenius algebra $(\mA,\ast,\mB)$ via
    \[
        \tau\bigl(\overline{\ell_1\ell_2}\bigr) \coloneqq \overline{\tau(\ell_1) \tau(\ell_2)},
    \]
    for all $\ell_1,\ell_2 \in \mL$.
\end{lemma}
\begin{proof}
    Since $\ast$ and $\mB$ are $\mL$-equivariant, this is of course true if $\tau$ is an inner automorphism of the Lie algebra $\mL$. By~\cite{Hum72}*{\S~16.5} we can assume that $\tau$ leaves the Cartan subalgebra $\mH$ and a fixed Borel subalgebra containing $\mH$ invariant, in other words, that it is a \emph{graph automorphism}.
    Thus $\tau$ acts on the set of highest weights of the irreducible subrepresentations of $S^2(\mL)$. Since $\mV$ is the only subrepresentation of $S^2(\mL)$ having the double of a root as its highest weight, $\tau$ stabilizes the subrepresentation $\mV$ globally. Thus $\tau$ commutes with the projection $\pi \colon S^2(\mL) \to \mA$ from \cref{def:A}. Therefore, by the definition of $\ast$ and~$\mB$ (see \cref{prop:defast}), $\tau$~must preserve $\ast$ and $\mB$.
\end{proof}

\begin{remark}
    Note that any automorphism of the root system $\Phi$, i.e., any isometry $\tau$ of $\mH$ such that $\tau(\Phi)=\Phi$, extends to an automorphism of the Lie algebra $\mL$ via the isomorphism theorem~\cite{Hum72}*{\S~14.2}.
\end{remark}

The $\mL$\dash{}module contains three different types of weights: the zero weight, the roots $\alpha \in \Phi$ and the sums of two orthogonal roots $\lambda \in \Lambda_0$.
In \cref{sec:0sub} we described the product of two vectors of weight zero.
We determine the product of a zero weight vector and a vector of weight $\alpha \in \Phi$ in \cref{prop:prod0a}.
The computation of the product of a zero weight vector and a vector of weight $\lambda \in \Lambda_0$ is the subject of \cref{prop:prod0l}.

\begin{proposition}%
    \label{prop:prod0a}
    Let $v,w \in \mA$ be weight vectors of respective weights 0 and $\alpha \in \Phi$. Then
    \begin{enumerate}
        \item\label{prop:prod0a:inv} $e_\alpha \cdot e_{-\alpha} \cdot w = 2w$,
        \item\label{prop:prod0a:ast} $(e_\alpha \cdot v) \ast (e_{-\alpha} \cdot w) = 0$ and therefore
            \[
                2 v \ast w = e_\alpha \cdot ( v \ast ( e_{-\alpha} \cdot w ) ),
            \]
        \item\label{prop:prod0a:B} $\mB(v,w)=0$.
    \end{enumerate}
\end{proposition}
\begin{proof}
    \begin{enumerate}
        \item By \cref{prop:Aunique}, we can write $w = \overline{e_\alpha h}$ for some $h \in \mH$ with $\kappa(\alpha,h)=0$. By \cref{prop:actionL} we have
            \begin{align*}
                e_\alpha \cdot e_{-\alpha} \cdot w &= e_\alpha \cdot e_{-\alpha} \cdot \overline{e_\alpha h}, \\
                                                   &= e_\alpha \cdot -\overline{h_\alpha h}, \\
                                                   &= 2\overline{e_\alpha h}, \\
                                                   &= 2w.
            \end{align*}
        \item Recall that $s_\alpha$ is the reflection about the hyperplane orthogonal to $\alpha$.
            Now $-s_\alpha \colon \mH \to \mH \colon h \mapsto -h^{s_\alpha}$ is an automorphism of the root system $\Phi$.
            By the isomorphism theorem~\cite{Hum72}*{\S~14.2} there exists an extension $\tau: \mL \to \mL$ of $-s_\alpha$ which is an automorphism of the Lie algebra $\mL$ and such that $\tau(e_\alpha)=e_\alpha$.
            By \cref{lem:Aaut}, the automorphism $\tau$ induces an automorphism of the Frobenius algebra $(\mA,\ast,\mB)$.
            
            Due to \cref{prop:Aunique} we can write any weight vector $x \in \mA$ of weight $\alpha$ as $x=\overline{e_\alpha h'}$ for some $h' \in \mH$ with $\kappa(\alpha,h') = 0$.
            Thus we have $\tau(x) = \overline{\tau(e_\alpha) \tau(h')} = - \overline{e_\alpha h'} = -x$ for any weight vector $x\in \mA$ of weight $\alpha$.
            Now $e_\alpha \cdot v$ is a weight vector with weight $\alpha$ and thus $\tau(e_\alpha \cdot v )= - e_\alpha \cdot v$.
            As we illustrated in part~\ref{prop:prod0a:inv} we can write $e_{-\alpha} \cdot w$ as $-\overline{h_\alpha h}$ for some $h \in \mH$ with $\kappa(\alpha,h)=0$. Thus
            \[
                \tau(e_{-\alpha} \cdot w) = \tau\bigl(\overline{h_\alpha h}\bigr) = - \overline{h_\alpha h} = -e_{-\alpha} \cdot w.
            \]
            Because $\ast$ is $\mL$-equivariant, the product $(e_\alpha \cdot v) \ast (e_{-\alpha} \cdot w)$ is a weight vector of weight $\alpha$. 
            As a result
            \[
                \tau((e_\alpha \cdot v) \ast (e_{-\alpha} \cdot w)) = -(e_\alpha \cdot v) \ast (e_{-\alpha} \cdot w).
            \]
            On the other hand, because $\tau$ is an automorphism of $(\mA,\ast,\mB)$, we have
            \begin{align*}
                \tau((e_\alpha \cdot v) \ast (e_{-\alpha} \cdot w)) &= \tau(e_\alpha \cdot v) \ast \tau(e_{-\alpha} \cdot w), \\
                                                                    &= (e_\alpha \cdot v) \ast (e_{-\alpha} \cdot w).
            \end{align*}
            We conclude that $(e_\alpha \cdot v) \ast (e_{-\alpha} \cdot w) = 0$.
            It follows that
            \begin{align*}
                2 v \ast w &= v \ast (e_\alpha \cdot e_{-\alpha} \cdot w ) \\
                           &= e_\alpha \cdot (v \ast (e_{-\alpha} \cdot w )) - (e_\alpha \cdot v) \ast (e_{-\alpha} \cdot w), \\
                            &= e_\alpha \cdot (v \ast (e_{-\alpha} \cdot w )),
            \end{align*}
            by~\labelcref{prop:prod0a:inv} and the $\mL$-equivariance of $\ast$.
        \item This follows from the fact that $\mB$ is $\mL$-equivariant and \cref{lem:Bweights}. \qedhere
    \end{enumerate}
\end{proof}

In a similar fashion, we will now express the product of a zero weight vector and a vector of weight $\lambda \in \Lambda_0$ in terms of products of zero weight vectors.

\begin{proposition}%
    \label{prop:prod0l}
    Let $v,w \in \mA$ be weight vectors of respective weights 0 and $\lambda \in \Lambda_0$.
    Recall from \cref{def:charV2} that $n_\lambda$ is the number of ways to write $\lambda$ as the sum of two (orthogonal) roots.
    Write
    \[
        \epsilon_\lambda \coloneqq \frac{1}{2n_\lambda} \sum_{\substack{\alpha \in \Phi \\ \kappa(\alpha,\lambda)=2}} \overline{e_\alpha e_{-\alpha}}.
    \]
    Then
    \begin{enumerate}
        \item\label{prop:prod0l:inv} $e_{\alpha_\lambda} \cdot e_{-\alpha_\lambda} \cdot w = 2 w$,
        \item $ v \ast w = \mB(v,\epsilon_\lambda)w$,
        \item $ \mB(v,w) = 0$.
    \end{enumerate}
\end{proposition}
\begin{proof}
    \begin{enumerate}
        \item By \cref{prop:Aunique} we can write $w = c\overline{e_{\alpha_\lambda}e_{\beta_\lambda}}$ for some $c \in \C$.
            As a result of \cref{prop:actionL} we have
            \begin{align*}
                e_{\alpha_\lambda} \cdot e_{-\alpha_\lambda} \cdot w &= c e_{\alpha_\lambda} \cdot e_{-\alpha_\lambda} \cdot \overline{e_{\alpha_\lambda}e_{\beta_\lambda}}, \\
                                                                     &= -c e_{\alpha_\lambda} \cdot \overline{h_{\alpha_\lambda}e_{\beta_\lambda}}, \\
                                                                     &= 2c\overline{e_{\alpha_\lambda}e_{\beta_\lambda}}, \\
                                                                     &= 2w.
            \end{align*}
        \item Note that by \cref{prop:Aunique} it suffices to prove this for $w = \overline{e_{\alpha_\lambda}e_{\beta_\lambda}}$.
            Since $\ast$ is $\mL$-equivariant, the product $v \ast w$ is a weight vector of weight $\lambda$.
            Because the weight-$\lambda$-space of $\mA$ is only 1-dimensional, $v \ast w$ must be a scalar multiple of $w$.
            If $a \in \mA$ such that $\mB(w,a) \neq 0$, then this scalar multiple must be
            \[
                \frac{\mB(v \ast w,a)}{\mB(w,a)}.
            \]
            We claim that we can take $a = \overline{e_{\alpha_{-\lambda}} e_{\beta_{-\lambda}}}$.
            Recall the definition of $\mB$ from \cref{prop:defast,def:prodS2}.
            We have
            \begin{align*}
                \mB(w,a) &= \mB(\overline{e_{\alpha_\lambda}e_{\beta_\lambda}},\overline{e_{\alpha_{-\lambda}} e_{\beta_{-\lambda}}}) \\
                         &= \frac{1}{n_\lambda^2}\mB\left(\sum_{\alpha+\beta=\lambda} f_{\alpha,\beta}e_\alpha e_\beta,\sum_{\alpha+\beta=-\lambda} f_{\alpha,\beta}
                             e_\alpha e_\beta\right) \\
                         &= \frac{1}{n_\lambda^2}\sum_{\alpha+\beta=\lambda} f_{\alpha,\beta}f_{-\alpha,-\beta} \mB(e_\alpha e_\beta, e_{-\alpha} e_{-\beta}) \\
                         &= \frac{f_{-\alpha_\lambda,-\beta_\lambda}}{2n_\lambda^2} \sum_{\alpha+\beta=\lambda} f_{\alpha,\beta}^2 \\
                         &= \frac{f_{-\alpha_\lambda,-\beta_\lambda}}{2n_\lambda^2} \sum_{\alpha+\beta=\lambda} 1 \\
                         &= \frac{f_{-\alpha_\lambda,-\beta_\lambda}}{2n_\lambda}
            \end{align*}
            by \cref{lem:A}, \cref{prop:f}~\ref{prop:f:-l} and because $n_\lambda=n_{-\lambda}$.

            The triple $(\mA,\ast,\mB)$ is a Frobenius algebra, so we have $\mB(v \ast w,a)=\mB(v,w \ast a)$. We compute $w \ast a$ explicitly using the definition of $\ast$ and $\bullet$ from \cref{prop:defast} and \cref{def:prodS2}.
            We have
            \begin{align*}
                w \ast a &= \overline{e_{\alpha_\lambda}e_{\beta_\lambda}} \ast \overline{e_{\alpha_{-\lambda},\beta_{-\lambda}}}, \\
                         &= \frac{1}{n_\lambda^2} \overline{\left(\sum_{\alpha+\beta=\lambda} f_{\alpha,\beta}e_\alpha e_\beta \right) \bullet \left(\sum_{\alpha+\beta=-\lambda} f_{\alpha,\beta}e_\alpha e_\beta \right)}, \\
                         &= \frac{1}{n_\lambda^2} \sum_{\alpha+\beta=\lambda} f_{\alpha,\beta}f_{-\alpha,-\beta} \overline{(e_\alpha e_\beta \bullet e_{-\alpha} e_{-\beta})}, \\
                         &= \frac{f_{-\alpha_\lambda,-\beta_\lambda}}{4n_\lambda^2} \sum_{\alpha+\beta=\lambda} (\overline{e_\alpha e_{-\alpha}} + \overline{e_\beta e_{-\beta}}),
            \end{align*}
            once again by \cref{lem:A}, \cref{prop:f}~\ref{prop:f:-l} and because $n_\lambda=n_{-\lambda}$.
            By \cref{lem:weights}~\labelcref{lem:weights:lambdaalpha} we know that $w \ast a = \frac{f_{-\alpha_\lambda,-\beta_\lambda}}{n_\lambda} \epsilon_\lambda$.
            As a result, we have
            \begin{align*}
                v \ast w &= \frac{\mB(v, w \ast a)}{\mB(w,a)} w, \\
                         &= \mB(v,\epsilon_\lambda) w.
            \end{align*}
        \item This follows from \cref{lem:Bweights}. \qedhere
    \end{enumerate}
\end{proof}

We are now ready to ``build'' the Frobenius algebra $(\mA,\ast,\mB)$.
As a first step, we describe its underlying vector space.

\begin{definition}%
    \label{def:Avec}
    Let $\mL$, $\mH$ and $\Phi$ be as in \cref{def:stage}.
    Let $\fA$ be the direct sum of the following spaces:
        \begin{itemize}
            \item the space $\fJ$ from \cref{def:JZS};
            \item for each $\alpha \in \Phi$, a copy $\fH_\alpha$ of the subspace $\alpha^\perp \coloneqq \{h\in \mH \mid \kappa(h,\alpha)=0\}$ of~$\mH$;
            \item a vector space with basis $\{x_\lambda \mid \lambda \in \Lambda_0\}$ indexed by the set $\Lambda_0$.
        \end{itemize}
    For each $h \in \mH$, we will denote its orthogonal projection onto $\fH_\alpha$ by ${[h]}_\alpha$.
\end{definition}

\begin{proposition}%
    \label{prop:Avec}
    Let $\mA$ be as in \cref{def:A}.
    For each $\lambda \in \Lambda_0$, choose roots $\alpha_\lambda,\beta_\lambda \in \Phi$ such that $\alpha_\lambda + \beta_\lambda = \lambda$.
    Let $\theta$ be the linear map defined by
    \[
        \theta : \fA \to \mA : \begin{cases}
            j_\alpha \mapsto \overline{h_\alpha h_\alpha} \\
            {[h]}_\alpha \mapsto \overline{e_\alpha h} \\
            x_\lambda \mapsto \overline{e_{\alpha_\lambda}e_{\beta_\lambda}}
        \end{cases}
    \]
    for all $\alpha \in \Phi$, $h \in \mH$ and $\lambda \in \Lambda_0$. Then $\theta$ is an isomorphism of vector spaces.
\end{proposition}
\begin{proof}
    The restriction of $\theta$ to $\fJ$ is the composition of the isomorphism from \cref{prop:JA} and the inverse of the isomorphism from \cref{prop:A0}~\ref{prop:A0:iso}. Since also $\overline{e_\alpha h_\alpha} = 0$ by \cref{prop:Aunique}, the linear map $\theta$ is well-defined. By \cref{prop:Aunique} it follows that $\theta$ is an isomorphism.
\end{proof}

Now we translate the action of $\mL$ on $\mA$ to $\fA$ using this isomorphism.
We also define bilinear maps $\fJ \times \fA \to \fA$ based on \cref{prop:prod0a,prop:prod0l}.

\begin{definition}%
    \label{def:APhiprod}
    \begin{enumerate}
        \item Transfer the action of $\mL$ on $\mA$ to $\fA$ via the isomorphism $\theta$ from \cref{prop:Avec}:
            \[
                \ell \cdot v \coloneqq \theta^{-1}(\ell \cdot \theta(v)).
            \]
            Note that is is possible to write this action down explicitly using \cref{prop:actionL}.
        \item For $\lambda \in \Lambda_0$, let $\mathfrak{e}_\lambda \coloneqq \theta^{-1}(\epsilon_\lambda)$, i.e.\,
            \[
                \mathfrak{e}_\lambda \coloneqq \frac{1}{4n_\lambda} \sum_{\substack{\alpha \in \Phi \\ \kappa(\alpha,\lambda)=2}} j_\alpha.
            \]
        \item Consider the Frobenius algebra $(\fJ,\pa,B_A)$ from \cref{def:Jast}.
            Note that for $h \in \alpha^\perp$ we have $e_{-\alpha} \cdot {[h]}_\alpha = -\theta^{-1}\bigl(\overline{h_\alpha h}\bigr) \in \fJ$ and also $\mathfrak{e}_\lambda \in \fJ$.
            Now define bilinear maps
            \begin{align*}
                \ast &\colon \fJ \times \fA \to \fA, \\
                \mB &\colon  \fJ \times \fA \to \C,
            \end{align*}
            such that $\theta(v \ast w) = \theta(v) \ast \theta(w)$ for all $v \in \fJ$ and $w \in \fA$.
            More precisely, by \cref{def:Jast,prop:prod0a,prop:prod0l} let
            \begin{align*}
                v \ast w &\coloneqq v \pa w &\text{if $w \in \fJ$}, \\
                v \ast w &\coloneqq e_\alpha \cdot (v \pa (e_{-\alpha} \cdot w)) &\text{if $w = {[h]}_\alpha$ for some $h \in \alpha^\perp$ and $\alpha \in \Phi$}, \\
                v \ast w &\coloneqq B_A(v,\mathfrak{e}_\lambda)w &\text{if $w = cx_\lambda$ for some $c \in \C$ and $\lambda \in \Lambda_0$}.
            \end{align*}
    \end{enumerate}
\end{definition}

We prove that we can uniquely extend the maps $\ast$ and $\mB$ to $\fA \times \fA$.

\begin{theorem}%
    \label{thm:extending}
    Let $\mL$ be a simple complex Lie algebra with root system $\Phi$ of type $A_n$ ($n \geq 3$), $D_n$ ($n \geq 4$) or $E_n$ ($n \in \{6,7,8\}$).
    Let $\fA$ be as in \cref{def:Avec} equipped with the $\mL$-action from \cref{def:APhiprod}.
    The maps $\ast$ and $\mB$ from \cref{def:APhiprod} uniquely extend to $\fA \times \fA$ such that $(\fA,\ast,\mB)$ is a Frobenius algebra for $\mL$.
    Moreover the isomorphism $\theta$ from \cref{prop:Avec} induces an isomorphism of Frobenius algebras for $\mL$ with the Frobenius algebra $(\mA,\ast,\mB)$ from \cref{prop:defast}.
\end{theorem}
\begin{proof}
    The extensions of $\ast$ and $\mB$ must be $\mL$-equivariant. Let $\alpha \in \Phi$, $h \in \alpha^\perp$ and $v \in \fA$.
    By definition of the action of $\mL$ on $\fA$ and \cref{prop:prod0a}~\ref{prop:prod0a:inv}, we have $e_\alpha \cdot e_{-\alpha} \cdot {[h]}_\alpha = {[2h]}_\alpha$.
    Because $\ast$ and $\mB$ must be $\mL$-equivariant, we have
    \[
        {[h]}_\alpha \ast w = \frac12 e_\alpha \cdot ((e_{-\alpha} \cdot {[h]}_\alpha) \ast w) - \frac12 (e_{-\alpha} \cdot {[h]}_\alpha) \ast (e_\alpha \cdot w)
    \]
    and
    \[
        \mB({[h]}_\alpha,w) = \frac12 \mB(e_{-\alpha} \cdot {[h]}_\alpha,e_\alpha \cdot w).
    \]
    Since $e_{-\alpha} \cdot {[h]}_\alpha \in \fJ$, this uniquely extends $\ast$ and $\mB$ to $\left(\fJ \oplus \bigoplus_{\alpha \in \Phi} \fH_\alpha \right) \times \fA$ with $\fH_\alpha$ as in \cref{def:Avec}.
    Analogously, we can extend $\ast$ and $\mB$ to $\fA \times \fA$ by using \cref{prop:prod0l}~\ref{prop:prod0l:inv} which implies that $e_{\alpha_\lambda} \cdot e_{-\alpha_\lambda} \cdot x_\lambda = 2x_\lambda$.

    Now consider the Frobenius algebra $(\mA,\ast,\mB)$ from \cref{prop:defast} and the isomorphism $\theta\colon \fA \to \mA$ from \cref{prop:Avec}.
    Then also the bilinear maps
    \begin{align*}
        &\fA \times \fA \to \fA \colon (v,w) \mapsto \theta^{-1}(\theta(v) \ast \theta(w)) \\
        &\fA \times \fA \to \C \colon (v,w) \mapsto \mB(\theta(v),\theta(w))
    \end{align*}
    are $\mL$-equivariant extensions of $\ast$ and $\mB$.
    Therefore $(\fA,\ast,\mB)$ must be a Frobenius algebra isomorphic via $\theta$ with $(\mA,\ast,\mB)$.
\end{proof}

\begin{remark}
    \begin{enumerate}
        \item The proof of \cref{thm:extending} is constructive.
            It allows us to define the product recursively starting from the Frobenius algebra $(\fJ,\pa,B_A)$ and some structure constants, namely the constants $c_{\alpha,\beta}$, of the Lie algebra.
            This is much more efficient than the construction of \cref{sec:construction} where we start with the symmetric square of the Lie algebra.
        \item From this explicit construction, it follows that we can pick a basis for $\fA$ in such a way that the structure constants for the algebra $(\fA,\ast,B)$ are integers.
            This allows to define this algebra over an arbitrary field by extension of scalars.
    \end{enumerate}
\end{remark}

\section{The story of the unit}%
\label{sec:unit}

Let $(\mA,\ast,\mB)$ be the Frobenius algebra from \cref{prop:defast}.
We prove that this algebra is \textit{unital}, which means that there exists an element $\one \in \mA$, called a \textit{unit}, such that $\one \ast a = a \ast \one = a$ for all $a \in A(\Phi)$.

From \cref{def:prodS2}, we know that $S^2(\mL)$ corresponds to the symmetric operators $\mL \to \mL$ via $\zeta$.
From the definition of the product $\bullet$ it will be immediately obvious that the identity operator $\id \colon \mL \to \mL \colon \ell \mapsto \ell$ corresponds to a unit for the algebra $(S^2(\mL),\bullet)$.

\begin{definition}%
    \label{def:unitS}
    Let $\zeta$ be as in \cref{def:prodS2}.
    Since the identity operator $\id$ is a symmetric operator, it is contained in the image of $\zeta$ and we can define
    \[
        C_\mL \coloneqq \zeta^{-1}(\id).
    \]
    More explicitly, let $\{b_1,b_2,\dots,b_m\}$ be a basis of $\mL$ and let $\{b_1^\ast,b_2^\ast,\dots,b_m^\ast\}$ be the basis of $\mL$ dual to this basis with respect to the Killing form $\kappa$.
    Then
    \[
        C_\mL \coloneqq \sum_{1 \leq i \leq m} b_i b_i^\ast.
    \]
    Note that $C_\mL$ is the Casimir element of $\mL$~\cite{Hum72}*{\S~6.2}.
    Observe that
    \begin{align*}
        B(C_\mL, \ell_1 \ell_2) &= \frac{1}{2} \sum_{1 \leq i \leq m} \kappa(b_i,\ell_1)\kappa(b_i^*,\ell_2) + \frac{1}{2} \sum_{1 \leq i \leq m} \kappa(b_i,\ell_2)\kappa(b_i^*,\ell_1) \\
        &= \kappa(\ell_1,\ell_2).
    \end{align*}
    for all $\ell_1,\ell_2 \in \mL$ and $B$ as in \cref{def:prodS2}.
    Since $B$ is non-degenerate by \cref{prop:nondeg}, this also uniquely defines $C_\mL$.
\end{definition}

\begin{proposition}%
    \label{prop:unitS}
    For all $a \in S^2(\mL)$ we have $C_\mL \bullet a = a$.
\end{proposition}
\begin{proof}
    For any $\ell_1, \ell_2 \in \mL$, we have
        \begin{align*}
            \begin{split}
                C_\mL \bullet \ell_1 \ell_2 &= \frac14 \left( \sum_i \kappa(b_i,\ell_1)b_i^*\ell_2 + \sum_i \kappa(b_i^*,\ell_1)b_i \ell_2 \right. \\& \left.\qquad + \sum_i \kappa(b_i,\ell_2) b_i^*\ell_1 + \sum_i \kappa(b_i^*,\ell_2)b_i\ell_1 \right),
            \end{split} \\
            &= \ell_1 \ell_2. \qedhere
        \end{align*}
\end{proof}

Next, we proof that $C_\mL \in \mA$ and that $C_\mL$ is also a unit for $(\mA,\ast)$.

\begin{proposition}%
    \label{prop:unitA}
    \begin{enumerate}
        \item\label{prop:unitA:trivial} We have $\ell \cdot C_\mL = 0$ for all $\ell \in \mL$.
        \item\label{prop:unitA:unitA} We have $C_\mL \in \mA$ and $C_\mL \ast a = a$ for all $a \in \mA$.
    \end{enumerate}
\end{proposition}
\begin{proof}
    \begin{enumerate}
        \item For any $\ell,\ell_1,\ell_2 \in \mL$ we have
            \begin{align*}
                B(\ell \cdot C_\mL,\ell_1\ell_2) &= B(C_\mL, \ell \cdot \ell_1\ell_2 ) \\
                                                 &= B(C_\mL, [\ell,\ell_1]\ell_2+\ell_1[\ell,\ell_2]) \\
                                                 &= \kappa([\ell,\ell_1],\ell_2)+\kappa(\ell_1,[\ell,\ell_2]) \\
                                                 &= 0
            \end{align*}
            because $B$ and $\kappa$ are $\mL$-equivariant.
            By \cref{prop:nondeg}, the bilinear form $B$ is non-degenerate and thus $\ell \cdot C_\mL = 0$ for all $\ell \in \mL$.
        \item For any $\alpha \in \Phi$, we have $B(C_\mL,e_\alpha e_\alpha) = \kappa(e_\alpha,e_\alpha) = 0$.
            Recall that $\mV$ is the $\mL$-representation generated by the elements $e_\alpha e_\alpha$.
            By~\ref{prop:unitA:trivial}, we have $B(C_\mL,v)=0$ for all $v \in \mV$ and thus $C_\mL \in \mA$.

            It follows from the definition of $\ast$ (see \cref{prop:defast}) and \cref{prop:unitS} that $C_\mL \ast a = a$ for all $a \in \mA$. \qedhere
    \end{enumerate}
\end{proof}

We transfer the unit $C_\mL$ from the algebra $(\mA,\ast)$ to the algebra $(\fA,\ast)$ via the isomorphism $\theta$ from \cref{thm:extending}.

\begin{definition}%
    \label{def:one}
    Recall $\fA$ from \cref{def:Avec} and the isomorphism $\theta$ from \cref{thm:extending}.
    Write
    \[
        \one = \theta^{-1}(C_\mL).
    \]
\end{definition}

\begin{theorem}
    Let $(\fA,\ast,\mB)$ be as in \cref{thm:extending} and $\one$ as in \cref{def:one}.
    Then $\one$ is a unit for the algebra $(\fA,\ast)$.
\end{theorem}
\begin{proof}
    This follows immediately from \cref{def:one,prop:unitA,thm:extending}.
\end{proof}

Note that by \cref{prop:unitA}~\ref{prop:unitA:trivial} the element $\one$ is contained in the zero weight space of $\fA$, this is $\fJ$.
We can write down $\one$ explicitly as a linear combination of the generating set $\{j_\alpha \mid \alpha \in \Phi^+\}$ for $\fJ$.
First, we prove the following lemma.

\begin{lemma}%
    \label{lem:sumj}
    Let $\beta \in \Phi$ and let $r$ be the number of positive roots $\alpha \in \Phi^+$ such that $\kappa(\beta,\alpha)= \pm 1$. Then
    \[
        \sum_{\alpha \in \Phi^+} j_\alpha = \frac{4+r}{2}\id_\mH
    \]
    where $\id_\mH \colon \mH \to \mH \colon h \mapsto h$.
\end{lemma}
\begin{proof}
    Since $\Phi$ is irreducible and simply laced, the value of $r$ is independent of the choice of $\beta$.
    It follows from \cref{prop:DR} that
    \[
        B \Bigl( \sum_{\alpha \in \Phi^+} j_\alpha , j_\beta \Bigr)=  4 + r = B \Bigl( \frac{4+r}{2}\id_\mH , j_\beta \Bigr)
    \]
    for all $\beta \in \Phi^+$.
    Because $B$ is non-degenerate on $\fJ$ by \cref{prop:DR}, we have indeed $\sum_{\alpha \in \Phi^+} j_\alpha = \frac{4+r}{2}\id_\mH$.
\end{proof}

\begin{remark}%
    \label{rem:r}
    Note that $r = 2n_\alpha$ for all $\alpha \in \Phi$ where $n_\alpha$ is as in \cref{def:charV2} because $\Phi$ is simply laced.
\end{remark}

\begin{proposition}
    We have
    \[
        \one = \frac{6+r}{2} \id_\mH.
    \]
\end{proposition}
\begin{proof}
    Let $v_1,\dots,v_n$ be an orthonormal basis for the Cartan subalgebra $\mH$ of $\mL$ with respect to the Killing form $\kappa$.
    Then $\{v_1,\dots,v_n\} \cup \{e_\alpha \mid \alpha \in \Phi\}$ is a basis for $\mL$.
    Note that for the dual basis, we have $v_i^* = v_i$ for all $i$ and $e_\alpha^* = e_{-\alpha}$ for all $\alpha \in \Phi$.
    Therefore we can write $C_\mL$ as
    \[
        C_\mL = \sum_{i} v_i v_i + \sum_{\alpha \in \Phi} e_\alpha e_{-\alpha}.
    \]
    By \cref{lem:sumj}
    \[
        \sum_i v_i v_i = \zeta^{-1}(\id_\mH) = \frac2{4+r} \sum_{\alpha \in \Phi^+} \zeta^{-1}(j_\alpha) = \frac2{4+r} \sum_{\alpha \in \Phi^+} h_\alpha h_\alpha.
    \]
    Recall the projection $\pi \colon S^2(\mL) \to \mA \colon v \mapsto \overline{v}$ from \cref{def:A}.
    By \cref{prop:gensetV} we have $2\overline{e_\alpha e_{-\alpha}} = \overline{h_\alpha h_\alpha}$.
    As a result
    \[
        \sum_{\alpha \in \Phi} \overline{e_\alpha e_{-\alpha}} = \sum_{\alpha \in \Phi^+} \overline{h_\alpha h_{\alpha}}.
    \]
    We have $C_\mL \in \mA$ by \cref{prop:unitA}~\ref{prop:unitA:unitA}, so $C_\mL = \overline{C_\mL}$. Thus
    \[
        C_\mL = \left( \frac2{4+r} + 1 \right) \sum_{\alpha \in \Phi^+} \overline{h_\alpha h_\alpha} = \frac{6+r}{4+r} \theta\Bigl(\sum_{\alpha \in \Phi^+} j_\alpha \Bigr).
    \]
    The statement now follows from \cref{lem:sumj}.
\end{proof}

\section{Decomposition algebras}%
\label{sec:decompositionalgebras}
In the following sections, we will give our algebra $(\fA,\ast)$ the structure of a decomposition algebra.
Decomposition algebras have only recently been introduced in~\cite{DMPSVC19} as a generalization of axial algebras.
We repeat the definition here.
The main idea is that the algebra decomposes (as a vector space) into many decompositions $A=\bigoplus_{x \in X} A_x$, each of which is indexed by the same set $X$.
The parts $A_x$ of each of these decompositions multiply according to a fixed fusion law.
More precisely, the fusion law will tell us when $A_x A_y$ has a non-zero component in $A_z$ for $x,y,z \in X$.

\begin{definition}[\cite{DMPSVC19}*{\S~2}]
    \begin{enumerate}
        \item A fusion law is a pair $(X,\fus)$ where $X$ is a set and $\fus$ is a map $X \times X \to 2^X$, where $2^X$ denotes the power set of $X$.
        \item An element $\funit \in X$ is called a unit for the fusion law $(X,\fus)$ if $\funit \fus x \subseteq \{x\}$ and $x \fus \funit \subseteq \{x\}$ for all $x \in X$.
        \item Let $(X,\fus)$ and $(Y,\fus)$ be fusion laws. The product of these fusions laws is the fusion law $(X \times Y,\fus)$ given by the rule
        \[
            (x_1,y_1) \fus (x_2,y_2) = \{(x,y) \mid x \in x_1 \fus x_2, y \in y_1 \fus y_2 \}
        \]
        for all $x_1,x_2 \in X$ and $y_1,y_2 \in Y$.
    \end{enumerate}
\end{definition}

An important class of fusion laws comes from abelian groups.

\begin{definition}[{\cite{DMPSVC19}*{Definitions 2.10 and 3.1}}]
    Let $G$ be an abelian group.
    \begin{enumerate}
        \item Let
            \[
                \fus : G \times G \mapsto 2^G : (g,h) \mapsto \{gh\}.
            \]
            Then we call $(G,\fus)$ the group fusion law of $G$.
        \item Suppose that $(X,\fus)$ is an arbitrary fusion law.
            A $G$-grading of $(X,\fus)$ is a map $\xi : X \mapsto G$ such that
            \[
                \xi(x \fus y) \subseteq \xi(x) \fus \xi(y),
            \]
            for all $x,y \in X$.
    \end{enumerate}
\end{definition}

We are ready to formulate the definition of a decomposition algebra.

\begin{definition}[\cite{DMPSVC19}*{\S~4}]
    Let $k$ be a field and $\mF = (X,\fus)$ be a fusion law.
    \begin{enumerate}
        \item An $\mF$-decomposition of a $k$-algebra $A$ (not assumed to be commutative, associative or unital) is a direct sum decomposition $A = \bigoplus_{x \in X} A_x$ (as a vector space) indexed by the set $X$ such that $A_x A_y \subseteq \bigoplus_{z \in x \fus y} A_z$. For the sake of readability, we will denote $\bigoplus_{z \in Z} A_z$ by $A_Z$ for any $Z \subseteq X$.
        \item An $\mF$-decomposition algebra is a triple $(A,\mI,\Omega)$ where $A$ is a $k$-algebra, $\mI$ is an index set and $\Omega$ is a tuple of decompositions $A = \bigoplus A_x^i$ indexed by the set $\mI$.
    \end{enumerate}
\end{definition}

There is a close connection between decomposition algebras with a graded fusion law and groups. We refer to~\cite{DMPSVC19}*{\S~6} for more details.

\begin{definition}
    Let $G$ be an abelian group and let $(X,\fus)$ be a $G$-graded fusion law with grading map $\xi$.
    Let $(A,\mI,\Omega)$ be an $(X,\fus)$-decomposition algebra over a field~$k$. For each linear character $\chi \in \Hom(G,k^\times)$ and each $i \in \mI$, we define an automorphism of~$A$ as follows:
    \[
        \tau_{i,\chi} : A \to A : a \mapsto \chi(\xi(x))a \quad \text{for all $x \in X$ and all $a \in A_x^i$.}
    \]
    Since $A = \bigoplus_x A_x^i$, this determines $\tau_{i,\chi}$, and the fact that this map is an automorphism of $A$ follows from the fusion law and its $G$-grading.
    The subgroup of $\Aut(A)$ generated by all $\tau_{i,\chi}$ for $i \in \mI$ and all $\chi \in \Hom(G,k^\times)$ is called the Miyamoto group of the decomposition algebra (with respect to the grading $\xi$).
\end{definition}

The definition of decomposition algebras originates from the theory of axial algebras. These are decomposition algebras where the decomposition are given by eigenspace decompositions of operators of the form
\[
    \ad_a : A \to A : b \mapsto ab
\]
for certain elements $a \in A$, called axes.
Axial decomposition algebras fulfill the role of axial algebras within the framework of decomposition algebras.
They are more general than axial algebras since we only demand that these operators $\ad_a$ act as a scalar on each part of the decomposition, allowing the possibility that some of these scalars coincide.

\begin{definition}[{\cite{DMPSVC19}*{\S 5}}]
    Let $\mF = (X , \fus)$ be a fusion law with a distinguished unit $\funit \in X$.
    \begin{enumerate}
        \item Let $\bigoplus_{x \in X} A_x$ be an $\mF$-decomposition of a $k$-algebra $A$. We call a nonzero element $a \in A_{\funit}$ an axis for this decomposition if there exist scalars $\nu_x$ for each $x \in X$ such that
            \[
                ab = \nu_x b
            \]
            for all $b \in A_x$. The map $\nu : X \mapsto k : x \mapsto \nu_x$ is called the evaluation map of the axis.
        \item An axial decomposition algebra with evaluation map $\nu : X \mapsto k : x \mapsto \nu_x$ is a quadruple $(A,\mI,\Omega,\ax)$ such that $(A,\mI,\Omega)$ is a decomposition algebra and $\ax : \mI \to A$ is a map such that for each $i \in \mI$, $\ax_i \coloneqq \ax(i)$ is an axis for the decomposition $A=\bigoplus_{x \in X}A_x^i$ with evaluation map $\nu$.
    \end{enumerate}
\end{definition}

We end this section by providing a procedure to obtain a decomposition algebra out of an algebra on which a group or Lie algebra acts by automorphisms.

\begin{lemma}%
    \label{lem:dec}
    Let $G$ be a finite group or complex semisimple Lie algebra.
    Let $A$ be an algebra for $G$, i.e., a complex $G$-representation equipped with a $G$-equivariant bilinear product.
    Let $A=\bigoplus_{x \in X} A_x$ be the decomposition of $A$ into $G$-isotypic components.
    This means that for each $x \in X$ we have $A_x = nW_x \neq 0$ for some $n \in \mathbb{N} \setminus \{0\}$ and some irreducible $G$-representation $W_x$ and $W_x \ncong W_y$ for $x \neq y$.
    Define
    \[
        \fus \colon X \times X \to 2^X \colon (x,y) \mapsto \{z \in X \mid \Hom_G(W_x \otimes W_y , W_z) \neq 0 \}
    \]
    where $\Hom_G$ stands for the space of homomorphisms of $G$-representations.
    Then $\bigoplus_{x \in X} A_x$ is an $(X,\fus)$\dash{}decomposition of $A$.
\end{lemma}
\begin{proof}
    The case where $G$ is a group follows from~\cite{DMPSVC19}*{Theorem 7.4}. The case where $G$ is a complex semi-simple Lie algebra is proven analogously.
\end{proof}

\begin{remark}%
    \label{rem:dec}
    Determining whether $\Hom_G(W_x \otimes W_y,W_z) \neq 0$ can be done using character theory.
    If $G$ is a finite group and $\chi_x$, $\chi_y$ and $\chi_z$ are the respective characters of $W_x$, $W_y$ and $W_z$ then $\Hom_G(W_x \otimes W_y,W_z) \neq 0$ if and only if $\langle \chi_x\chi_y ,\chi_z \rangle \neq 0$ where $\langle \; ,\; \rangle$ denotes the inner product on the space of class functions of $G$.
    A similar argument works of $G$ is a semisimple complex Lie algebra where we have to use formal characters~\cite{Hum72}*{\S 22.5}.
    For a group or semisimple Lie algebra $G$, we write $\Irr(G)$ for its set of irreducible characters.
\end{remark}

We will use this lemma to obtain a ``global decomposition'' and a class of ``local decompositions'' for our algebra.

\begin{definition}%
    \label{def:gldec}
    Let $G$ be a finite group or a complex semisimple Lie algebra.
    Let $\mI$ be an index set and $(H_i\mid i\in \mI)$ an $\mI$-tuple of conjugate subgroups or conjugate semisimple subalgebras respectively.
    Let $A$ be an algebra for $G$.
    \begin{enumerate}
        \item Apply \cref{lem:dec} to the group $G$ and the algebra $A$ to obtain a decomposition $A = \bigoplus_{x \in X_g} A_x$ of the algebra $A$. Its components are the $G$-isotypic components of $A$ as $G$-representation. We call this decomposition the global decomposition of $A$ with respect to $G$. Denote the fusion law by $(X_g,\fus)$. (The subscript ``g'' stands for ``global''. In \cref{sec:0dec,sec:dec} we will denote elements of $X_g$ by letters.)
        \item Apply \cref{lem:dec} to each $H_i$ to obtain a decomposition $A = \bigoplus_{x \in X_l} A^i_x$ for each $i \in \mI$. Note that these decompositions are all conjugate since we assume the $H_i$'s to be conjugate. Therefore, we can index these these decomposition by the same index set $X_l$.  We call these decompositions the local decompositions of $A$ with respect to $(H_i \mid i\in \mI)$. Also the corresponding fusion law $(X_l,\fus)$ does not depend on $i \in \mI$. (The subscript ``l'' stands for ``local''. In \cref{sec:0dec,sec:dec} we will denote elements of $X_l$ by numbers.)
    \end{enumerate}
\end{definition}

We combine the global decomposition with each of the local decompositions to obtain a new decomposition which is a ``refinement'' of both.

\begin{lemma}%
    \label{lem:gldec}
    Consider the situation of \cref{def:gldec}. For each $x \in X_g$, $y \in X_l$ and $i\in \mI$ let $A^i_{x,y}=A_x \cap A^i_y$.
    Let $\mF$ be the direct product of the fusion laws $(X_g,\fus)$ and $(X_l,\fus)$. Then $A=\bigoplus_{x,y} A^i_{x,y}$ is a $\mF$\dash{}decomposition of $A$ for each $i \in \mI$.
\end{lemma}
\begin{proof}
    Because each $H_i$ is a subgroup or semisimple subalgebra of $G$, we have that $\bigoplus_{y} A^i_{x,y}$ must be the decomposition of $A_x$ into $H_i$-isotypic components. This proves that $A=\bigoplus_{x,y} A^i_{x,y}$ is a decomposition of $A$. Since $A = \bigoplus_{x \in X_g} A_x$ is an $(X_g,\fus)$\dash{}decomposition of $A$ and $A = \bigoplus_{x \in X_l} A^i_x$ is an $(X_l,\fus)$\dash{}decomposition of $A$, this decomposition is an $\mF$-decomposition.
\end{proof}

\section{Decompositions of the zero weight subalgebra}%
\label{sec:0dec}

The goal of this section is to give the algebra $(\fJ,\pa)$ the structure of a decomposition algebra.
We will describe the general procedure and give the explicit decomposition for each of the possible types ($A_n$, $D_n$ or $E_n$) afterwards.
From now on we will only consider the product $\pa$ on the space $\fJ$ so will simply omit it from our notation.

\subsection{The general procedure}

Note that $\fJ$ is the zero weight space of $\fA$ as $\mL$-representation.
Therefore, the Weyl group $W$ of $\mL$ acts by automorphisms on the algebra $\fJ$; see \cref{def:Waction}.
We can use the ideas and terminology from \cref{def:gldec} to obtain a decomposition algebra.

\begin{definition}%
    \label{def:0dec}
    \begin{enumerate}
        \item For each $\alpha \in \Phi^+$ let $C_W(s_\alpha)$ be the centralizer in $W$ of the reflection $s_\alpha$. Since $\Phi$ is irreducible and simply laced, these subgroups are conjugate inside $W$.
        \item Let $\bigoplus_{x \in X_g^0} J_x$ be the global decomposition of $\fJ$ with respect to $W$, cf.\ \cref{def:gldec}.
            Denote its global fusion law by $(X_g^0,\fus)$.
            Let $\bigoplus_{x \in X_l^0} J_x^\alpha$ be the local decompositions of $\fJ$ with respect to $(C_W(s_\alpha)\mid \alpha \in \Phi^+)$.
            Write $(X_l^0,\fus)$ for the corresponding fusion law.
        \item As in \cref{lem:gldec} let $J_{x,y}^\alpha \coloneqq J_x \cap J^\alpha_y$ for $x \in X_g^0$ and $y \in X_l^0$. Let $\mF_0$ be the direct product of the fusion laws $(X_g^0,\fus)$ and $(X_l^0,\fus)$. Write $\Omega_0$ for the $\Phi^+$\dash{}tuple of $\mF_0$\dash{}decompositions $\bigoplus_{x,y} J_{x,y}^\alpha$. Then $(\fJ,\Phi^+,\Omega_0)$ is an $\mF_0$-decomposition algebra.
    \end{enumerate}
\end{definition}

We prove that $C_W(s_\alpha)$ is a reflection subgroup of $W$, this is, a subgroup generated by reflections.
This makes it easier to determine the local decompositions and their fusion law.

\begin{proposition}%
    \label{prop:refl}
    For each $\alpha \in \Phi^+$, the centralizer $C_W(s_\alpha)$ is a reflection subgroup of $W$.
    It is generated by the reflections $s_\beta$ for which $\kappa(\alpha,\beta) = 0$ or $\beta = \pm \alpha$.
    Its Dynkin diagram can be obtained by removing the neighbors of the extending node from the extended Dynkin diagram of $\Phi$.
\end{proposition}
\begin{proof}
    Recall that the extending node of the Dynkin diagram corresponds to the negative of the highest root~\cite{Bou02}*{Chapter~VI, \S~3}.
    Since $W$ acts transitively on $\Phi$, it suffices to prove this when $\alpha$ is the highest root of $\Phi$.
    For $w \in W$, we have $\conj{w}{s_\alpha} = s_\alpha$ if and only if $w$ fixes the hyperplane orthogonal to $\alpha$.
    This means that $w$ must map $\alpha$ to $\pm \alpha$.
    From~\cite{Hum72}*{\S~10.3, Lemma~B} it follows that those $w$ that fix $\alpha$ must be a product of reflections $s_\beta$ with $\kappa(\alpha,\beta) = 0$.
    Since $s_\alpha(\alpha) = -\alpha$, this proves the statement.
\end{proof}

\begin{remark}%
    \label{rem:type}
    The following table gives the type of the subsystem \[\{\pm \alpha\} \cup \{\beta \in \Phi \mid \kappa(\alpha,\beta) = 0 \}\] for each of the possible types of $\Phi$.
    \[\begin{array}{cc} \toprule
        W & C_W(s_\alpha) \\ \midrule
        A_n (n \geq 3) & A_1 \times A_{n-2} \\
        D_n (n \geq 4) & D_2 \times D_{n-2} \\
        E_6 & A_1 \times A_5 \\
        E_7 & A_1 \times D_6 \\
        E_8 & A_1 \times E_7 \\ \bottomrule
    \end{array}\]
    Here we use the convention that $D_2 \cong A_1 \times A_1$ and $D_3 \cong A_3$.
\end{remark}

The fusion law $\mF_0$ is $\Z/2\Z$\dash{}graded and the corresponding Miyamoto group of the $\mF_0$\dash{}decomposition algebra from \cref{def:0dec} is isomorphic to $W$.
First we prove that the fusion law $(X_l^0,\fus)$ is $\Z/2\Z$\dash{}graded.

\begin{lemma}%
    \label{lem:grad0}
    The fusion law $(X_l^0,\fus)$ has a non-trivial $\Z/2\Z$-grading $\xi \colon X_l^0 \to \Z/2\Z$.
    Let $\chi$ be the non-trivial linear $\C$\dash{}character of $\Z/2\Z$.
    The Miyamoto map $\tau_{\alpha,\chi}$ of $(\fJ,\Phi^+,\Omega_0)$ is precisely the element $s_\alpha$ in its action on $\fJ$.
\end{lemma}
\begin{proof}
    This follows from~\cite{DMPSVC19}*{Example 7.3}.
\end{proof}

This $\Z/2\Z$-grading induces a $\Z/2\Z$-grading of the fusion law $\mF_0$.

\begin{definition}%
    \label{def:grad0}
    The $\Z/2\Z$-grading $\xi$ of $(X_l^0,\fus)$ from \cref{lem:grad0} induces a $\Z/2\Z$-grading of $\mF_0$:
    \[
        \xi: \mF_0 \to \Z/2\Z : (x,y) \mapsto \xi(y).
    \]
\end{definition}

\begin{proposition}%
    \label{prop:0dec}
    Let $(\fJ,\Phi^+,\Omega_0)$ be the $\mF_0$-decomposition algebra from \cref{def:0dec}.
    The Miyamoto group with respect to the grading of $\mF_0$ from \cref{def:grad0} is the Weyl group $W$ in its action on $\fJ$.
\end{proposition}
\begin{proof}
    This follows from the definitions and \cref{lem:grad0}.
\end{proof}

\begin{remark}
    There are two ways to refine the decomposition and fusion law.
    \begin{enumerate}
        \item Of course, many of the intersections $J^\alpha_{x,y} = J_x \cap J^\alpha_y$ for $x \in X_g^0$ and $y \in X_l^0$ are trivial.
            Therefore, we can omit them from our fusion law.
        \item Instead of considering the global decomposition with respect to $W$, we can consider the global decomposition with respect to the automorphism group $\Aut(\Phi)$ of the root system $\Phi$.
            If the Dynkin diagram of $\Phi$ admits a non-trivial graph automorphism, then this group is possibly larger than $W$ but still acts by automorphisms on $\fJ$.
            This leads to another global decomposition $\bigoplus_{x \in X_a} J_x$ with fusion law $(X_a^0,\fus)$.
            Let $J_{x,y,z}^\alpha \coloneqq J_x \cap J_y \cap J^\alpha_z$ for $x \in X_a^0$, $y \in X_g^0$ and $z \in X_l^0$.
            Then $\bigoplus_{x,y,z} J_{x,y,z}^\alpha$ will be a decomposition of $\fJ$ whose fusion law is the direct product of $(X_a^0,\fus)$, $(X_g^0,\fus)$ and $(X_l^0,\fus)$.
    \end{enumerate}
\end{remark}

It would be cumbersome to include all the computations that were needed to obtain the explicit decompositions.
Instead, we will only present the results, hoping that the reader can fill in the details if necessary.
First, we give a construction for the simply laced root systems.

\begin{example}%
    \label{ex:rs}
    \begin{enumerate}
        \item\label{ex:rs:An} Consider a Euclidean space $E$ of dimension $n + 1$ and pick an orthonormal basis $b_0,b_1,\dots,b_n$ for $E$. Then
            \[
                \Phi = \{b_i - b_j \mid 0 \leq i,j \leq n , i \neq j \rangle\}
            \]
            is a root system in the subspace consisting of the vectors $\sum_{i = 0}^n \lambda_i b_i$ for which $\sum_{i=0}^n \lambda_i = 0$.
            The following vectors form a base for $\Phi$.
            \[
                \dynkin[labels={b_0-b_1,b_1-b_2,b_{n-2}-b_{n-1},b_{n-1}-b_n},edge length=1.5cm]A{}
            \]
            With respect to this base, we have $\Phi^+=\{b_i - b_j \mid 0 \leq i < j \leq n\}$.
            This root system is of type $A_n$.
            The action of its Weyl group can be extended to $E$ such that it permutes the basis elements $\{ b_0,\dots,b_n \}$. This defines an isomorphism with the symmetric group $S_{n+1}$ on $n+1$ elements.
        \item\label{ex:rs:Dn} Let $b_1,\dots,b_n$ be an orthonormal basis for a Euclidean space of dimension $n$. Then
            \[
                \Phi = \{\pm b_i \pm b_j \mid 1 \leq i,j \leq n, i\neq j\}
            \]
            forms a root system of type $D_n$. A base for $\Phi$ is given by the following vectors.
            
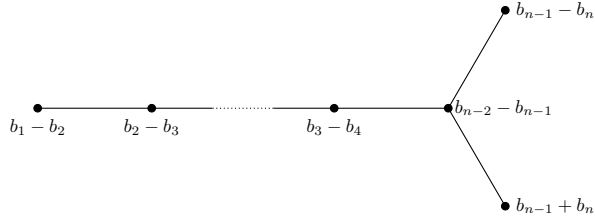
\[
                \begin{dynkinDiagram}[edge length=1.5cm,labels={b_1-b_2,b_2-b_3,b_3-b_4,,b_{n-1}-b_n,b_{n-1}+b_n}]D{} \node[anchor=west,scale=0.7] at (root 4) {$b_{n-2}-b_{n-1}$}; \end{dynkinDiagram}
            \]
            We have $\Phi^+ = \{ b_i \pm b_j \mid 1 \leq i < j \leq n \}$.

            Consider the subgroup $W_n$ of $\GL(E)$ generated by the the elements $\theta$ for which there exists a permutation $\pi$ of $\{1,\dots,n\}$ such that $\theta(b_i) = \pm b_{\pi(i)}$ for all $1 \leq i \leq n$.
            Then $W_n$ is a split extension of $S_n$ by an elementary abelian 2\dash{}group of order $2^n$.
            The Weyl group of $\Phi$ is an index 2 subgroup of $W_n$.
            It consists of those elements of $W_n$ that have determinant one.
            We denote this group by $W_n'$.
            It is an extension of $S_n$ by an elementary abelian 2\dash{}group of order $2^{n-1}$.
        \item\label{ex:rs:En} Let $E$ be a Euclidean space of dimension 8 and $\{b_1,\dots,b_n\}$ an orthonormal basis for $E$. Then
            \[
                \Phi = \{ \pm b_i \pm b_j \mid 1 \leq i,j\leq 8\} \cap \{\textstyle \tfrac{1}{2}\sum_{i=1}^8 \epsilon_i b_i \mid \epsilon_i = \pm 1 , \prod_{i=1}^8 \epsilon_i = -1 \}
            \]
            is a root system of type $E_8$. Consider the roots:
            \begin{align*}
                \alpha_1 & \coloneqq \tfrac12 (-b_1-b_2-b_3-b_4-b_5+b_6+b_7+b_8), \\
                \alpha_2 & \coloneqq \tfrac12 (-b_1-b_2-b_3+b_4+b_5+b_6+b_7+b_8).
            \end{align*}
            Then following roots form a base for $\Phi$.
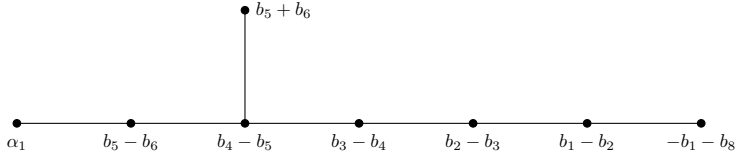
            \[
                \dynkin[edge length=1.5cm,labels={\alpha_1,b_5+b_6,b_5-b_6,b_4-b_5,b_3-b_4,b_2-b_3,b_1-b_2,-b_1-b_8}]E8
            \]
            The roots contained in the subspace of vectors $\sum_{i=1}^8 \lambda_i b_i$ satisfying $\lambda_7=\lambda_8$ form a root system of type $E_7$. We have the following base for this root system.
            \[
                \dynkin[edge length=1.5cm,labels={\alpha_1,b_5+b_6,b_5-b_6,b_4-b_5,b_3-b_4,b_2-b_3,b_1-b_2}]E7
            \]
            Next, we consider the roots contained in the subspace of vectors $\sum_{i=1}^8 \lambda_i b_i$ satisfying $\lambda_7=\lambda_8$ and $\sum_{i=1}^6 \lambda_i = 0$. They form a root system of type $E_6$. A base for this root system is given by the following roots.
            \[
                \dynkin[edge length=1.5cm,labels={b_1-b_2,\alpha_2,b_2-b_3,b_3-b_4,b_4-b_5,b_5-b_6}]E6
            \]
    \end{enumerate}
\end{example}

The necessary information about the character theory of Weyl groups can be found in~\cite{GP00}.
The irreducible characters of Weyl groups of type $A_n$ and $D_n$ can be described using compositions and partitions.

\begin{definition}
    A \emph{composition} of a positive integer $n$ is an ordered sequence $\lambda=(\lambda_1,\dots,\lambda_r)$ such that $\abs{\lambda} \coloneqq \sum_{i=1}^r \lambda_i = n$.
    A \emph{partition} of a positive integer $n$ is an unordered sequence $\lambda = [\lambda_1,\dots,\lambda_r]$ such that $\abs{\lambda} \coloneqq \sum_{i=1}^r \lambda_i = n$.
    For a composition $\lambda$ we write $[\lambda]$ for its corresponding partition.
    For each partition we define a corresponding integer $a(\lambda)$, called the \emph{$a$\dash{}invariant} of $\lambda$, by the formula
    \[
        a(\lambda) \coloneqq \sum_{1 \leq i < j \leq r} \min\{\lambda_i,\lambda_j\}.
    \]
    If we order the sequence $\lambda$ such that $\lambda_1 \geq \lambda_2 \geq \cdots \geq \lambda_r$ then $a(\lambda) = \sum_{i=1}^r (i-1)\lambda_i$.
\end{definition}

Let us now describe the characters of the Weyl group of type $A_n$, which is isomorphic to the symmetric group on $n+1$ elements; see \cref{ex:rs}~\labelcref{ex:rs:An}.
Note that we use the notation $\Ind_H^G(\chi)$ and $\Res_H^G(\chi)$ for induced, respectively restricted, characters for groups $H \leq G$.

\begin{definition}
    Let $\lambda = (\lambda_1,\dots,\lambda_r)$ be a composition of $n+1$.
    We denote by $S_{\lambda}$ the subgroup of $S_{n+1}$ that permutes amongst themselves the first $\lambda_1$ numbers, the next $\lambda_2$ numbers and so on.
    Denote by $\one_{\lambda}$ the trivial character of $S_{\lambda}$.
    Then we can index the irreducible characters of $S_{n+1}$ by the partitions of $n+1$.
    We write $\chi_\mu$ for the character corresponding to the partition $\mu$.
    This can be done in such a way that
    \[
        \Ind_{S_\lambda}^{S_{n+1}}(\one_\lambda) = \chi_{[\lambda]} + \text{ a linear combination of $\chi_\mu$ with $a(\mu) < a([\lambda])$}.
    \]
    See~\cite{GP00}*{Theorem~5.4.7, p.\ 158} for a proof of this fact.
\end{definition}

The characters of the Weyl group of type $D_n$ can be described in a similar manner. Recall the definitions of the groups $W_n$ and $W_n'$ from \cref{ex:rs}~\labelcref{ex:rs:Dn}.

\begin{definition}
    The irreducible characters of $W_n$ can be indexed by pairs $(\lambda,\mu)$ of partitions such that $\abs{\lambda} + \abs{\mu} = n$.
    We allow that $\abs{\lambda}=n$ or $\abs{\mu}=n$ and in that case we write $\emptyset$ for the other partition.
    We write $\chi_{(\lambda,\mu)}$ for the character of $W_n$ corresponding to the partition $(\lambda,\mu)$.
    If $\lambda \neq \mu$ then the restriction $\chi_{(\lambda,\mu)}' \coloneqq \Res_{W_n'}^{W_n} \chi_{(\lambda,\mu)}$ is an irreducible character of $W_n'$.
    We have $\chi_{(\lambda,\mu)}' = \chi_{(\mu,\lambda)}'$.
    If $\lambda = \mu$, then $\Res_{W_n'}^{W_n} \chi_{(\lambda,\lambda)}$ is the sum of two distinct irreducible characters for $W_n'$. We will denote these by $\chi_{(\lambda,+)}$ and $\chi_{(\lambda,-)}$.
    These characters exhibit all irreducible characters of $W_n'$.
\end{definition}

Next, we illustrate how we can compute the representation fusion law for these groups.
We will do this for the Weyl group of type $A_n$ or, equivalently, the symmetric group on $n+1$ elements.
In order to determine the representation fusion law of a group $G$, it suffices, by \cref{rem:dec}, to decide whether $\langle \chi_1\chi_2 ,\chi_3 \rangle = 0$ for $\chi_1,\chi_2,\chi_3 \in \Irr(G)$.
To this end, we can use the following lemma.

\begin{lemma}
    Let $\lambda$ be a composition of $n+1$ and $\psi,\psi' \in \Irr(S_{n+1})$. Then
    \[
        \langle \chi_{[\lambda]}\psi ,\psi' \rangle = \langle \Ind_{S_\lambda}^{S_{n+1}} \Res_{S_\lambda}^{S_{n+1}} \psi ,\psi' \rangle - \langle (\Ind_{S_\lambda}^{S_{n+1}} \one_\lambda - \chi_{[\lambda]})\psi , \psi' \rangle.
    \]
\end{lemma}
\begin{proof}
    We have
    \begin{align*}
        \chi_{[\lambda]}\psi &= (\Ind_{S_\lambda}^{S_{n+1}} \one_\lambda) \psi - (\Ind_{S_\lambda}^{S_{n+1}} \one_\lambda - \chi_{[\lambda]} )\psi, \\
                             &= \Ind_{S_\lambda}^{S_{n+1}} \Res_{S_\lambda}^{S_{n+1}} \psi - (\Ind_{S_\lambda}^{S_{n+1}} \one_\lambda - \chi_{[\lambda]})\psi.
    \end{align*}
    This proves the statement.
\end{proof}

We can compute the irreducible constituents of $\Ind_{S_\lambda}^{S_{n+1}} \Res_{S_\lambda}^{S_{n+1}} \psi$ and $\Ind_{S_\lambda}^{S_{n+1}} \one_\lambda$ using the Littlewood--Richardson rule; see~\cite{GP00}*{\S~6.1}. 
Recall that the partitions corresponding to the irreducible consituents of $\Ind_{S_\lambda}^{S_{n+1}} \one_\lambda - \chi_{[\lambda]}$ have $a$-invariant less than $a([\lambda])$.
Therefore we can use induction on the $a$-invariant of $[\lambda]$ to compute $\langle \chi_{[\lambda]}\psi,\psi'\rangle$.

Determining the representation fusion law for $W_n'$ is a bit more cumbersome but the same technique applies.
The necessary background can be found in~\cite{GP00}*{\S~5.5, \S~5.6 and~\S~6.1}.

We are now ready to give local and global decompositions of $(\fJ,\pa)$ and their fusion law.
Note that this does only depend on the structure of $\fJ$ as a representation for the Weyl group $W$.

For each of the possible types, we will give the following information:

\begin{enumerate}
    \item some more notation about the root system that allows to describe the decompositions;
    \item the global decomposition $\bigoplus_{x \in X_g^0} J_x$ (elements of $X_g^0$ will be denoted by letters);
    \item the characters and dimensions of the $W$\dash{}representations $J_x$ for $x \in X_g^0$;
    \item the global fusion law $(X_g^0,\fus)$;
    \item the elements of the full decomposition $\bigoplus_{x \in \mF} J_x^\alpha$ with respect to a root $\alpha$; from this the local decomposition $\bigoplus_{i \in X_l^0} J_i^\alpha$ can be derived (elements of $X_l^0$ will be denoted by numbers);
    \item the characters and dimensions of the $C_W(s_\alpha)$\dash{}representations $J_i^\alpha$ for $i \in X_l^0$;
    \item the local fusion law $(X_l^0,\fus)$.
\end{enumerate}

\subsection{Type \texorpdfstring{$A_n$}{An}}

\begin{enumerate}
    \item We use the description of the root system of type $A_n$ from \cref{ex:rs}~\labelcref{ex:rs:An}. Denote the orthogonal projection onto $\langle \Phi \rangle = \langle b_i - b_j \mid 0 \leq i,j \leq n\rangle$ of a basis vector $b_i$ by $b_i'$. We identify $\fJ$ with $S^2(\mH)=S^2(\langle \Phi \rangle)$ using \cref{prop:S}. We will give the full decomposition $\bigoplus_{x \in \mF} J_x^\alpha$ with respect to the root $\alpha \coloneqq b_0 - b_n$. This is the highest root with respect to the base from \cref{ex:rs}~\labelcref{ex:rs:An}. We use the isomorphisms $W \cong S_{n+1}$ and $C_W(s_\alpha) \cong S_2 \times S_{n-1}$ to describe the characters.

    \item\begin{minipage}[t]{\linewidth}
        \begin{align*}
            J_a &\coloneqq \langle \sum_{i=0}^{n} b_i'b_i' \rangle \\
            J_b &\coloneqq \langle b_i'b_i' - b_j'b_j' \mid 0 \leq i,j \leq n \rangle \\
            J_c &\coloneqq \langle (b_i-b_j)(b_k-b_l) \mid 0 \leq i,j,k,l \leq n,\{i,j\} \cap \{k,l\} = \emptyset \rangle
        \end{align*}
    \end{minipage}

    \item\begin{minipage}[t]{\linewidth}
        \captionsetup{type=table}
        \[\begin{array}{ccc} \toprule
            \text{Component} & \text{Character} & \text{Dimension} \\ \midrule
            J_a & \chi_{[n+1]} & 1 \\
            J_b & \chi_{[n,1]} & n \\
            J_c & \chi_{[n-1,2]} & \frac{n^2-n-2}{2} \\ \bottomrule
        \end{array}\]
        \caption{Characters and dimensions for the global decomposition of \texorpdfstring{$\fJ$}{J} for type \texorpdfstring{$A_n$}{An}.}%
        \label{tbl:0decgAchar}
    \end{minipage}

    \item\begin{minipage}[t]{\linewidth}
        \captionsetup{type=table}
        \[\begin{array}{c|ccc} 
            \fus &  a   &   b   &   c \\ \hline 
            a   &   a   &   b   &   c \\
            b   &   b   &   a,b,c   & b,c^\sspec    \\
            c   &   c   &   b,c^\sspec  &   a,b^\sspec,c \\
        \end{array}\]
        \caption{The fusion law \texorpdfstring{$(X_g^0,\fus)$}{} for type \texorpdfstring{$A_n$}{An}. \texorpdfstring{Entries marked with $\sspec$ should be left out for $n=3$.}{}}%
        \label{tbl:0decgAfus}
    \end{minipage}

    \item\begin{minipage}[t]{\linewidth}
        \begin{align*}
            J_{a,1}^\alpha &\coloneqq \langle \sum_{i=0}^{n} b_i'b_i' \rangle \\
            J_{b,1}^\alpha &\coloneqq \langle (n+1)(b_0'b_0' + b_{n}'b_{n}') - 2\sum_{k=0}^{n} b_k'b_k' \rangle \\
            J_{b,2}^\alpha &\coloneqq \langle b_k'b_k' - b_l'b_l' \mid 1 \leq k,l \leq n-1 \rangle \\
            J_{b,4}^\alpha &\coloneqq \langle b_0'b_0' - b_{n}'b_{n}' \rangle \\
            J_{c,1}^\alpha &\coloneqq \langle nb_0'b_0' + nb_{n}'b_{n}' + n(n-1)b_0'b_{n}' - \sum_{k=0}^{n} b_k'b_k' \rangle \\
            J_{c,2}^\alpha &\coloneqq \langle ((n-1)(b_0'+b_{n}')+2(b_k'+b_l'))(b_k'-b_l') \mid 1 \leq k,l \leq n-1 \rangle \\
            J_{c,3}^\alpha &\coloneqq \begin{multlined}[t] \langle (b_{k_1}' - b_{l_1}')(b_{k_2}'-b_{l_2}') \mid 1 \leq k_1,k_2,l_1,l_2 \leq n-1,\\ \{k_1,l_1\} \cap \{k_2,l_2\} = \emptyset  \rangle
            \end{multlined} \\
            J_{c,5}^\alpha &\coloneqq \langle (b_0'-b_{n}')(b_k'-b_l') \mid 1 \leq k,l \leq n-1 \rangle
        \end{align*}
    \end{minipage}

    \item\begin{minipage}[t]{\linewidth}
        \captionsetup{type=table}
        \[\begin{array}{ccc} \toprule
            \text{Component} & \text{Character} & \text{Dimension} \\ \midrule
            J_1^\alpha & 3\cdot \chi_{[2]} \times \chi_{[n-1]} & 3\cdot 1\\
            J_2^\alpha & 2\cdot \chi_{[2]} \times \chi_{[n-2,1]} & 2 \cdot (n-2) \\
            J_3^\alpha & \chi_{[2]} \times \chi_{[n-3,2]} & \frac{n^2-5n-4}{2} \\
            J_4^\alpha & \chi_{[1,1]} \times \chi_{[n-1]} & 1 \\
            J_5^\alpha & \chi_{[1,1]} \times \chi_{[n-2,1]} & n-2 \\ \bottomrule
        \end{array}\]
        \caption{Characters and dimensions for the local decomposition of \texorpdfstring{$\fJ$}{J} for type \texorpdfstring{$A_n$}{An}.}%
        \label{tbl:0declAchar}
    \end{minipage}

    \item\begin{minipage}[t]{\linewidth}
        \captionsetup{type=table}
        \[ \begin{array}{c|ccccc} 
            \fus &  1   &   2   &   3^\spec   &   4   &   5 \\ \hline 
            1   &   1   &   2   &   3^\spec &   4   &   5 \\
            2   &   2   &   1,2^\sspec,3^\spec  &   2^\spec,3^\spec &   5   &   4,5^\sspec \\
            3^\spec   &   3^\spec &   2^\spec,3^\spec &   1^\spec,2^\spec,3^\spec &   \emptyset   &   5^\spec \\
            4   &   4   &   5   &   \emptyset   &   1   &   2   \\
            5   &   5   &   4,5^\sspec  &   5^\spec &   2   &   1,2^\sspec,3^\spec \\
        \end{array} \]
        \caption{The fusion law \texorpdfstring{$(X_l^0,\fus)$}{} for type \texorpdfstring{$A_n$}{An}. \texorpdfstring{Entries marked with $\sspec$ should be left out for $n = 3$ and those marked with $\spec$ should be left out for $n \in \{3,4\}$.}{}}%
        \label{tbl:0declAfus}
    \end{minipage}
\end{enumerate}

\subsection{Type \texorpdfstring{$D_n$}{Dn}}

\begin{enumerate}
    \item The root system of type $D_n$ is described in \cref{ex:rs}~\labelcref{ex:rs:Dn}.
        Once again, we identify $\fJ$ with $S^2(\mH) = S^2(\langle\Phi\rangle)$ using the isomorphism from \cref{prop:S}.
        Local decompositions will be given with respect to the root $\alpha \coloneqq b_1 + b_2$, the highest root with respect to the base from \cref{ex:rs}~\labelcref{ex:rs:Dn}.
        We use the isomorphisms $W \cong W_n'$ and $C_W(s_\alpha) \cong W_2' \times W_{n-2}'$ to describe the characters of $W$ and $C_W(s_\alpha)$.
        For the global decomposition, we make a distinction between $n=4$ and $n>4$.
        For the local decomposition, we restrict to $n>6$.
        The given decomposition remain decomposition for $n \in \{4,5,6\}$ but the components are not isotypic.

    \item
        \begin{minipage}[t]{\linewidth}$n=4$
                \begin{align*}
                    J_a &\coloneqq \langle \sum_{i=1}^{4} b_i b_i \rangle \\
                    J_b &\coloneqq \langle b_i b_i - b_j b_j \mid 1 \leq i,j \leq 4 \rangle \\
                    J_c &\coloneqq \langle b_i b_j + b_k b_l \mid \{i,j,k,l\} = \{1,2,3,4\} \rangle \\
                    J_d &\coloneqq \langle b_i b_j - b_k b_l \mid \{i,j,k,l\} = \{1,2,3,4\} \rangle
                \end{align*}
        \end{minipage}

        \noindent\begin{minipage}[t]{\linewidth}$n>4$
                \begin{align*}
                    J_a &\coloneqq \langle \sum_{i=1}^{n+1} b_i b_i \rangle \\
                    J_b &\coloneqq \langle b_i b_i - b_j b_j \mid 1 \leq i < j \leq n \rangle \\
                    J_c &\coloneqq \langle b_i b_j \mid 1 \leq i < j \leq n \rangle
                \end{align*}
        \end{minipage}

    \item
        \begin{minipage}[t]{\linewidth}$n=4$
            \captionsetup{type=table}
            \[\begin{array}{ccc} \toprule
                \text{Component} & \text{Character} & \text{Dimension} \\ \midrule
                J_a & \chi_{([4],\emptyset)}' & 1 \\
                J_b & \chi_{([3,1],\emptyset)}' & 3 \\
                J_c & \chi_{([2],+)} & 3 \\
                J_d & \chi_{([2],-)} & 3 \\ \bottomrule
            \end{array}\]
            \caption{Characters and dimensions for the global decomposition of \texorpdfstring{$\fJ$}{J} for type \texorpdfstring{$D_4$}{D4}.}%
            \label{tbl:0decgD4char}
        \end{minipage}

        \noindent\begin{minipage}[t]{\linewidth}$n>4$
            \captionsetup{type=table}
            \[\begin{array}{ccc} \toprule
                \text{Component} & \text{Character} & \text{Dimension} \\ \midrule
                J_a & \chi_{([n],\emptyset)}' & 1 \\
                J_b & \chi_{([n-1,1],\emptyset)}' & n-1 \\
                J_c & \chi_{([n-2],[2])}' &  \frac{n(n-1)}{2} \\ \bottomrule
            \end{array}\]
            \caption{Characters and dimensions for the global decomposition of \texorpdfstring{$\fJ$}{J} for type \texorpdfstring{$D_n$ ($n>4$)}{Dn}.}%
            \label{tbl:0decgDchar}
        \end{minipage}

    \item
        \begin{minipage}[t]{\linewidth}$n=4$
            \captionsetup{type=table}
            \[\begin{array}{c|cccc} 
                \fus &  a   &   b   &   c & d\\ \hline 
                a   &   a   &   b   &   c & d\\
                b   &   b   &   a,b &   d & c\\
                c   &   c   &   d   &   a,c & b \\
                d   &   d   &   c   &   b   &   a,d
            \end{array}\]
            \caption{The fusion law \texorpdfstring{$(X_g^0,\fus)$}{} for type \texorpdfstring{$D_4$}{D4}.}%
            \label{tbl:0decgD4fus}
        \end{minipage}

        \noindent\begin{minipage}[t]{\linewidth}$n>4$
            \captionsetup{type=table}
            \[\begin{array}{c|ccc} 
                \fus &  a   &   b   &   c \\ \hline 
                a   &   a   &   b   &   c \\
                b   &   b   &   a,b &   c \\
                c   &   c   &   c   &   a,b,c \\
            \end{array}\]
            \caption{The fusion law \texorpdfstring{$(X_g^0,\fus)$}{} for type \texorpdfstring{$D_n$ ($n>4$)}{Dn}.}%
            \label{tbl:0decgDfus}
        \end{minipage}

    \item\begin{minipage}[t]{\linewidth}$n>6$
        \begin{align*}
            J_{a,1}^\alpha &\coloneqq \langle \sum_{i=1}^{n+1} b_i b_i \rangle \\
            J_{b,1}^\alpha &\coloneqq \langle n(b_1 b_1 + b_2 b_2) - 2\sum_{k=1}^{n} b_k b_k \rangle \\
            J_{b,2}^\alpha &\coloneqq \langle b_k b_k - b_l b_l \mid 3 \leq k,l \leq n \rangle \\
            J_{b,6}^\alpha &\coloneqq \langle b_1b_1 - b_2b_2 \rangle \\
            J_{c,1}^\alpha &\coloneqq \langle b_1b_2 \rangle \\
            J_{c,3}^\alpha &\coloneqq \langle b_k b_l \mid 3 \leq k < l \leq n \rangle \\
            J_{c,4}^\alpha &\coloneqq \langle (b_1 + b_2)b_k \mid 3 \leq k \leq n \rangle \\
            J_{c,5}^\alpha &\coloneqq \langle (b_1 - b_2)b_k \mid 3 \leq k \leq n \rangle
        \end{align*}
    \end{minipage}

    \item\begin{minipage}[t]{\linewidth}$n>6$
        \captionsetup{type=table}
        \[\begin{array}{ccc} \toprule
            \text{Component} & \text{Character} & \text{Dimension} \\ \midrule
            J_1 & 3 \cdot \chi_{([2],\emptyset)}' \times \chi_{([n-2],\emptyset)}' & 3 \cdot 1 \\
            J_2 & \chi_{([2],\emptyset)}' \times \chi_{([n-3,1],\emptyset)}' & n-3 \\
            J_3 & \chi_{([2],\emptyset)}' \times \chi_{([n-4],[2])}' & \frac{n^2-5n+6}{2} \\
            J_4 & \chi_{([1],+)} \times \chi_{([n-3],[1])}' & n-2 \\
            J_5 & \chi_{([1],-)} \times \chi_{([n-3],[1])}' & n-2 \\
            J_6 & \chi_{([1,1],\emptyset)}' \times \chi_{([n-2],\emptyset)}' &  1 \\ \bottomrule
        \end{array}\]
        \caption{Characters and dimensions for the local decomposition of \texorpdfstring{$\fJ$}{J} for type \texorpdfstring{$D_n$ ($n>6$)}{Dn}.}%
        \label{tbl:0declDchar}
    \end{minipage}

    \item\begin{minipage}[t]{\linewidth}$n>6$
        \captionsetup{type=table}
        \[\begin{array}{c|cccccc} 
            \fus &  1   &   2   &   3   &   4   &   5   &   6 \\ \hline 
            1   &   1   &   2   &   3   &   4   &   5   &   6  \\
            2   &   2   &   1,2 &   3   &   4   &   5   &      \\
            3   &   3   &   3   &   1,2,3   &   4   &   5   &  \\
            4   &   4   &   4   &   4   &   1,2,3   &   6   &   5 \\
            5   &   5   &   5   &   5   &   6   &   1,2,3   &   4 \\
            6   &   6   &   &   &   5   &   4   &   1
        \end{array}\]
        \caption{The fusion law \texorpdfstring{$(X_l^0,\fus)$}{} for type \texorpdfstring{$D_n$ ($n>6$)}{Dn}.}%
        \label{tbl:0declDfus}
    \end{minipage}
\end{enumerate}

\begin{remark}
    Observe that the local decomposition with respect to $\alpha = b_1 + b_2$ is the same as the one with respect to $b_1 - b_2$ up to the order of the terms.
    This is due to the fact that the centralizers of their corresponding reflections are equal.
    As a result, the local fusion law is $\Z/2\Z \times \Z/2\Z$\dash{}graded.
\end{remark}

\subsection{Type \texorpdfstring{$E_n$}{En}}

\begin{enumerate}
    \item We use the description and notation for the root systems of type $E_n$ from \cref{ex:rs}~\labelcref{ex:rs:En}.
        As usual, we identify $\fJ$ with $S^2(\mH)$.
        Local decompositions will be given with respect to the highest root corresponding to the base from \cref{ex:rs}~\labelcref{ex:rs:En}.
        Those are the roots $b_7 + b_8$, $b_7+b_8$ and $b_7 - b_8$ for $n=6,7,8$ respectively.
        The characters of the Weyl groups are given in the notation from~\cite{GP00}*{Table~C.4 to~C.6}.
        Also recall the Frobenius form $B_A$ for $\fJ$ from \cref{def:Jast}.

    \item\begin{minipage}[t]{\linewidth}
        \begin{align*}
            J_a &\coloneqq \langle \id \in \fJ \rangle \\
            J_b &\coloneqq \langle v \in \fJ \mid B_A(v,\id)=0 \rangle
        \end{align*}
    \end{minipage}

    \item
        \begin{minipage}[t]{\linewidth}$n=6$
            \captionsetup{type=table}
            \[\begin{array}{ccc} \toprule
                \text{Component} & \text{Character} & \text{Dimension} \\ \midrule
                J_a & 1_p & 1 \\
                J_b & 20_p & 20 \\ \bottomrule
            \end{array}\]
            \caption{Characters and dimensions for the global decomposition of \texorpdfstring{$\fJ$}{J} for type \texorpdfstring{$E_6$}{E6}.}%
            \label{tbl:0decgE6char}
        \end{minipage}

        \noindent\begin{minipage}[t]{\linewidth}$n=7$
            \captionsetup{type=table}
            \[\begin{array}{ccc} \toprule
                \text{Component} & \text{Character} & \text{Dimension} \\ \midrule
                J_a & 1_a & 1 \\
                J_b & 27_a & 27 \\ \bottomrule
            \end{array}\]
            \caption{Characters and dimensions for the global decomposition of \texorpdfstring{$\fJ$}{J} for type \texorpdfstring{$E_7$}{E7}.}%
            \label{tbl:0decgE7char}
        \end{minipage}

        \noindent\begin{minipage}[t]{\linewidth}$n=8$
            \captionsetup{type=table}
            \[\begin{array}{ccc} \toprule
                \text{Component} & \text{Character} & \text{Dimension} \\ \midrule
                J_a & 1_x & 1 \\
                J_b & 35_x & 35 \\ \bottomrule
            \end{array}\]
            \caption{Characters and dimensions for the global decomposition of \texorpdfstring{$\fJ$}{J} for type \texorpdfstring{$E_8$}{E8}.}%
            \label{tbl:0decgE8char}
        \end{minipage}

    \item\begin{minipage}[t]{\linewidth}
        \captionsetup{type=table}
        \[\begin{array}{c|cc} 
            \fus & a & b \\ \hline 
            a & a & b \\
            b & b & a,b
        \end{array}\]
        \caption{The fusion law \texorpdfstring{$(X_g^0,\fus)$}{} for type \texorpdfstring{$E_n$}{En}.}%
        \label{tbl:0decgEfus}
    \end{minipage}

    \item
        \begin{minipage}[t]{\linewidth}$n=6$
            \begin{align*}
                J_{a,1}^\alpha &\coloneqq \langle \sum_{i=1}^8 b_i b_i \rangle \\
                J_{b,1}^\alpha &\coloneqq \langle (b_7+b_8)(b_7+b_8) - \frac{1}{3}\sum_{i=1}^8 b_i b_i \rangle \\
                J_{b,2}^\alpha &\coloneqq \langle b_i'b_i'-b_j'b_j' \mid 1 \leq i < j \leq 6 \rangle \\
                J_{b,3}^\alpha &\coloneqq \langle (b_i-b_j)(b_k-b_l) \mid 1 \leq i,j,k,l \leq 6 , \{i,j\} \cap \{k,l\} = \emptyset \rangle \\
                J_{b,5}^\alpha &\coloneqq \langle (b_7+b_8)(b_i-b_j) \mid 1 \leq i < j \leq 6 \rangle
            \end{align*}
        \end{minipage}

        \noindent\begin{minipage}[t]{\linewidth}$n=7$
            \begin{align*}
                J_{a,1}^\alpha &\coloneqq \langle \id \in \fJ \rangle \\
                J_{b,1}^\alpha &\coloneqq \langle (b_7+b_8)(b_7+b_8) - \frac{2}{7}\sum_{i=1}^8 b_i b_i \rangle \\
                J_{b,2}^\alpha &\coloneqq \langle b_i b_i-b_j b_j \mid 1 \leq i < j \leq 6 \rangle \\
                J_{b,3}^\alpha &\coloneqq \langle b_i b_j \mid 1 \leq i < j \leq 6 \rangle \\
                J_{b,5}^\alpha &\coloneqq \langle (b_7+b_8)b_i \mid 1 \leq i \leq 6 \rangle
            \end{align*}
        \end{minipage}

        \noindent\begin{minipage}[t]{\linewidth}$n=8$
            \begin{align*}
                J_{a,1}^\alpha &\coloneqq \langle \id \in \fJ \rangle \\
                J_{b,1}^\alpha &\coloneqq \langle \alpha \alpha - \frac{1}{4}\sum_{i=1}^8 b_i b_i \rangle \\
                J_{b,3}^\alpha &\coloneqq \langle vw - \frac{\kappa(v,w)}{7}\sum_{i=1}^8 b_i b_i + \frac{\kappa(v,w)}{14} \alpha \alpha \mid v,w \in \alpha^\perp \rangle \\
                J_{b,5}^\alpha &\coloneqq \langle \alpha v \mid v \in \alpha^\perp \rangle
            \end{align*}
        \end{minipage}

    \item
        \begin{minipage}[t]{\linewidth}$n=6$
            \captionsetup{type=table}
            \[\begin{array}{ccc} \toprule
                \text{Component} & \text{Character} & \text{Dimension} \\ \midrule
                J_1^\alpha & 2 \cdot \chi_{[2]} \times \chi_{[6]} & 2 \cdot 1 \\
                J_2^\alpha & \chi_{[2]} \times \chi_{[5,1]} & 5 \\
                J_3^\alpha & \chi_{[2]} \times \chi_{[4,2]} & 9 \\
                J_5^\alpha & \chi_{[1,1]} \times \chi_{[5,1]} & 5 \\ \bottomrule
            \end{array}\]
            \caption{Characters and dimensions for the local decomposition of \texorpdfstring{$\fJ$}{J} for type \texorpdfstring{$E_6$}{E6}.}%
            \label{tbl:0declE6char}
        \end{minipage}

        \noindent\begin{minipage}[t]{\linewidth}$n=7$
            \captionsetup{type=table}
            \[\begin{array}{ccc} \toprule
                \text{Component} & \text{Character} & \text{Dimension} \\ \midrule
                J_1^\alpha & 2 \cdot \chi_{[2]} \times \chi_{([6],\emptyset)} & 2 \cdot 1 \\
                J_2^\alpha & \chi_{[2]} \times \chi_{([5,1],\emptyset)} & 5 \\
                J_3^\alpha & \chi_{[2]} \times \chi_{([4],[2])} & 15 \\
                J_5^\alpha & \chi_{[1,1]} \times \chi_{([5],[1])} & 6 \\ \bottomrule
            \end{array}\]
            \caption{Characters and dimensions for the local decomposition of \texorpdfstring{$\fJ$}{J} for type \texorpdfstring{$E_7$}{E7}.}%
            \label{tbl:0declE7char}
        \end{minipage}

        \noindent\begin{minipage}[t]{\linewidth}$n=8$
            \captionsetup{type=table}
            \[\begin{array}{ccc} \toprule
                \text{Component} & \text{Character} & \text{Dimension} \\ \midrule
                J_1^\alpha & 2 \cdot \chi_{[2]} \times 1_a & 2 \cdot 1 \\
                J_3^\alpha & \chi_{[2]} \times 27_a & 27 \\
                J_5^\alpha & \chi_{[1,1]} \times 7_a' & 7 \\ \bottomrule
            \end{array}\]
            \caption{Characters and dimensions for the local decomposition of \texorpdfstring{$\fJ$}{J} for type \texorpdfstring{$E_8$}{E8}.}%
            \label{tbl:0declE8char}
        \end{minipage}

    \item
        \begin{minipage}[t]{\linewidth}$n=6$
            \captionsetup{type=table}
            \[ \begin{array}{c|cccc} 
                \fus &  1   &   2   &   3   &   5 \\ \hline 
                1   &   1   &   2   &   3   &   5 \\
                2   &   2   &   1,2,3   &   2,3 &   5 \\
                3   &   3   &   2,3 &   1,2,3   &   5 \\
                5   &   5   &   5   &   5   &   1,2,3 \\
            \end{array} \]
            \caption{The fusion law \texorpdfstring{$(X_l^0,\fus)$}{} for type \texorpdfstring{$E_6$}{E6}.}%
            \label{tbl:0declE6fus}
        \end{minipage}

        \noindent\begin{minipage}[t]{\linewidth}$n=7$
            \captionsetup{type=table}
            \[ \begin{array}{c|cccc} 
                \fus &  1   &   2   &   3   &   5 \\ \hline 
                1   &   1   &   2   &   3   &   5 \\
                2   &   2   &   1,2 &   3   &   5 \\
                3   &   3   &   3   &   1,2,3   &   5 \\
                5   &   5   &   5   &   5   &   1,2,3 \\
            \end{array} \]
            \caption{The fusion law \texorpdfstring{$(X_l^0,\fus)$}{} for type \texorpdfstring{$E_7$}{E7}.}%
            \label{tbl:0declE7fus}
        \end{minipage}

        \noindent\begin{minipage}[t]{\linewidth}$n=8$
            \captionsetup{type=table}
            \[ \begin{array}{c|ccc} 
                \fus & 1 & 3 & 5 \\ \hline 
                1 & 1 & 3 & 5 \\
                5 & 3 & 1,3 & 5 \\
                5 & 5 & 5 & 1,3
            \end{array} \]
            \caption{The fusion law \texorpdfstring{$(X_l^0,\fus)$}{} for type \texorpdfstring{$E_8$}{E8}.}%
            \label{tbl:0declE8fus}
        \end{minipage}
\end{enumerate}

\section{A decomposition of the algebra}%
\label{sec:dec}

In this section we will give the whole algebra $(\fA,\ast)$ the structure of a decomposition algebra.  Once again, we apply the techniques from \cref{def:gldec,lem:gldec}.
This time, we will use the fact that $(\fA,\ast)$ is an algebra for $\mL$.
As for the decompositions of the zero weight subalgebra, we will first illustrate the general procedure and give the results for each of the possible types afterwards.

\subsection{The general procedure}

In order to use \cref{def:gldec}, we look for a class of conjugate subalgebras of $\mL$ to obtain local decompositions.
Recall the notation for a Chevalley basis from \cref{def:stage}.
Since we used the reflection subgroups $C_W(s_\alpha) = N_W(\langle s_\alpha \rangle)$ to obtain local decompositions of $\fJ$ in \cref{sec:0dec}, a natural candidate are the subalgebras of the form $N_\mL(\langle h_\alpha,e_\alpha,e_{-\alpha}\rangle)$.

\begin{definition}%
    \label{def:idx}
    Let $\idx$ be the class of subalgebras of $\mL$ conjugate to the subalgebra $\langle h_\alpha,e_{\alpha},e_{-\alpha} \rangle$ for some $\alpha \in \Phi$.
    Note that $\idxe \cong \mathfrak{sl}_2(\C)$ for each $\idxe \in \idx$.
\end{definition}

\begin{proposition}%
    \label{prop:reductive}
    Let $\alpha \in \Phi$.
    Consider the subalgebra $\idxe = \llangle e_{\alpha},e_{-\alpha} \rrangle$ of $\mL$.
    Then
    \[
        N_\mL(\idxe) = H \oplus \langle e_\beta \mid \beta \in \{\pm \alpha \} \cup \{\beta \in \Phi \mid \kappa(\alpha,\beta)=0 \} \rangle.
    \]
    In particular $N_\mL(\idxe)$ is reductive for each $\idxe \in \idx$.
\end{proposition}
\begin{proof}
    Let $\mH$ be the Cartan subalgebra from \cref{def:stage}.
    Clearly we have $\mH \leq N_\mL(\idxe)$.
    Thus $\mH$ normalizes $N_\mL(\idxe)$.
    As a result, the subalgebra $N_\mL(\idxe)$ must be a direct sum of common eigenspaces of the adjoint action of $\mH$. This means that $N_\mL(\idxe)$ is of the form
    \[
        \mH \oplus \langle e_\beta \mid \beta \in S \rangle
    \]
    for some $S \subseteq \Phi$.
    The first assertion follows because $e_\beta \in N_\mL(\idxe)$ for $\beta \in \Phi$ if and only if $\beta \in \{\pm \alpha \} \cup \{\beta \in \Phi \mid \kappa(\alpha,\beta)=0 \}$.
    Now $N_\mL(\idxe)$ is reductive by~\cite{Bou05}*{\S~VIII.3 Proposition~2}.
    Since all elements $\idxe \in \idx$ are conjugate, the same is true for the subalgebras $N_\mL(\idxe)$.
\end{proof}

\begin{definition}%
    \label{def:L_a}
    Let $\idx$ be as in \cref{def:idx}.
    For each $\idxe \in \idx$, let \[\mL_\idxe \coloneqq [N_\mL(\idxe),N_\mL(\idxe)].\]
    By \cref{prop:reductive} and~\cite{Bou89}*{Chapter~I, \S~6.3, Proposition~5} the subalgebra $\mL_\idxe$ is semisimple.
    For example, we have \[\mL_{\llangle e_\alpha e_{-\alpha}\rrangle} = \llangle e_\beta \mid \beta \in \{\pm \alpha \} \cup \{\beta \in \Phi \mid \kappa(\alpha,\beta)=0 \} \rrangle.\]
    Since all elements $\idxe \in \idx$ are conjugate, so are all $\mL_\idxe$ for $\idxe \in \idx$.
\end{definition}

\begin{remark}
    The type of $\mL_{\llangle e_\alpha , e_{-\alpha} \rrangle}$ is given by \cref{rem:type}.
    Note that the Weyl group of $\mL_{\llangle e_\alpha , e_{-\alpha} \rrangle}$ is precisely $C_W(s_\alpha)$ by \cref{prop:refl}.
\end{remark}

\begin{definition}%
    \label{def:dec}
    \begin{enumerate}
        \item Let $\bigoplus_{x \in X_g} A_x$ be the global decomposition of $\fA$ with respect to $\mL$. Denote its fusion law by $(X_g,\fus)$. Let $\bigoplus_{x \in X_l} A_x^\idxe$ be the local decomposition of $\fA$ with respect to $(\mL_\idxe \mid \idxe \in \idx)$. Let $(X_l,\fus)$ the corresponding fusion law.
        \item For each $x \in X_g$ and $y \in X_l$, let $A_{x,y}^\idxe \coloneqq A_x \cap A_y^\idxe$ as in \cref{lem:gldec}. Let $\mF$ be the direct product of the fusion laws $(X_g,\fus)$ and $(X_l,\fus)$. Then we know by \cref{lem:gldec} that $\bigoplus_{x,y} A_{x,y}^\idxe$ is an $\mF$-decomposition of $\fA$. Let $\Omega$ be the $\idx$-tuple of these decompositions. Then $(\fA,\idx,\Omega)$ is an $\mF$-decomposition algebra.
    \end{enumerate}
\end{definition}

As in \cref{sec:0dec} we will define a $\Z/2\Z$-grading of the fusion law $\mF$ and determine the corresponding Miyamoto group of the $\mF$-decomposition algebra $(\fA,\idx,\Omega)$.
In \cref{sec:0dec} we obtained the $\Z/2\Z$-grading from \cref{lem:grad0} implicitly by restriction to the central subgroup $\langle s_\alpha \rangle \leq C_W(\langle s_\alpha \rangle)$.
Similarly, we will obtain a $\Z/2\Z$-grading by restricting to $\idxe \leq \mL_\idxe$.

\begin{definition}%
    \label{def:grad}
    \begin{enumerate}
        \item Let $\idxe \in \idx$.
            Recall that $\idxe \cong \mathfrak{sl}_2(\C)$.
            Let $h \in \idxe$ be one of the two coroots with respect to some Cartan subalgebra of $\idxe$.
            Write $V$ for the standard representation of $\idxe$.
            Then the eigenvalues for the action of the element $h$ on $V$ are $1$ and $-1$.
            Therefore the eigenvalues of the adjoint action of $h$ on $V^{\otimes n}$ are odd (respectively even) integers if $n$ is odd (respectively even).
            Since any irreducible representation of $\idxe$ is some subrepresentation of $V^{\otimes n}$ for some $n \in \mathbb{N}$, this divides the irreducible representations into two parts.
            The irreducible representations of $\idxe$ for which the eigenvalues of the action of $h$ are odd (respectively even), are called odd (respectively even) representations.
            Also note that the tensor product of two odd (or two even) representations is a direct sum of even representations and that the tensor product of an odd and an even representation is a direct sum of odd representations.
            See also~\cite{FH91}*{\S 11.8 p. 150}.
        \item Obviously, $\idxe$ is an ideal of $\mL_\idxe$.
            Therefore, the $\mL_\idxe$-representations $A_x^\idxe$ for $x \in X_l$, restricted to $\idxe$, are isomorphic to $n W_x$ for some irreducible representation $W_x$ of $\idxe$ and some $n \in \mathbb{N}$.
            Now define
            \[
                \xi : X_l \to \Z/2\Z : x \mapsto \begin{cases}
                    0 & \text{if $W_x$ is even,} \\
                    1 & \text{if $W_x$ is odd.}
                \end{cases}
            \]
    \end{enumerate}
\end{definition}

\begin{lemma}%
    \label{lem:grad}
    The map $\xi$ induces a non-trivial $\Z/2\Z$-grading of the fusion law $(X_l,\fus)$.
\end{lemma}
\begin{proof}
    The tensor product of two odd (or two even) representations is a direct sum of even representations and the tensor product of an odd and an even representation is a direct sum of odd representations.
    Since $(X_l,\fus)$ is the representation fusion law on $X_l$, it follows that $\xi$ defines a grading of $(X_l,\fus)$.
    To prove that this grading is non-trivial it suffices to show that the $\fA$ has an odd irreducible $\llangle e_{\alpha},e_{-\alpha} \rrangle$-subrepresentation.
    Equivalently, we need to show that one of the coroots of $\llangle e_\alpha ,e_{-\alpha} \rrangle$, e.g.\ $h_\alpha$, has an eigenvector in $\fA$ with an odd eigenvalue.
    Since e.g.\ $h_\alpha \cdot {[h]}_\beta = {[h]}_\beta$ for any roots $\alpha,\beta \in \Phi$ with $\kappa(\alpha,\beta)=1$ and any $h \in \mH$, it follows that the grading is non-trivial.
\end{proof}

This grading induces a non-trivial grading of $\mF$.

\begin{definition}%
    \label{def:gradF}
    The $\Z/2\Z$-grading $\xi$ of $(X_l,\fus)$ induces a $\Z/2\Z$-grading of $\mF$:
    \[
        \xi : X_g \times X_l \to \Z/2\Z : (x,y) \mapsto \xi(y).
    \]
\end{definition}

Now we determine the corresponding Miyamoto group of the $\mF$-decomposition algebra $(A(\Phi),\idx,\Omega)$.
It turns out that this Miyamoto group is isomorphic to the group of inner automorphism (this is, the adjoint Chevalley group) of $\mL$.
We repeat some terminology about inner automorphisms.
We refer to~\cite{Stei68} for more details.

\begin{definition}%
    \label{defn:int}
    Let $\mathfrak{L}$ be an arbitrary complex, semisimple Lie algebra. Let $\rho : \mathfrak{L} \to \mathfrak{gl}(V)$ be a representation of $\mathfrak{L}$. Suppose that $\ell \in \mathfrak{L}$ such that $\rho(\ell)$ is nilpotent which means that ${\rho(\ell)}^k = 0$ for some $k \in \mathbb{N} \setminus \{0\}$. Let
    \[
        \exp ( \rho(\ell)) \coloneqq \sum_{i = 0}^{k-1} \frac{{\rho(\ell)}^i}{i!}.
    \]
    Then $\exp( \rho(\ell))$ acts as an automorphism by conjugation on the Lie algebra $\rho(\mathfrak{L})$. We call the subgroup of $\GL(V)$ generated by these automorphisms for all possible choices of $\ell$ the Chevalley group of $(\mathfrak{L},\rho)$ and denote it by $\Int(\mathfrak{L},\rho)$. The isomorphism class of $\Int(\mathfrak{L},\rho)$ only depends on the lattice $\Lambda(\rho)$ spanned by the weights of~$\rho$. If $\Lambda(\rho)$ is equal to the weight lattice of $\mathfrak{L}$, then we call $\Int(\mathfrak{L},\rho)$ the fundamental Chevalley group of $\mathfrak{L}$ and denote it by $\HInt\mathfrak{L}$. For any representation $\rho$ of $\mathfrak{L}$ there exists an epimorphism $\HInt\mathfrak{L} \to \Int(\mathfrak{L},\rho)$ such that the kernel is contained in the center of $\HInt\mathfrak{L}$. Any representation $\rho$ of $\mathfrak{L}$ can therefore be viewed as a representation for $\HInt(\mL)$. On the other hand, if $\Lambda(\rho)$ is equal to the root lattice of $\mathfrak{L}$, then $\Int \mathfrak{L} \coloneqq \Int(\mathfrak{L},\rho)$ is called the adjoint Chevalley group of $\mathfrak{L}$. For any representation $\rho$ of $\mathfrak{L}$, there exist an epimorphism $\Int(\mathfrak{L},\rho) \to \Int \mathfrak{L}$ with kernel contained in the center of $\Int(\mathfrak{L},\rho)$. If $\mathfrak{L}$ is simple, then so is $\Int \mathfrak{L}$.
\end{definition}

The following example explains how the grading coming from odd and even representations of $\mathfrak{sl}_2(\C)$ gives rise to involutions.

\begin{example}%
    \label{ex:sl2}
    Consider a Lie algebra $\idxe \cong \sll_2(\C)$ together with its standard 2-dimensional representation $\rho$.
    Then $\HInt\idxe = \Int(\idxe,\rho) \cong \SL_2(\C)$ while $\Int \idxe \cong \PSL_2(\C)$.
    Denote the unique non-trivial element in the center of $\HInt \idxe$ by $\sigma_\idxe$.
    Then $\sigma_\idxe$ acts trivially on the representation $\rho$ of $\idxe$ (viewed as a representation for $\HInt \idxe$) if and only if the weight lattice $\Lambda(\rho)$ is equal to the root lattice of $\idxe$.
    More precisely, $\sigma_\idxe$ acts as 1 (respectively $-1$) on the even (respectively odd) representations of $\idxe$.
\end{example}

Now we are ready to determine the Miyamoto group of $(\fA,\idx,\Omega)$.

\begin{theorem}%
    \label{thm:Amiy}
    Let $\Phi$ be an irreducible simply laced root system.
    Consider the $\mF$-decomposition algebra $(\fA,\idx,\Omega)$ from \cref{def:dec}.
    The Miyamoto group of this algebra corresponding to the $\Z/2\Z$-grading of $\mF$ from \cref{def:gradF} is $\Int(\mL,\fA)$, the adjoint Chevalley group of type $\Phi$.
\end{theorem}
\begin{proof}
    From the definition of the $\Z/2\Z$-grading of $\mF$ (\cref{def:grad,def:gradF}) and \cref{ex:sl2}, it follows that the action of $\tau_{\idxe,\chi}$ (with $\chi$ the non-trivial character of $\Z/2\Z$) corresponds to the action of $\sigma_\idxe$.
    This action is non-trivial by \cref{lem:grad}.
    Since the index set $\idx$ is closed under the action of $\Int \mL$, the elements $\{ \sigma_\idxe \mid \idxe \in \idx \}$ form a conjugacy class of involutions of $\Int(\mL,\fA)$.
    Since the weights of $\fA$ are contained in the root lattice $\mL$, we have that $\Int(\mL,\fA)$ is isomorphic to the adjoint Chevalley group of type $\Phi$ and therefore simple.
    So the group generated by the Miyamoto maps must be isomorphic to it.
\end{proof}

\begin{remark}
    Note that we have never used any information about the algebra product $\ast$ on $\fA$.
    Indeed, the technique that we used here is applicable to any algebra on which the Lie algebra $\mL$ acts (non-trivially) by derivations, for example the Lie algebra itself.
    It will be possible to give the algebra the structure of a decomposition algebra with a $\Z/2\Z$-graded fusion law.
    If this grading is non-trivial then the corresponding Miyamoto group will be a Chevalley group of type $\Phi$ (but not necessarily adjoint).
\end{remark}

Let us now give an overview of some of the techniques that we used to explicitly obtain the local and global decompositions.

In \cref{sec:0dec}, we described the decompositions of the zero weight space of $\fA$.
We can use the results from~\cite{Bro95}*{Corollary 1} and~\cite{Ree98} to extend these decompositions of $\fJ$ to decompositions in $\fA$.
They introduce the terminology of \emph{small modules} which means that the double of a root is not a weight of the module.
More precisely, they prove that, if $V$ is a small module for a semisimple Lie algebra $\mL$, then its zero weight space is (almost always) irreducible as a representation for the Weyl group of $\mL$.
Note that $\fA$ is small as a module for $\mL$ or $\mL_\idxe$.
If $V$ is now an irreducible subrepresentation of $\fJ$ for $W$ (resp. $C_W(s_\alpha)$), then it follows from these results that the $\mL$\dash{}module (resp.\ $\mL_{\llangle e_{\alpha},e_{-\alpha} \rrangle}$\dash{}module) generated by $V$ is also irreducible.
Moreover, we can determine the highest weight of this module from the character of $V$.
This already helps to get a lot of components of the global and local decompositions of $\fA$.

The representation fusion laws for $\mL$ and $\mL_\idxe$ can be determined using the results from~\cite{FH91}*{\S~25.3}.

For each of the types $A_n$, $D_n$ and $E_n$, we will continue to use the notation introduced in the corresponding subsection of \cref{sec:0dec}.
In particular, we recall the index sets $X_g^0$ and $X_l^0$ for the local and global decomposition.
The global decomposition can then be given as follows.
For each $x \in X_g^0$ we let $A_x$ be the $\mL$\dash{}submodule generated by $J_x$.
From the discussion above, it follows that each $A_x$ is an isotypic component of the $\mL$\dash{}modules.
In fact, these are all the isotypic components of $\fA$ as $\mL$\dash{}module.
Therefore we can take $X_g^0 = X_g$.
We will give the following additional information about the decompositions of $\fA$.

\begin{enumerate}
    \item For each isotypic component $A_x$ for $x \in X_g^0$, we will give its highest weight and dimension.
        We say that $A_x$ has highest weight $n \cdot w$ if $A_x$ is the isotypic component corresponding to the dominant weight $w$ and the weight $w$ has multiplicity $n$ in $A_x$.
        The weight $w$ is given with respect to the basis of fundamental weights.
        We have ordered this basis with respect to the numbering of the nodes of the Dynkin diagram from \cref{fig:dynkindiagrams}.
    \item We give the global fusion law $(X_g,\fus)$.
    \item The full decomposition $\bigoplus_{(x,i) \in \mF} A_{x,i}^\idxe$ with respect to $\idxe \coloneqq \llangle e_\alpha , e_{-\alpha} \rrangle$ is given, where $\alpha$ is the highest root as in \cref{sec:0dec}.
        For $(x,i) \in X_g^0 \times X_l^0$ we let $A_{(x,i)}^\idxe$ be the $\mL_\idxe$\dash{}submodule generated by $J_{(x,i)}^\alpha$.
        We extend $X_l^0$ to $X_l$ and give $A_{(x,i)}^\idxe$ for each $x \in X_g$ and $i \in X_l \setminus X_l^0$ for which $A_{(x,i)}^\idxe \neq 0$.
    \item We give the highest weight and dimension of each of the components of the local decomposition $\bigoplus_{i \in X_l} A_i^\idxe$.
    \item Lastly, the local fusion law $(X_l,\fus)$ is given.
\end{enumerate}

\begin{figure}
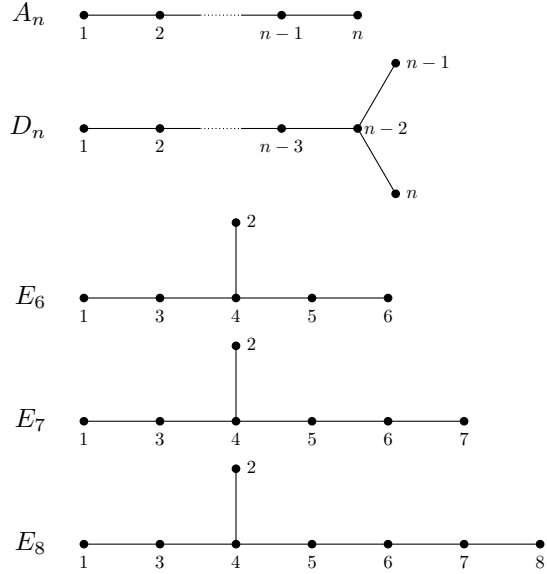

    \begin{tabular}{rl}
        $A_n$ & \dynkin[labels={1,2,n-1,n}]A{} \\
        $D_n$ & \begin{dynkinDiagram}[labels={1,2,n-3,,n-1,n}]D{} \node[anchor=west,scale=0.7] at (root 4) {$n-2$}; \end{dynkinDiagram} \\
        $E_6$ & \dynkin[label]E{6} \\
        $E_7$ & \dynkin[label]E{7} \\
        $E_8$ & \dynkin[label]E{8} \\
    \end{tabular}
    \caption{Dynkin diagrams of the irreducible simply laced root systems.}%
    \label{fig:dynkindiagrams}
\end{figure}

\subsection{Type \texorpdfstring{$A_n$}{An}}

We restrict to the case where $n>3$ for the global decomposition and fusion law and to $n>5$ for the local decomposition and fusion law.

\begin{enumerate}
    \item\begin{minipage}[t]{\linewidth}
        \captionsetup{type=table}
        \[\begin{array}{cccc} \toprule
            \text{Component} & \text{Highest weight} & \text{Dimension} \\ \midrule
            A_a & (0,0,\dots,0) & 1 \\
            A_b & (1,0,\dots,0,1) & n(n+2) \\
            A_c & (0,1,0,\dots,0,1,0) & \frac{(n+2){(n+1)}^2(n-2)}{4} \\ \bottomrule
        \end{array}\]
        \caption{Highest weights and dimensions for the global decomposition of \texorpdfstring{$\fA$}{A} for type \texorpdfstring{$A_n$ ($n>3$)}{An}.}%
        \label{tbl:decgAchar}
    \end{minipage}

    \item\begin{minipage}[t]{\linewidth}
        \captionsetup{type=table}
        \[\begin{array}{c|ccc} 
            \fus &  a   &   b   &   c \\ \hline 
            a   &   a   &   b   &   c \\
            b   &   b   &   a,b,c   & b,c   \\
            c   &   c   &   b,c &   a,b,c \\
        \end{array}\]
        \caption{The fusion law \texorpdfstring{$(X_g,\fus)$}{} for type \texorpdfstring{$A_n$ ($n>3$)}{An}.}%
        \label{tbl:decgAfus}
    \end{minipage}
    
    \item\begin{minipage}[t]{\linewidth}
        \begin{align*}
            A_{b,8}^\idxe &\coloneqq \langle {[b_1'+b_i']}_{b_1-b_i},{[b_{n+1}'+b_i']}_{b_{n+1}-b_i} \mid 1 \leq i \leq n-1 \rangle \\
            A_{b,9}^\idxe &\coloneqq \langle {[b_1'+b_i']}_{b_i-b_1},{[b_{n+1}'+b_i']}_{b_i-b_{n+1}} \mid 1 \leq i \leq n-1 \rangle \\
            A_{c,6}^\idxe &\coloneqq \langle x_{b_1+b_{n+1}-b_i-b_j} \mid 1 \leq i < j \leq n-1 \rangle \\
            A_{c,7}^\idxe &\coloneqq \langle x_{b_i+b_j-b_1+b_{n+1}} \mid 1 \leq i < j \leq n-1 \rangle \\
            A_{c,8}^\idxe &\coloneqq \langle {[b_1']}_{b_{n+1}-b_i},{[b_{n+1}']}_{b_1-b_i} \mid 1 \leq i \leq n-1\} \rangle \\
            A_{c,9}^\idxe &\coloneqq \langle {[b_1']}_{b_i - b_{n+1}},{[b_{n+1}']}_{b_i - b_1} \mid 1 \leq i \leq n-1\} \rangle\\
            A_{c,10}^\idxe &\coloneqq \begin{multlined}[t]
                \langle {[b_i']}_{b_1-b_j},{[b_i']}_{b_{n+1}-b_j},x_{b_1+b_i-b_j-b_k},x_{b_{n+1}+b_i-b_j-b_k} \mid \\ 1 \leq i,j,k \leq n-1 , \abs{\{i,j,k\}}=3 \rangle
                \end{multlined} \\
            A_{c,11}^\idxe &\coloneqq \begin{multlined}[t]
                \langle {[b_i']}_{b_j-b_1},{[b_i']}_{b_j-b_{n+1}},x_{b_j+b_k-b_1-b_i},x_{b_j+b_k-b_{n+1}-b_i} \mid \\ 1 \leq i,j,k \leq n-1 , \abs{\{i,j,k\}}=3 \rangle
                \end{multlined}
        \end{align*}
    \end{minipage}

    \item\begin{minipage}[t]{\linewidth}
        \captionsetup{type=table}
        \[\begin{array}{cccc} \toprule
            \text{Component} & \text{Highest weight} & \text{Dimension} \\ \midrule
            A_1^\idxe & 3 \cdot (0;0,\dots,0) & 3 \cdot 1\\
            A_2^\idxe & 2 \cdot (0;1,0,\dots,0,1) & 2 \cdot n(n-2) \\
            A_3^\idxe & (0;0,1,0,\dots,0,1,0) & \frac{n{(n-1)}^2(n-4)}{2} \\
            A_4^\idxe & (2;0,\dots,0) & 3 \\
            A_5^\idxe & (2;1,0,\dots,0,1) & 3n(n-2) \\
            A_6^\idxe & (0;0,\dots,0,1,0) & \frac{(n-1)(n-2)}{2}\\
            A_7^\idxe & (0;0,1,0,\dots,0) & \frac{(n-1)(n-2)}{2} \\
            A_8^\idxe & 2 \cdot (1;0,\dots,0,1) & 2 \cdot 2(n-1) \\
            A_9^\idxe & 2 \cdot (1;1,0,\dots,0) & 2 \cdot 2(n-1) \\
            A_{10}^\idxe & (1;1,0,\dots,0,1,0) & n(n-1)(n-3) \\
            A_{11}^\idxe & (1;0,1,0,\dots,0,1) & n(n-1)(n-3)\\ \bottomrule
        \end{array}\]
        \caption{Highest weights and dimensions for the local decomposition of \texorpdfstring{$\fA$}{A} for type \texorpdfstring{$A_n$ ($n>5$)}{An}.}%
        \label{tbl:declAchar}
    \end{minipage}

    \item\begin{minipage}[t]{\linewidth}
        \captionsetup{type=table}
        \[\resizebox{\linewidth}{!}{$\displaystyle\begin{array}{c|ccccccccccc} 
            \fus & 1 & 2 & 3 & 4 & 5 & 6 & 7 & 8 & 9 & 10 & 11\\ \hline 
            1 & 1 & 2 & 3 & 4 & 5 & 6 & 7 & 8 & 9 & 10 & 11 \\
            2 & 2 & 1,2,3 & 2,3 & 5 & 4,5 & 6 & 7 & 8,10 & 9,11 & 8,10 & 9,11 \\
            3 & 3 & 2,3 & 1,2,3 &  & 5 & 6 & 7 & 10 & 11 & 8,10 & 9,11 \\
            4 & 4 & 5 &  & 1,4 & 2,5 &  &  & 8 & 9 & 10 & 11 \\
            5 & 5 & 4,5 & 5 & 2,5 & 1,2,3,4,5 &  &  & 8,10 & 9,11 & 8,10 & 9,11 \\
            6 & 6 & 6 & 6 &  &  & 7^\sspec & 1,2,3 &  & 8,10 &  & 8,10 \\
            7 & 7 & 7 & 7 &  &  & 1,2,3 & 6^\sspec & 9,11 &  & 9,11 &  \\
            8 & 8 & 8,10 & 10 & 8 & 8,10 &  & 9,11 & 6 & 1,2,4,5 & 6 & 2,3,5 \\
            9 & 9 & 9,11 & 11 & 9 & 9,11 & 8,10 &  & 1,2,4,5 & 7 & 2,3,5 & 7 \\
            10 & 10 & 8,10 & 8,10 & 10 & 8,10 &  & 9,11 & 6 & 2,3,5 & 6 & 1,2,3,4,5 \\
            11 & 11 & 9,11 & 9,11 & 11 & 9,11 & 8,10 &  & 2,3,5 & 7 & 1,2,3,4,5 & 7 \\
        \end{array}$}\]
        \caption{The fusion law \texorpdfstring{$(X_l,\fus)$}{} for type \texorpdfstring{$A_n$ ($n>5$)}{An}. \texorpdfstring{Entries marked with $\sspec$ should be left out for $n \neq 7$.}{}}%
        \label{tbl:declAfus}
    \end{minipage}
\end{enumerate}

\subsection{Type \texorpdfstring{$D_n$}{Dn}}

We restrict to the case where $n>5$ for the global decomposition and fusion law and to $n>7$ for the local decomposition and fusion law.

\begin{enumerate}
    \item\begin{minipage}[t]{\linewidth}
        \captionsetup{type=table}
        \[\begin{array}{cccc} \toprule
            \text{Component} & \text{Highest weight} & \text{Dimension} \\ \midrule
            A_a & (0,0,\dots,0) & 1 \\ 
            A_b & (2,0,\dots,0,0) & (2n-1)(n+1) \\
            A_c & (0,0,0,1,0,\dots,0) &  \frac{(2n-3)(2n-1)(n-1)n}{6}\\ \bottomrule
        \end{array}\]
        \caption{Highest weights and dimensions for the global decomposition of \texorpdfstring{$\fA$}{A} for type \texorpdfstring{$D_n$ ($n>5$)}{Dn}.}%
        \label{tbl:decgDchar}
    \end{minipage}

    \item\begin{minipage}[t]{\linewidth}
        \captionsetup{type=table}
        \[\begin{array}{c|ccc} 
            \fus &	a	&	b	&	c \\ \hline 
            a	&	a	&	b	&	c \\
            b	&	b	&	a,b	&  c	\\
            c	&	c	&	c	&	a,b,c \\
        \end{array}\]
        \caption{The fusion law \texorpdfstring{$(X_g,\fus)$}{} for type \texorpdfstring{$D_n$ ($n>5$)}{Dn}.}%
        \label{tbl:decgDfus}
    \end{minipage}

    \item\begin{minipage}[t]{\linewidth}
        \begin{align*}
            A_{b,7}^\idxe &\coloneqq \langle {[b_1]}_{\pm b_1 \pm b_i}, {[b_2]}_{\pm b_2 \pm b_i} \mid 3 \leq i \leq n \rangle \\
            A_{c,7}^\idxe &\coloneqq \langle {[b_2]}_{\pm b_1 \pm b_i}, {[b_1]}_{\pm b_2 \pm b_i} \mid 3 \leq i \leq n \rangle \\
            A_{c,8}^\idxe &\coloneqq \begin{multlined}[t]
                \langle {[b_k]}_{\pm b_1 \pm b_i}, {[b_k]}_{\pm b_2 \pm b_i} , x_{\pm b_1 \pm b_i \pm b_k \pm b_l}, x_{\pm b_2 \pm b_i \pm b_k \pm b_l} \mid \\
                3 \leq i,j,k \leq n , \abs{\{i,j,k\}} = 3 \rangle
                \end{multlined}
        \end{align*}
    \end{minipage}
    
    \item\begin{minipage}[t]{\linewidth}
        \captionsetup{type=table}
        \[\begin{array}{cccc} \toprule
            \text{Component} & \text{Highest weight} & \text{Dimension} \\ \midrule
            A_1^\idxe & 3 \cdot (0;0;0,\dots,0) & 3 \cdot 1\\
            A_2^\idxe & (0;0;2,0,\dots,0) & (2n-5)(n-1) \\
            A_3^\idxe & (0;0;0,0,0,1,0,\dots,0) & \frac{(2n-7)(2n-5)(n-3)(n-2)}{6} \\
            A_4^\idxe & (2;0;0,1,0,\dots,0) & 3(2n-5)(n-2) \\
            A_5^\idxe & (0;2;0,1,0,\dots,0) & 3(2n-5)(n-2) \\
            A_6^\idxe & (2;2;0,\dots,0) & 9 \\
            A_7^\idxe & 2 \cdot (1;1;1,0,\dots,0) & 2 \cdot 8(n-2) \\
            A_8^\idxe & (1;1;0,0,1,0,\dots,0) & \frac{8(2n-5)(n-3)(n-2)}{3} \\ \bottomrule
        \end{array}\]
        \caption{Highest weights and dimensions for the local decomposition of \texorpdfstring{$\fA$}{A} for type \texorpdfstring{$D_n$ ($n>7$)}{Dn}.}%
        \label{tbl:declDchar}
    \end{minipage}

    \item\begin{minipage}[t]{\linewidth}
        \captionsetup{type=table}
        \[\begin{array}{c|cccccccc} 
            \fus & 1 & 2 & 3 & 4 & 5 & 6 & 7 & 8 \\ \hline 
            1 & 1 & 2 & 3 & 4 & 5 & 6 & 7 & 8 \\
            2 & 2 & 1,2 & 3 & 4 & 5 &  & 7 & 8 \\
            3 & 3 & 3 & 1,2,3 & 4 & 5 &  & 8 & 7,8 \\
            4 & 4 & 4 & 4 & 1,2,3,4 & 6 & 5 & 7,8 & 7,8 \\
            5 & 5 & 5 & 5 & 6 & 1,2,3,5 & 4 & 7,8 & 7,8 \\
            6 & 6 &  &  & 5 & 4 & 1,6 & 7 & 8 \\
            7 & 7 & 7 & 8 & 7,8 & 7,8 & 7 & 1,2,4,5,6 & 3,4,5 \\
            8 & 8 & 8 & 7,8 & 7,8 & 7,8 & 8 & 3,4,5 & 1,2,3,4,5,6 \\
        \end{array}\]
        \caption{The fusion law \texorpdfstring{$(X_l,\fus)$}{} for type \texorpdfstring{$D_n$ ($n>7$)}{Dn}.}%
        \label{tbl:declDfus}
    \end{minipage}
\end{enumerate}

\subsection{Type \texorpdfstring{$E_n$}{En}}%
\label{subsec:decEn}

\begin{enumerate}
    \item
        \begin{minipage}[t]{\linewidth}$n=6$
            \captionsetup{type=table}
            \[\begin{array}{ccc} \toprule
                \text{Component} & \text{Highest weight}     & \text{Dimension} \\ \midrule
                A_a & (0,0,0,0,0,0) & 1 \\
                A_b & (1,0,0,0,0,1) & 650 \\ \bottomrule
            \end{array}\]
            \caption{Highest weights and dimensions for the global decomposition of \texorpdfstring{$\fA$}{A} for type \texorpdfstring{$E_6$}{E6}.}%
            \label{tbl:decgE6char}
        \end{minipage}

        \noindent\begin{minipage}[t]{\linewidth}$n=7$
            \captionsetup{type=table}
            \[\begin{array}{ccc} \toprule
                \text{Component} & \text{Highest weight}     & \text{Dimension} \\ \midrule
                A_a & (0,0,0,0,0,0,0) & 1 \\
                A_b & (0,0,0,0,0,1,0) & 1539 \\ \bottomrule
            \end{array}\]
            \caption{Highest weights and dimensions for the global decomposition of \texorpdfstring{$\fA$}{A} for type \texorpdfstring{$E_7$}{E7}.}%
            \label{tbl:decgE7char}
        \end{minipage}

        \noindent\begin{minipage}[t]{\linewidth}$n=8$
            \captionsetup{type=table}
            \[\begin{array}{ccc} \toprule
                \text{Component} & \text{Highest weight}     & \text{Dimension} \\ \midrule
                A_a & (0,0,0,0,0,0,0,0) & 1 \\
                A_b & (1,0,0,0,0,0,0,0) & 3875 \\ \bottomrule
            \end{array}\]
            \caption{Highest weights and dimensions for the global decomposition of \texorpdfstring{$\fA$}{A} for type \texorpdfstring{$E_8$}{E8}.}%
            \label{tbl:decgE8char}
        \end{minipage}

    \item\begin{minipage}[t]{\linewidth}
        \captionsetup{type=table}
        \[\begin{array}{c|cc} 
            \fus & a & b \\ \hline 
            a & a & b \\
            b & b & a,b
        \end{array}\]
        \caption{The fusion law \texorpdfstring{$(X_g,\fus)$}{} for type \texorpdfstring{$E_n$}{En}.}%
        \label{tbl:decgEfus}
    \end{minipage}

    \item
        \begin{minipage}[t]{\linewidth}$n=6$

            Let $S = \{ \beta \in \Phi \mid \kappa(\alpha,\beta) = \pm 1 \}$.
            \begin{align*}
                A_{b,12}^\idxe &\coloneqq \langle {[b_7+b_8]}_\beta \mid \beta \in S \rangle \\
                A_{b,13}^\idxe &\coloneqq \begin{multlined}[t]
                    \langle {[b_i-b_j]}_\beta,x_{\beta + b_i - b_j}\mid 1 \leq i < j \leq 6, \beta \in S, \qquad \\
                    \kappa(b_i+b_j,\beta) = 1, \kappa(b_i-b_j,\beta)=0 \rangle
                    \end{multlined} \\
                A_{b,14}^\idxe &\coloneqq \begin{multlined}[t]
                    \langle {[b_i-b_j]}_\beta,x_{\beta + b_i - b_j}\mid 1 \leq i < j \leq 6, \beta \in S, \qquad \\
                    \kappa(b_i+b_j,\beta) = -1, \kappa(b_i-b_j,\beta)=0 \rangle
                    \end{multlined}
            \end{align*}
        \end{minipage}

        \noindent\begin{minipage}[t]{\linewidth}$n=7$

            Also here we let $S = \{ \beta \in \Phi \mid \kappa(\alpha,\beta) = \pm 1\}$.
            \begin{align*}
                A_{b,9}^\idxe &\coloneqq \langle {[b_7+b_8]}_{\beta} \mid \beta \in S \rangle \\
                A_{b,10}^\idxe &\coloneqq \langle {[\gamma]}_\beta , x_{\beta+\gamma} \mid \gamma = \pm b_i \pm b_j \text{ for } 1 \leq i < j \leq 6, \kappa(\gamma,\beta)=0 \rangle
            \end{align*}
        \end{minipage}

        \noindent\begin{minipage}[t]{\linewidth}$n=8$
            \begin{align*}
                A_{b,6}^\idxe &\coloneqq \langle {[\alpha]}_{\beta} \mid \beta \in \Phi, \kappa(\alpha,\beta)=\pm 1 \rangle \\
                A_{b,7}^\idxe &\coloneqq \langle {[\gamma]}_\beta , x_{\beta+\gamma} \mid \gamma,\beta \in \Phi, \kappa(\alpha,\beta)=\pm 1 , \kappa(\gamma,\alpha)=\kappa(\gamma,\beta)=0 \rangle
            \end{align*}
        \end{minipage}

    \item
        \begin{minipage}[t]{\linewidth}$n=6$
            \captionsetup{type=table}
            \[\begin{array}{ccc} \toprule
                \text{Component} & \text{Highest weight}     & \text{Dimension} \\ \midrule
                A_1^\idxe & 2 \cdot (0;0,0,0,0,0)  & 2 \cdot 1\\
                A_2^\idxe & (0;1,0,0,0,1) & 35 \\
                A_3^\idxe & (0;0,1,0,1,0) & 189 \\
                A_5^\idxe & (2;1,0,0,0,1) & 105 \\
                A_{12}^\idxe & (1;0,0,1,0,0) & 40 \\
                A_{13}^\idxe & (1;1,1,0,0,0) & 140 \\
                A_{14}^\idxe & (1;0,0,0,1,1) & 140 \\ \bottomrule
            \end{array}\]
            \caption{Highest weights and dimensions for the local decomposition of \texorpdfstring{$\fA$}{A} for type \texorpdfstring{$E_6$}{E6}.}%
            \label{tbl:declE6char}
        \end{minipage}

        \noindent\begin{minipage}[t]{\linewidth}$n=7$
            \captionsetup{type=table}
            \[\begin{array}{ccc} \toprule
                \text{Component} & \text{Highest weight}     & \text{Dimension} \\ \midrule
                A_1^\idxe & 2 \cdot (0;0,0,0,0,0,0)  & 2 \cdot 1\\
                A_2^\idxe & (0;2,0,0,0,0,0) & 77 \\
                A_3^\idxe & (0;0,0,0,1,0,0) & 495 \\
                A_5^\idxe & (2;0,1,0,0,0,0) & 198 \\
                A_{9}^\idxe & (1;0,0,0,0,1,0) & 64 \\
                A_{10}^\idxe & (1;1,0,0,0,0,1) & 704 \\ \bottomrule
            \end{array}\]
            \caption{Highest weights and dimensions for the local decomposition of \texorpdfstring{$\fA$}{A} for type \texorpdfstring{$E_7$}{E7}.}%
            \label{tbl:declE7char}
        \end{minipage}

        \noindent\begin{minipage}[t]{\linewidth}$n=8$
            \captionsetup{type=table}
            \[\begin{array}{ccc} \toprule
                \text{Component} & \text{Highest weight}     & \text{Dimension} \\ \midrule
                A_1^\idxe & 2 \cdot (0;0,0,0,0,0,0,0)  & 2 \cdot 1\\
                A_3^\idxe & (0;0,0,0,0,0,1,0) & 1539 \\
                A_5^\idxe & (2;1,0,0,0,0,0,0) & 399 \\
                A_{6}^\idxe & (1;0,0,0,0,0,0,1) & 112 \\
                A_{7}^\idxe & (1;0,1,0,0,0,0,0) & 1824 \\ \bottomrule
            \end{array}\]
            \caption{Highest weights and dimensions for the local decomposition of \texorpdfstring{$\fA$}{A} for type \texorpdfstring{$E_8$}{E8}.}%
            \label{tbl:declE8char}
        \end{minipage}

    \item
        \begin{minipage}[t]{\linewidth}$n=6$
            \captionsetup{type=table}
            \[\resizebox{\linewidth}{!}{$\displaystyle\begin{array}{c|ccccccc} 
                \fus & 1 & 2 & 3 & 5 & 12 & 13 & 14 \\ \hline 
                1 & 1 & 2 & 3 & 5 & 12 & 13 & 14 \\
                2 & 2 & 1,2,3 & 2,3 & 5 & 12,13,14 & 12,13 & 12,14 \\
                3 & 3 & 2,3 & 1,2,3 & 5 & 12,13,14 & 12,13,14 & 12,13,14 \\
                5 & 5 & 5 & 5 & 1,2,3,5 & 12,13,14 & 12,13 & 12,14 \\
                12 & 12 & 12,13,14 & 12,13,14 & 12,13,14 & 1,2,3,5 & 2,3,5 & 2,3,5 \\
                13 & 13 & 12,13 & 12,13,14 & 12,13 & 2,3,5 & 3 & 1,2,3,5 \\
                14 & 14 & 12,14 & 12,13,14 & 12,14 & 2,3,5 & 1,2,3,5 & 3 \\
            \end{array}$}\]
            \caption{The fusion law \texorpdfstring{$(X_l,\fus)$}{} for type \texorpdfstring{$E_6$}{E6}.}%
            \label{tbl:declE6fus}
        \end{minipage}

        \noindent\begin{minipage}[t]{\linewidth}$n=7$
            \captionsetup{type=table}
            \[\begin{array}{c|cccccc} 
                \fus & 1 & 2 & 3 & 5 & 9 & 10 \\ \hline 
                1 & 1 & 2 & 3 & 5 & 9 & 10 \\
                2 & 2 & 1,2 & 3 & 5 & 10 & 9,10 \\
                3 & 3 & 3 & 1,2,3 & 5 & 9,10 & 9,10 \\
                5 & 5 & 5 & 5 & 1,2,3,5 & 9,10 & 9,10 \\
                9 & 9 & 10 & 9,10 & 9,10 & 1,3,5 & 2,3,5 \\
                10 & 10 & 9,10 & 9,10 & 9,10 & 2,3,5 & 1,2,3,5 \\
            \end{array}\]
            \caption{The fusion law \texorpdfstring{$(X_l,\fus)$}{} for type \texorpdfstring{$E_7$}{E7}.}%
            \label{tbl:declE7fus}
        \end{minipage}

        \noindent\begin{minipage}[t]{\linewidth}$n=8$
            \captionsetup{type=table}
            \[\begin{array}{c|ccccc} 
                \fus & 1 & 3 & 5 & 6 & 7 \\ \hline 
                1 & 1 & 3 & 5 & 6 & 7 \\
                3 & 3 & 1,3 & 5 & 6,7 & 6,7 \\
                5 & 5 & 5 & 1,3,5 & 6,7 & 6,7 \\
                6 & 6 & 6,7 & 6,7 & 1,3,5 & 3,5 \\
                7 & 7 & 6,7 & 6,7 & 3,5 & 1,3,5 \\
            \end{array}\]
            \caption{The fusion law \texorpdfstring{$(X_l,\fus)$}{} for type \texorpdfstring{$E_8$}{E8}.}%
            \label{tbl:declE8fus}
        \end{minipage}
\end{enumerate}

\section{An algebra for \texorpdfstring{$E_8$}{E8}}%
\label{sec:E8}

In this section, we direct some more attention to the case where $\Phi$ is of type $E_8$ since this was the original algebra of interest.
We prove that $\fA$ belongs to a one-parameter family of algebras.
Each of these can be given the structure of an axial decomposition algebra.

\begin{definition}%
    \label{def:pprod}
    Let $\fA$, equipped with the $\mL$-equivariant bilinear product $\ast$ and bilinear form $\mB$, be as in \cref{sec:extending}.
    Let $\one$ be the unit for $(\fA,\ast)$ constructed in \cref{sec:unit} and $A'$ the orthogonal complement of $\langle \one \rangle$ with respect to $\mB$.
    Consider a parameter $p \in \C$.
    Define the following product and bilinear form on $\fA$ that depends on the parameter $p$:
    \begin{align*}
        (c_1 \one + a_1) \pprod (c_2 \one + a_2) &\coloneqq (c_1c_2 + p\mB(a_1,a_2) )\one + c_1a_2 + c_2a_1 + a_1 \ast a_2 \\
        \mB_p (c_1\one + a_1,c_2\one + a_2) & \coloneqq c_1c_2 \mB(\one,\one) + \mB(a_1,a_2) (1 + p\mB(\one,\one))
    \end{align*}
    for all $c_1,c_2 \in \C$ and $a,b \in A'$. Note that we retrieve the original product $\ast$ and bilinear form $\mB$ if we put $p = 0$.
\end{definition}

The important properties of $\ast$ and $\mB$ from \cref{prop:defast} still hold for this new product and bilinear form.

\begin{proposition}%
    \label{prop:pprod}
    The triple $(\fA,\pprod,\mB_p)$ is a Frobenius algebra for $\mL$ for any choice of the parameter $p \in \C$ with $1 + p\dim(\mL) \neq 0$.
\end{proposition}
\begin{proof}
    The $\mL$-equivariance of $\pprod$ follows from the definition of $\pprod$ and \cref{prop:defast}.
    Also from \cref{prop:defast} we have
    \begin{multline*}
        \mB((c_1 \one + a_1) \pprod (c_2 \one + a_2) , (c_3 \one + a_3)) \\
        \begin{aligned}
            &= c_1c_2c_3\mB(\one,\one)+  (1+p\mB(1,1))(c_1 \mB(a_2,a_3) \\
            &\qquad + c_2\mB(a_1,a_3)+c_3\mB(a_1,a_2) + \mB(a_1 \ast a_2, a_3)) \\
            &= \mB((c_{\pi(1)} \one + a_{\pi(1)}) \pprod (c_{\pi(2)} \one + a_{\pi(2)}) , (c_{\pi(3)} \one + a_{\pi(3)})).
        \end{aligned}
    \end{multline*}
    for all $c_1,c_2,c_3 \in \C$, $a_1,a_2,a_3 \in A'$ and any permutation $\pi$ of $\{1,2,3\}$.
    The non-degeneracy of $\mB_p$ follows from the non-degeneracy of $\mB$ if $1+p\mB(\one,\one) \neq 0$.
    From the construction of the unit in \cref{sec:unit} it follows that $\mB(\one,\one) = \tr(\id_{\mL}) = \dim(\mL)$.
\end{proof}

Next, we show that also the structure as a decomposition algebra from \cref{def:dec} transfers to this new algebra product.

\begin{proposition}
    Let $\idx$, $\Omega$ and $\mF$ be as in \cref{def:dec}.
    Then $(\fA,\idx,\Omega)$ is an $\mF$\dash{}decomposition algebra for the algebra product $\pprod$ on $\fA$.
\end{proposition}
\begin{proof}
    Let $X_g$, $X_l$, $\mL_\idxe$, $A_x$, $A_y^\idxe$ and $A_{x,y}^\idxe$ for $x \in X_g$, $y \in X_l$ and $\idxe \in \idx$ be as in \cref{def:dec}.
    Let $e_g \in X_g$ (resp.\ $e_l \in X_l$) be the element corresponding to the trivial $\mL$-module (resp.\ $\mL_\idxe$-module).
    Then $A_{e_g} = A^\idxe_{e_g,e_l} = \langle \one \rangle$ (as can be seen from \cref{subsec:decEn}).
    Note that $(e_g,e_l)$ is a unit for the fusion law $\mF$ and since $\one$ is still a unit for the algebra $(\fA,\pprod)$ we have indeed $A_{e_g,e_l}^\idxe \pprod A_{x,y} \subseteq A_{x,y}$ for all $(x,y) \in X_g \times X_l$.
    Also, since $\mB$ is $\mL$-equivariant, we have $\mB(A_{e_g,e_l}^\idxe , A_{x,y}^\idxe) = 0$ for all $(x,y) \in X_g \times X_l \setminus \{(e_g,e_l)\}$.
    Thus, for all $x_1,x_2 \in X_g$ and $y_1,y_2 \in X_l$ such that $(x_1,y_1) \neq (e_g,e_l) \neq (x_2,y_2)$:
    \begin{align*}
        A_{x_1,y_2}^\idxe \pprod A_{x_2,y_2}^\idxe & \subseteq A_{x_1,y_1}^\idxe \ast A_{x_2,y_2}^\idxe + \langle \one \rangle \\
        & \subseteq A_{(x_1,y_1) \fus (x_2,y_2)}^\idxe + \mB(A_{x_1,y_1}^\idxe,A_{x_2,y_2}^\idxe)\langle \one \rangle.
    \end{align*}
    Suppose that $\mB(A_{x_1,y_1}^\idxe,A_{x_2,y_2}^\idxe) \neq 0$, then there exists an $\mL$-equivariant map $A_{x_1} \otimes A_{x_2} \to \langle \one \rangle$ and an $\mL_\idxe$-equivariant map $A_{y_1}^\idxe \otimes A_{y_2}^\idxe \to \langle \one \rangle$.
    By definition of $\mF$ this means that $(e_g,e_l) \in (x_1,y_1)\fus (x_2,y_2)$.
    Therefore
    \[
        A_{(i_1,j_1) \fus (i_2,j_2)}^\idxe + \mB(A_{i_1,j_1}^\idxe,A_{i_2,j_2}^\idxe)\langle \one \rangle \subseteq A_{(i_1,j_1) \fus (i_2,j_2)}^\idxe. \qedhere
    \]
\end{proof}

In the remainder of this section we restrict to the case where $\Phi$ is of type $E_8$.
The decomposition $\bigoplus_{x,i} A^\idxe_{x,i}$ is given in \cref{subsec:decEn}.
Note that there are only six non-zero components in this decomposition, namely $A_{a,1}^\idxe$, $A_{b,1}^\idxe$, $A_{b,3}^\idxe$, $A_{b,5}^\idxe$, $A_{b,6}^\idxe$ and $A_{b,7}^\idxe$ of respective dimensions 1, 1, 1539, 399, 112 and 1824.
Each of these is irreducible as an $\mL_\idxe$\dash{}representation.
The corresponding sublaw of $\mF$ on these components is given in \cref{tbl:E8fus}.
(To preserve space we have denoted $(x,i)$ by $xi$.)

\begin{table}[ht]
    \[\begin{array}{c|cccccc} 
        \fus & a1 & b1 & b3 & b5 & b6 & b7 \\ \hline 
        a1 & a1 & b1 & b3 & b5 & b6 & b7 \\
        b1 & b1 & a1,b1 & b3 & b5 & b6 & b7 \\
        b3 & b3 & b3 & a1,b1,b3 & b5 & b6,b7 & b6,b7 \\
        b5 & b5 & b5 & b5 & a1,b1,b3,b5 & b6,b7 & b6,b7 \\
        b6 & b6 & b6 & b6,b7 & b6,b7 & a1,b1,b3,b5 & b3,b5 \\
        b7 & b7 & b7 & b6,b7 & b6,b7 & b3,b5 & a1,b1,b3,b5
    \end{array}\]
    \caption{The fusion law for type $E_8$.}\label{tbl:E8fus}
\end{table}

We want to look for an axis for this decomposition on which $\mL_\idxe$ acts trivially.
Such an axis must be contained in $A_1^\idxe = A_{a,1}^\idxe \oplus A_{b,1}^\idxe$.
Therefore, we need to know the action of $A_1^\idxe$ on $\fA$ by multiplication.
Consider a Chevalley basis of $\mL$ as in \cref{def:stage} and take $\idxe \coloneqq \langle\!\langle e_\alpha, e_{-\alpha} \rangle\!\rangle$ for some root $\alpha \in \Phi^+$ as before.
Let $a_\alpha$ be the projection with respect to $\mB$ of $j_\alpha$ onto $A_{b,1}$.
Then $A_1^\idxe = \langle \one , a_\alpha \rangle$ and $\mB(\one,a_\alpha) = 0$.
Since $\one$ is a unit for $(\fA,\pprod)$, it suffices to describe
\[
    \ad_{a_\alpha} : \fA \to \fA : a \mapsto a_\alpha \pprod a.
\]
Note that $\mL_\idxe$ fixes $a_\alpha$ and hence $\ad_{a_\alpha}$ is an isomorphism of $\mL_\idxe$\dash{}representations.
Since each $\mL_\alpha$-isotypic component $A_i^\idxe$ is irreducible if $i \neq 1$, the operator $\ad_{a_\alpha}$ must act as a scalar on each $A_i^\idxe$ for $i \neq 1$, by Schur's lemma.

\begin{proposition}%
    \label{prop:ada}
    The linear map $\ad_{a_\alpha}$ is defined by
    \begin{align*}
        \one &\mapsto a_\alpha, \\
        a_\alpha &\mapsto \left(\tfrac{1}{496}+\tfrac12 p\right)\one + \tfrac{9}{98} a_\alpha, \\
        a &\mapsto -\tfrac{3}{196} a &\text{if $a \in A_{b,3}^\idxe$}, \\
        a &\mapsto \tfrac{9}{196} a &\text{if $a \in A_{b,5}^\idxe$}, \\
        a &\mapsto \tfrac{9}{196} a &\text{if $a \in A_{b,6}^\idxe$}, \\
        a &\mapsto 0 &\text{if $a \in A_{b,7}^\idxe$}.
    \end{align*}
\end{proposition}
\begin{proof}
    Since $\mL_\idxe$ acts trivially on $\langle a_\alpha \rangle$, it follows by Schur's lemma that $\ad_{a_\alpha}$ must act as a scalar on each $\mL_\idxe$-isotypic component that is irreducible.
    More precisely, if $A_{x,y}^\idxe$ is irreducible (this is true for $(x,y) \in \{(b,3),(b,5),(b,6),(b,7)\}$), then for $a \in A_{x,y}^\idxe$ we have $a_\alpha \pprod a = \lambda a$ for some $\lambda \in \C$ that does not depend on the choice of $a \in A_{x,y}^\idxe$.
    So it suffices to compute $a_\alpha \pprod a$ for any $a \in A_{x,y}^\idxe \setminus \{0\}$ to determine $\lambda$.
    We can get such an element $a$ from the explicit description of the decomposition from \cref{subsec:decEn}.
    Thus we only have to compute a few products together with the product $a_\alpha \pprod a_\alpha$.
    We have computed these products using a computer but the computation (although lengthy) can be done by hand.
\end{proof}

If $\ax_\idxe$ is an axis, then we must have $\ax_\idxe \pprod \ax_\idxe \in \langle \ax_\idxe \rangle$.
Therefore, we search for idempotents or nilpotents in $A_1^\idxe$.

\begin{proposition}%
    \label{prop:idempotents}
    \begin{enumerate}
        \item\label{prop:idempotents:neq} If $p \neq -\tfrac{614}{74431}$, then the subalgebra $(A_1^\idxe,\pprod)$ of $(\fA,\pprod)$ is generated by two primitive, orthogonal idempotents.
        \item If $p = -\tfrac{614}{74431}$, then the subalgebra $(A_1^\idxe,\pprod)$ of $(\fA,\pprod)$ is generated by $\one$ and a nilpotent element.
    \end{enumerate}
\end{proposition}
\begin{proof}
    An arbitrary element of $A_1^\idxe$ is of the form $c_1 \one + c_2 a_\alpha$ for $c_1,c_2 \in \C$. From \cref{prop:ada} it follows that
    \[
        {(c_1 \one + c_2 a_\alpha)}^2 = \left(c_1^2 + \left(\frac{1}{496}+\frac{1}{2}p \right)c_2^2\right) \one + \left( 2c_1c_2 + \frac{9}{98}c_2^2 \right) a_\alpha.
    \]
    Expressing that this element is an idempotent amounts to solving a system of two non-linear equations. A small calculation shows that we have 4 solutions (including the trivial solutions $c_1=c_2=0$ and $c_1=1$, $c_2=0$) if $p \neq - \frac{614}{74431}$. If not, then we only have the two trivial solutions but then $\one - \tfrac{196}{9} a_\alpha$ is nilpotent.
\end{proof}

From now on we will always assume that $p \neq -\frac{614}{74431}$.

\begin{definition}%
    \label{def:idempotents}
    Let $\ax_\idxe$ be one of the idempotents from \cref{prop:idempotents}~\ref{prop:idempotents:neq}.
    Then $\one-\ax_\idxe$ is the other idempotent.
    Write $\ax_\idxe = c_1 \one + c_2 a_\alpha$ where $a_\alpha \in A'$.
    Explicitly we have
    \begin{align*}
        c_1 &= \frac{1}{2} \pm \frac{9\sqrt{62}}{4\sqrt{74431p+2456}}, \\
        c_2 &= \frac{49\sqrt{62}}{4\sqrt{74431p+2456}}.
    \end{align*}
    Because we assume $p \neq -\frac{614}{74431}$, we have $c_1 \neq \frac{1}{2}$.
    Therefore, we can distinguish between the two idempotents by computing $\mB(\ax_\idxe,\one) = c_1 \mB(\one,\one)$.
    Now we pick $\ax_\idxe$ for each $\idxe \in \idx$ such that $\mB(\ax_\idxe,\one)$ is constant for all $\idxe\in\idx$.
    Also let $A_\ax^\idxe \coloneqq \langle \ax_\idxe \rangle$ and $A_{\ax'}^\idxe \coloneqq \langle \one - \ax_\idxe \rangle$.
\end{definition}

\begin{theorem}%
    \label{thm:E8}
    Let $\Phi$ be an irreducible root system of type $E_8$, $(\fA,\pprod)$ the algebra parametrized by $p$ from \cref{def:pprod} and $c_1$ as in \cref{def:idempotents}.
    Let $\idx$ be as in \cref{def:dec}.
    For each $\idxe \in \idx$ let $\ax_\idxe$, $A_\ax^\idxe$ and $A_{\ax'}^\idxe$ be as in \cref{def:idempotents} and $A_3^\idxe$, $A_5^\idxe$, $A_6^\idxe$ and $A_7^\idxe$ as in \cref{subsec:decEn}.
    \begin{enumerate}
        \item\label{thm:E8:dec} The decomposition $A_\ax^\idxe \oplus A_{\ax'}^\idxe \oplus A_{3}^\idxe \oplus A_{5}^\idxe \oplus A_{6}^\idxe \oplus A_{7}^\idxe$ is an $\mF'$-decomposition of $(\fA,\pprod)$ where $\mF'$ is the fusion law from \cref{tbl:E8fusnew}.
            The element $\ax$ is a unit for $\mF'$.
        \item\label{thm:E8:ax} Let $\Omega$ be the tuple of decompositions from~\ref{thm:E8:dec} indexed by $\idx$.
            Then the quadruple $(\fA,\idx,\Omega,\idxe \mapsto \ax_\idxe)$ is an axial decomposition algebra with evaluation map
        \begin{align*}
            \ax &\mapsto 1 \\
            \ax' &\mapsto 0 \\
            3 &\mapsto \tfrac{4}{3}c_1 - \tfrac{1}{6}\\
            5 &\mapsto \tfrac{1}{2} \\
            6 &\mapsto \tfrac{1}{2} \\
            7 &\mapsto c_1
        \end{align*}
        \item\label{thm:E8:gen} The algebra $(\fA,\pprod)$ is generated by the idempotents $\ax_\idxe$ for $\idxe \in \idx$ if $c_1 \neq 0$.
            If $c_1 = 0$, then these idempotents generate the subalgebra $A'$ with $A'$ as in \cref{def:pprod}.
    \end{enumerate}
\end{theorem}
\begin{proof}
    Part~\ref{thm:E8:dec} and~\ref{thm:E8:ax} follow immediately from the calculations in \cref{subsec:decEn,prop:ada,prop:idempotents}.
    Since the elements $\idxe \in \idx$ are conjugate for the action of $\mL$, also the idempotents $\ax_\idxe$ must be conjugate.
    Hence they span a $\mL$-invariant subspace of $\fA$.
    Assume that $c_1 \neq 0$.
    From the global decomposition (\cref{tbl:decgE8char}) we know that $\fA$ only has two proper $\mL$-invariant subspaces, namely $\langle \one \rangle$ and its orthogonal complement with respect to $\mB$.
    Since $\mL$ acts non-trivial on the idempotents $\ax_\idxe$ and $\mB(\one,\ax_\idxe) = c_1\mB(\one,\one) \neq 0$ it follows that $\fA$ is spanned by the elements $\ax_\idxe$.
    In particular the algebra $(\fA,\pprod)$ is generated by them.
    If $c_1=0$, then $p=-\tfrac{1}{248}$ and $\mB(\one,\ax_\idxe) = 0$.
    Moreover, if $a,b \in A'$, then $\mB_p(a \pprod b , \one) = \mB_p(a , b) = 0$ by definition of $\mB_p$.
    Therefore $A'$ is a subalgebra of $(\fA,\pprod)$.
    Since $A'$ is irreducible as $\mL$\dash{}module, it must be spanned, and therefore generated, by the idempotents $\ax_\idxe$ for $\idxe \in \idx$.
\end{proof}

\begin{table}[ht]
    \[\begin{array}{c|cccccc} 
        \fus & \ax & \ax' & 3 & 5 & 6 & 7 \\ \hline 
        \ax & \ax &  & 3 & 5 & 6 & 7 \\
        \ax' &  & \ax' & 3 & 5 & 6 & 7 \\
        3 & 3 & 3 & \ax,\ax',3 & 5 & 6,7 & 6,7 \\
        5 & 5 & 5 & 5 & \ax,\ax',3,5 & 6,7 & 6,7 \\
        6 & 6 & 6 & 6,7 & 6,7 & \ax,\ax',3,5 & 3,5 \\
        7 & 7 & 7 & 6,7 & 6,7 & 3,5 & \ax,\ax',3,5
    \end{array}\]
    \caption{The fusion law $\mF'$.}%
    \label{tbl:E8fusnew}
\end{table}

\begin{remark}
    \begin{enumerate}
        \item Because the global decomposition for types $A_n$ and $D_n$ contains three terms (see \cref{sec:dec}), it is possible to write down an $\mL$\dash{}equivariant product, as in \cref{def:pprod}, with two degrees of freedom instead of one. If we write $A_a = \langle \one \rangle$, $A_b$ and $A_c$ for the components of the global decomposition, then these subspace are orthogonal with respect to the Frobenius form $\mB$. We can define a new product on $\fA$ with two parameters $p_1$ and $p_2$ such that
            \[
                a \pprod b = a \ast b + p_1 \mB(a_b,b_b) \one + p_2 \mB(a_c,b_c) \one
            \]
            where $a_x$ (resp.\ $b_x$) is the projection of $a$ (resp.\ $b$) onto $A_x$ for $x \in \{b,c\}$.
        \item If $\Phi$ is of type $A_n$, $D_n$, $E_6$ or $E_7$, we can also try to find idempotents $\ax_\idxe$ in the subalgebra $A_1^\idxe$.
            However, in \cref{prop:ada} we used Schur's lemma to derive the adjoint action of the elements of $A_a^\idxe$.
            Note that this is no longer possible for the terms of the local decomposition that are not irreducible $\mL_\idxe$-representations.
            This would lead to further difficulties when trying to establish the diagonalizability of the adjoint action of such an idempotent~$\ax_\idxe$.
    \end{enumerate}
\end{remark}

\appendix

\section{The character computation of \texorpdfstring{$\mV$}{V}}%
\label{app:charV}

In this section we prove \cref{prop:charV}, which gives the character of $\mV$ as an $\mL$\dash{}representation. We use Freudenthal's formula~\cite{Hum72}*{\S 22.3} to compute this character in a combinatorial way. Although this character is essential in \cref{prop:gensetV}, the computation is quite technical. We use the notation from \cref{def:stage}. Recall the definition of $\Lambda_i$ for $-2 \leq i \leq 2$ and $n_\lambda$ for $\lambda \in \Lambda_i$ from \cref{def:charV2}. In addition to \cref{lem:weights} we prove a few more combinatorial properties about the weights $\lambda \in \Lambda_i$.

\begin{lemma}%
    \label{lem:charV}
    \begin{enumerate}
        \item $n_0 = \frac{\abs{\Phi}}{2}$ and $n_\lambda = 1$ for $\lambda \in \Lambda_1 \cup \Lambda_2$.
        \item\label{lem:charV:omega+psi} Suppose $\omega \in \Phi$ is the highest root of $\Phi$ and $\lambda \in \bigcup_{0 \leq i \leq 2} \Lambda_i$ is dominant. Then $\lambda = \omega + \psi$ for some $\psi \in \Phi^+$.
        \item\label{lem:charV:eq} If $\lambda \in \Lambda_k$ and $\lambda + i\alpha \in \Lambda_j$ for $i\geq 1$, $\alpha \in \Phi$ and $-2 \leq j,k \leq 2$, then $j = i^2 + i\kappa(\lambda,\alpha)+k$.
        \item\label{lem:charV:left} Let $\lambda \in \Lambda$ be dominant and $\omega \in \Phi$ the highest root.
            Suppose $f \in W$ such that $f(\lambda) = \lambda$ and ${f(2\omega-\lambda)} = \lambda-2\omega$.
            Then $\sum_{\alpha \in \Phi^+} \kappa(2\omega - \lambda , \alpha) = \sum_{\substack{\alpha \in \Phi^+ \\ \kappa(\lambda,\alpha)=0}} \kappa(2\omega,\alpha)$.
        \item\label{lem:charV:count} Suppose $\alpha,\beta \in \Phi$ such that $\kappa(\alpha,\beta)=0$. The number of roots $\gamma \in \Phi^+$ such that $\kappa(\alpha,\gamma) = -\kappa(\beta,\gamma) = \pm 1$ is $2(n_{\alpha+\beta} - 1)$.
        \item\label{lem:charV:nl>1} For each $\lambda \in \Lambda_0$, we have $n_\lambda >1$.
    \end{enumerate}
\end{lemma}
\begin{proof}
    \begin{enumerate}
        \item Of course $n_0 = \frac{\abs{\Phi}}{2}$ because $\alpha + \beta = 0$ for $\alpha , \beta \in \Phi$ if and only if $\alpha =-\beta$. Let $\lambda \in \Lambda_1 \cup \Lambda_2$. Suppose $\lambda = \alpha + \beta = \alpha' + \beta'$ for some $\alpha,\beta,\alpha',\beta' \in \Phi$ for which $\kappa(\alpha,\beta) > 0$. Then $\kappa(\alpha',\alpha + \beta) \geq 3$. Since $\Phi$ is simply laced, $\kappa(\alpha',\alpha) = 2$ or $\kappa(\alpha',\beta) = 2$ and thus $\alpha' = \alpha$ or $\beta'=\beta$. Therefore $n_\lambda = 1$.
        \item Let $\lambda \in \bigcup_{0 \leq i \leq 2} \Lambda_i$ be dominant and write $\lambda = \alpha + \beta$ for some $\alpha, \beta \in \Phi$. Suppose that $\alpha$ is maximal with respect to the partial order $\preccurlyeq$ induced by the base $\Delta$. If $\alpha \neq \omega$, then there exists a root $\gamma \in \Phi^+$ for which $\kappa(\alpha,\gamma) = -1$. Since $\lambda$ is dominant and $\gamma \in \Phi^+$, we have $\kappa(\beta,\gamma)=1$. Therefore, both $\alpha + \gamma$ and $\beta - \gamma$ are roots and $\lambda = (\alpha + \gamma) + (\beta -\gamma)$. This contradicts the fact that $\alpha$ was maximal. Thus $\lambda = \omega + \psi$ for some $\psi \in \Phi$. Because $\kappa(\lambda,\psi) > 1$, the root $\psi$ must be positive.
        \item Suppose $\lambda + i\alpha \in \Lambda_j$ for some $i\geq 1$, $\alpha \in \Phi^+$ and $-2 \leq j \leq 2$. Then, by \cref{lem:weights}~\ref{lem:weights:disj}, $4+2j = \kappa(\lambda + i\alpha,\lambda + i\alpha) = 4+2k + 2i\kappa(\lambda,\alpha) + 2i^2$. Thus $j = i^2 + i\kappa(\lambda,\alpha) + k$.
        \item If $\alpha \in \Phi^+$ such that $\kappa(\lambda,\alpha) >0$ then also $\kappa(\lambda,f(\alpha)) > 0$. Because $\lambda$ is dominant, $f(\alpha)$ must be positive. Now $\kappa(2\omega-\lambda,\alpha + f(\alpha)) = 0$.
        \item If $\gamma \in \Phi^+$ such that $\kappa(\alpha,\gamma) = -\kappa(\beta,\gamma) = \pm 1$ then $\{s_\gamma(\alpha),s_\gamma(\beta)\} \in N_{\alpha + \beta}$. Conversely, if $\{\gamma,\delta\} \in N_{\alpha+\beta} \setminus \{\{\alpha,\beta\}\}$, then   precisely two of the four roots $\pm(\gamma-\alpha)$ and $\pm(\gamma-\beta)$ are positive and they satisfy the necessary requirement.
        \item Suppose that $\lambda = \omega + \psi$ for some $\omega , \psi \in \Phi$. We claim that there exists a root $\beta \in \Phi$ such that $\kappa(\omega,\beta)=\pm 1$ and $\kappa(\psi,\beta)=\pm 1$. If not, then every root would be orthogonal to either $\omega$ or $\psi$, which contradicts the irreducibility of $\Phi$. Now $\omega$, $\psi$ and $\beta$ form a root subsystem of $\Phi$ of type $A_3$ and inside this subsystem, we can find another way to write $\lambda$ as the sum of two orthogonal roots. \qedhere
    \end{enumerate}
\end{proof}

We are ready to prove \cref{prop:charV}.

\begin{proposition}%
    \label{prop:charVapp}
    The character of $\mV$ is given by
    \[
        \ch_\mV = n_0 + \sum_{\lambda \in \Lambda_{-1}} (n_\lambda + 1) e^{\lambda} + \sum_{\lambda \in \Lambda_0} (n_\lambda - 1)e^\lambda + \sum_{\lambda \in \Lambda_1 \cup \Lambda_2} e^\lambda.
    \]
\end{proposition}
\begin{proof}
    Let $m_\lambda$ be the dimension of the weight-$\lambda$-space of $\mV$.
    Write $\rho$ for the half-sum of all positive roots, this is, $\rho = \frac{1}{2} \sum_{\alpha \in \Phi^+} \alpha$ and $\omega$ for the highest root of $\Phi$.
    According to Freudenthal's formula~\cite{Hum72}*{\S 22.3}:
    \begin{equation} \label{eq:Freudenthal}
        (\kappa(2\omega + \rho,2\omega + \rho) - \kappa(    \lambda + \rho , \lambda + \rho)) m_\lambda = 2 \sum_{\alpha \in \Phi^+} \sum_{i \geq 1} \kappa(\lambda + i\alpha , \alpha)m_{\lambda + i\alpha}.
    \end{equation}
    We can rewrite the left hand side as
    \[
        \left(\kappa(2\omega,2\omega) - \kappa(\lambda,\lambda) + \sum_{\alpha \in \Phi^+} \kappa(2\omega-\lambda,\alpha)\right)m_\lambda.
    \]
    We compute the values of $m_\lambda$ inductively.

    \medskip\noindent\underline{\emph{Claim 1: $m_\lambda = 1$ for all $\lambda \in \Lambda_2$.}} Since $\mV$ is isomorphic to the highest weight representation of $\mL$ of highest weight $2\omega$, we have $m_{2\omega} = 1$. Because $\Phi$ is irreducible and simply laced, $W$ acts transitively on $\Phi$. Therefore $m_{2\alpha} = 1$ for all $\alpha \in \Phi$, or $m_\lambda = 1$ for all $\lambda \in \Lambda_2$.

    \medskip\noindent\underline{\emph{Claim 2: $m_\lambda = 1$ for all $\lambda \in \Lambda_1$.}}
    Now, let $\lambda \in \Lambda_1$ be dominant.
    Suppose $\lambda + i\alpha \in \Lambda_j$ for some $\alpha \in \Phi^+$ and $i \geq 1$.
    Because $\lambda$ is dominant $\kappa(\lambda,\alpha) \geq 0$ and, by \cref{lem:charV}~\ref{lem:charV:eq}, this is only possible if $j=2$, $i=1$ and $\kappa(\lambda,\alpha)=0$.
    Thus $\lambda + \alpha = 2\beta$ for some $\beta \in \Phi$ and $\kappa(\beta,\alpha)=1$.
    Thus $\beta-\alpha$ is a root.
    Because $n_\lambda = 1$, $\beta + (\beta - \alpha)$ is the unique way to write $\lambda$ as the sum of two roots.
    By \cref{lem:charV}~\ref{lem:charV:omega+psi} $\beta = \omega$ and $\alpha = 2\omega-\lambda \in \Phi^+$.
    Thus there exists precisely one $\alpha \in \Phi^+$ for which $\lambda + \alpha \in \Lambda_2$.
    From equation~\eqref{eq:Freudenthal} we have
        \begin{align*}
            \left(\kappa(2\omega,2\omega) - \kappa(\lambda,\lambda) + \sum_{\alpha \in \Phi^+} \kappa(2\omega-\lambda,\alpha)\right)m_\lambda = 2\kappa(2\omega,\lambda-2\omega) = 4
        \end{align*}
    In order to calculate the left hand side we write $\lambda = \omega+\psi$ for some $\psi \in \Phi^+$.
    Then $2\omega-\lambda = \omega-\psi$ is a positive root.
    Apply \cref{lem:charV}~\labelcref{lem:charV:left} with $f = s_{\omega-\psi}$.
    If $\alpha \in \Phi^+$ and $\kappa(\lambda,\alpha)=0$ then $\{s_\alpha(\omega),s_\alpha(\psi)\} \in N_\lambda$.
    Because $n_\lambda = 1$ either $\kappa(\omega,\alpha)=1$ or $\alpha = \omega-\psi$.
    Thus we have
        \begin{align*}
            \left(\kappa(2\omega,2\omega)-\kappa(\lambda,\lambda) + \sum_{\substack{\alpha \in \Phi^+\\\kappa(\lambda,\alpha)=0}} \kappa(2\omega,\alpha)\right)m_\lambda &= 4, \\
            \left(8-6 + \kappa(2\omega,\omega-\psi)\right) m_\lambda &= 4.
        \end{align*}
        Hence $m_\lambda = 1$ for all $\lambda \in \Lambda_1$.

    \medskip\noindent\underline{\emph{Claim 3: $m_\lambda = n_\lambda-1$ for all $\lambda \in \Lambda_0$.}}
    Consider a dominant weight $\lambda \in \Lambda_0$.
    If $\lambda + i\alpha \in \Lambda_j$ for some $i\geq 1$, $\alpha \in \Phi^+$ and $-2 \leq j \leq 2$ then, by \cref{lem:charV}~\labelcref{lem:charV:eq}, we have $i=1$, $j\geq 1$ and $\kappa(\lambda,\alpha)=j-1$.
    If $j=2$ then $\kappa(\lambda,\alpha) = 1$ and $\kappa(\alpha,\lambda + \alpha) = 3$ which is impossible since $\lambda + \alpha$ has to be the double of a root.
    The only remaining case is where $j=1$ and $\kappa(\lambda,\alpha)=0$. Since $\lambda + \alpha \in \Lambda_1$, $\lambda + \alpha$ can be written uniquely as the sum of two roots $\beta$ and $\gamma$.
    Of course $\alpha \neq \beta$ and $\alpha \neq \gamma$ because otherwise $\lambda$ would be a root.
    Thus $\kappa(\alpha,\beta) = \kappa(\alpha,\gamma)=1$ and $(\beta-\alpha) + \gamma$ and $\beta + (\gamma - \alpha)$ are two ways to write $\lambda$ as the sum of two roots.
    Conversely, if $\lambda = \delta + \varepsilon$ with $\delta,\varepsilon \in \Phi$ and $\alpha \in \Phi^+$ such that $\kappa(\delta,\alpha) = -\kappa(\varepsilon,\alpha)=\pm 1$, then $\lambda + \alpha \in \Lambda_1$.
    \cref{lem:charV}~\ref{lem:charV:count} and a double counting argument gives us the number of $\alpha \in \Phi^+$ for which $\lambda + \alpha \in \Lambda_1$: $n_\lambda (n_\lambda - 1)$. The right hand side of~\eqref{eq:Freudenthal} becomes $4n_\lambda (n_\lambda - 1)$.

    As far as the left hand side goes, we write $\lambda = \omega + \psi$ for some $\psi \in \Phi^+$ using \cref{lem:charV}~\ref{lem:charV:eq}.
    By \cref{lem:charV}~\ref{lem:charV:nl>1} we can find $\{\omega',\psi'\} \in N_{\omega + \psi} \setminus \{\{\omega,\psi\}\}$.
    Let $\beta_1 \coloneqq \omega - \omega'$ and $\beta_2 \coloneqq \omega - \psi'$.
    Then $\beta_1,\beta_2 \in \Phi$ and $\kappa(\omega,\beta_1)=\kappa(\omega,\beta_2)=1$ and $\kappa(\psi,\beta_1)=\kappa(\psi,\beta_2)=-1$.
    Apply \cref{lem:charV}~\labelcref{lem:charV:left} with $f = s_{\beta_1}s_{\beta_2}$.
    The left hand side of~\eqref{eq:Freudenthal} reduces to
    \[
        \left(4 + \sum_{\substack{\alpha \in \Phi^+\\\kappa(\lambda,\alpha)=0}} \kappa(2\omega,\alpha)\right)m_\lambda.
    \]
    The number of $\alpha \in \Phi^+$ for which $\kappa(\omega,\alpha)=1$ and $\kappa(\lambda,\alpha)=0$ is, because $\omega$ is dominant, equal to $2(n_\lambda-1)$ by \cref{lem:charV}~\ref{lem:charV:count}.
    Also $\kappa(\omega,\alpha)>0$ since $\omega$ is dominant and $\kappa(\omega,\alpha)=2$ if and only if $\alpha=\omega$ (but then $\kappa(\lambda,\alpha) = 2 \neq 0$).
    Thus \[\sum_{\substack{\alpha \in \Phi^+\\\kappa(\lambda,\alpha)=0}} \kappa(2\omega,\alpha) = 4(n_\lambda-1).\]
    We conclude, by \labelcref{eq:Freudenthal}, that $4n_\lambda m_\lambda = 4n_\lambda(n_\lambda-1)$ and thus $m_\lambda = n_\lambda-1$ for all $\lambda \in \Lambda_0$.

    \medskip\noindent\underline{\emph{Claim 4: $m_\lambda = n_\lambda+1$ for all $\lambda \in \Lambda_{-1}$.}}
    The only dominant weight in $\Lambda_{-1} = \Phi$ is $\omega$.
    Now, by \cref{lem:charV}~\ref{lem:charV:eq}, if $\omega + i\alpha \in \Lambda_j$ for $i\geq 1$ and $\alpha \in \Phi^+$ then $i=1$ and $\kappa(\omega,\alpha) = j$.
    Obviously, the converse is also true.
    Hence, for the right hand side of~\eqref{eq:Freudenthal}:
    \[
        2 \sum_{\alpha \in \Phi^+} \kappa(\omega+\alpha,\alpha)m_{\omega+\alpha} = 2 \sum_{\substack{\alpha \in \Phi^+ \\ \kappa(\omega,\alpha)=0}} 2\cdot (n_{\omega+\alpha}-1) + 2 \sum_{\substack{\alpha \in \Phi^+ \\ \kappa(\omega,\alpha)=1}} 3\cdot 1 + 2 \sum_{\substack{\alpha \in \Phi^+ \\ \kappa(\omega,\alpha)=2}} 4 \cdot 1
    \]

    Let $\alpha \in \Phi^+$ such that $\kappa(\omega,\alpha)=0$.
    Let $\beta \in \Phi^+$ such that $\kappa(\beta,\omega)=1$ and $\kappa(\beta,\alpha)=-1$.
    Since $\omega$ is dominant, there are precisely $2(n_{\alpha+\omega}-1)$ choices for $\beta$ by \cref{lem:charV}~\ref{lem:charV:count}.
    Then $\{\beta,\omega-\beta\}$ and $\{\alpha + \beta, \omega - \alpha + \beta\}$ are two different elements of $N_\omega$.
    Conversely, let $\{\gamma,\delta\}$ and $\{\varepsilon,\zeta\}$ be two different elements of $N_\omega$.
    Then $\kappa(\gamma,\varepsilon)=1$ or $\kappa(\gamma,\zeta)=1$.
    Without loss of generality, assume that $\kappa(\gamma,\varepsilon) = 1$.
    Also assume that $\gamma-\varepsilon$ is positive (otherwise take $\varepsilon-\gamma$).
    Since $\kappa(\omega,\varepsilon)=1$ and $\omega$ is dominant, the root $\varepsilon$ must be positive.
    Thus $\alpha \coloneqq \gamma-\varepsilon \in \Phi^+$ and $\beta \coloneqq \varepsilon \in \Phi^+$ are positive roots for which $\kappa(\omega,\alpha)=0$, $\kappa(\omega,\beta) = 1$ and $\kappa(\alpha,\beta)=-1$.
    The same reasoning applies when $\gamma$ is replaced by $\delta$ but leads to the same $\alpha$ and $\beta$.
    This double counting argument gives us
    \[
        \sum_{\substack{\alpha \in \Phi^+ \\ \kappa(\omega,\alpha)=0}} 2\cdot (n_{\omega+\alpha}-1) = n_\omega(n_\omega-1).
    \]

    Now consider $\alpha \in \Phi^+$ such that $\kappa(\omega,\alpha)=1$.
    Then $\{\alpha,\omega-\alpha\} \in N_\omega$.
    Conversely, if $\{\beta,\gamma\} \in N_\omega$ then $\kappa(\omega,\beta)=\kappa(\omega,\gamma)=1$.
    Thus
    \begin{equation}\label{eq:2nomega}
        \sum_{\substack{\alpha \in \Phi^+ \\ \kappa(\omega,\alpha)=1}} 1 = n_\omega \cdot 2.
    \end{equation}
    The only root $\alpha \in \Phi^+$ such that $\kappa(\omega,\alpha)=2$, is $\omega$ itself.

    So the right hand side of~\eqref{eq:Freudenthal} equals
    \[
        2 n_\omega(n_\omega-1) + 12n_\omega + 8 = 2n_\omega^2 + 10n_\omega + 8 = 2(n_\omega+1)(n_\omega+4).
    \]

    As far as the left hand side goes, we have
    \[
        \left(\kappa(2\omega,2\omega) - \kappa(\omega,\omega) + \sum_{\alpha \in \Phi^+} \kappa(\omega,\alpha)\right)m_\omega = \left(8-2+\sum_{\substack{\alpha \in \Phi^+\\ \kappa(\omega,\alpha)=1}} 1 + 2\right)m_\omega.
    \]
    Using~\eqref{eq:2nomega}, we conclude
    \[
        (2n_\omega+8)m_\omega = 2(n_\omega+1)(n_\omega+4).
    \]
    Hence $m_\lambda=n_\lambda + 1$ for all $\lambda \in \Lambda_{-1}$.

    \medskip\noindent\underline{\emph{Claim 5: $m_0 = n_0$.}}
    Finally, we compute $m_0$, once again using~\eqref{eq:Freudenthal}.
    By~\eqref{eq:2nomega}, the left hand side equals $4(n_\omega+3)m_0$.
    Obviously $0+i\alpha \in \Lambda_j$ for $i \geq 0$ and $\alpha \in \Phi^+$ if and only if $i=1$ and $j=-1$ or $i=2$ and $j=2$.
    Because $W$ acts transitively on $\Phi$, we have $n_\alpha = n_\omega$ for all $\alpha \in \Phi$.
    The right hand side becomes
    \begin{align*}
        2\sum_{\alpha \in \Phi^+} 2\cdot (n_\alpha + 1) + 2 \sum_{\alpha \in \Phi^+} 4\cdot 1 &= 4 \sum_{\alpha \in \Phi^+} n_\alpha + 12 \sum_{\alpha \in \Phi^+} 1, \\
            &= \frac{\abs{\Phi}}{2} \cdot (4n_\omega + 12)
    \end{align*}
    Hence $m_0 = \frac{\abs{\Phi}}{2} = n_0$.
\end{proof}

\bibliography{bibliography}

\begin{bibdiv}
\begin{biblist}

\bib{BI84}{book}{
      author={Bannai, Eiichi},
      author={Ito, Tatsuro},
       title={Algebraic combinatorics. {I} {A}ssociation schemes},
   publisher={The Benjamin/Cummings Publishing Co., Inc., Menlo Park, CA},
        date={1984},
        ISBN={0-8053-0490-8},
      review={\MR{882540}},
}

\bib{Bou02}{book}{
      author={Bourbaki, Nicolas},
       title={Lie groups and {L}ie algebras. {C}hapters 4--6},
      series={Elements of Mathematics (Berlin)},
   publisher={Springer-Verlag, Berlin},
        date={2002},
        ISBN={3-540-42650-7},
         url={https://doi.org/10.1007/978-3-540-89394-3},
        note={Translated from the 1968 French original by Andrew Pressley},
      review={\MR{1890629}},
}

\bib{Bou05}{book}{
      author={Bourbaki, Nicolas},
       title={Lie groups and {L}ie algebras. {C}hapters 7--9},
      series={Elements of Mathematics (Berlin)},
   publisher={Springer-Verlag, Berlin},
        date={2005},
        ISBN={3-540-43405-4},
        note={Translated from the 1975 and 1982 French originals by Andrew
  Pressley},
      review={\MR{2109105}},
}

\bib{Bou89}{book}{
      author={Bourbaki, Nicolas},
       title={Lie groups and {L}ie algebras. {C}hapters 1--3},
      series={Elements of Mathematics (Berlin)},
   publisher={Springer-Verlag, Berlin},
        date={1989},
        ISBN={3-540-50218-1},
        note={Translated from the French original, Reprint of the 1975
  edition},
      review={\MR{979493}},
}

\bib{Bro95}{article}{
      author={Broer, Abraham},
       title={The sum of generalized exponents and {C}hevalley's restriction
  theorem for modules of covariants},
        date={1995},
        ISSN={0019-3577},
     journal={Indag. Math. (N.S.)},
      volume={6},
      number={4},
       pages={385\ndash 396},
         url={https://doi.org/10.1016/0019-3577(96)81754-X},
      review={\MR{1365182}},
}

\bib{Car72}{book}{
      author={Carter, Roger~W.},
       title={Simple groups of {L}ie type},
      series={Pure and Applied Mathematics},
   publisher={John Wiley \& Sons, London-New York-Sydney},
        date={1972},
      volume={28},
      review={\MR{0407163}},
}

\bib{Con85}{article}{
      author={Conway, John~H.},
       title={A simple construction for the {F}ischer-{G}riess monster group},
        date={1985},
        ISSN={0020-9910},
     journal={Invent. Math.},
      volume={79},
      number={3},
       pages={513\ndash 540},
         url={https://doi.org/10.1007/BF01388521},
      review={\MR{782233}},
}

\bib{DMPSVC19}{article}{
      author={De~Medts, Tom},
      author={Peacock, Simon~F.},
      author={Shpectorov, Sergey},
      author={Van~Couwenberghe, Michiel},
       title={Decomposition algebras and axial algebras},
        date={2020},
     journal={J. Algebra},
        note={To appear. Available at arXiv:1905.03481 [math.RA]},
}

\bib{DR17}{article}{
      author={De~Medts, Tom},
      author={Rehren, Felix},
       title={Jordan algebras and 3-transposition groups},
        date={2017},
        ISSN={0021-8693},
     journal={J. Algebra},
      volume={478},
       pages={318\ndash 340},
         url={https://doi.org/10.1016/j.jalgebra.2017.01.025},
      review={\MR{3621676}},
}

\bib{FH91}{book}{
      author={Fulton, William},
      author={Harris, Joe},
       title={Representation theory. {A} first course.},
      series={Graduate Texts in Mathematics},
   publisher={Springer-Verlag, New York},
        date={1991},
      volume={129},
        ISBN={0-387-97527-6; 0-387-97495-4},
         url={https://doi.org/10.1007/978-1-4612-0979-9},
      review={\MR{1153249}},
}

\bib{GG15}{article}{
      author={Garibaldi, Skip},
      author={Guralnick, Robert~M.},
       title={Simple groups stabilizing polynomials},
        date={2015},
     journal={Forum Math. Pi},
      volume={3},
       pages={e3, 41},
         url={https://doi.org/10.1017/fmp.2015.3},
      review={\MR{3406824}},
}

\bib{GP00}{book}{
      author={Geck, Meinolf},
      author={Pfeiffer, G{\"{o}}tz},
       title={Characters of finite {C}oxeter groups and {I}wahori-{H}ecke
  algebras},
      series={London Mathematical Society Monographs. New Series},
   publisher={The Clarendon Press, Oxford University Press, New York},
        date={2000},
      volume={21},
        ISBN={0-19-850250-8},
      review={\MR{1778802}},
}

\bib{Gri82}{article}{
      author={Griess, Robert~L., Jr.},
       title={The friendly giant},
        date={1982},
        ISSN={0020-9910},
     journal={Invent. Math.},
      volume={69},
      number={1},
       pages={1\ndash 102},
         url={https://doi.org/10.1007/BF01389186},
      review={\MR{671653}},
}

\bib{HRS15}{article}{
      author={Hall, Jonathan~I.},
      author={Rehren, Felix},
      author={Shpectorov, Sergey},
       title={Universal axial algebras and a theorem of {S}akuma},
        date={2015},
        ISSN={0021-8693},
     journal={J. Algebra},
      volume={421},
       pages={394\ndash 424},
         url={https://doi.org/10.1016/j.jalgebra.2014.08.035},
      review={\MR{3272388}},
}

\bib{Hum72}{book}{
      author={Humphreys, James~E.},
       title={Introduction to {L}ie algebras and representation theory},
      series={Graduate Texts in Mathematics},
   publisher={Springer-Verlag, New York-Berlin},
        date={1972},
      volume={9},
      review={\MR{0323842}},
}

\bib{Iva09}{book}{
      author={Ivanov, Alexander~A.},
       title={The {M}onster group and {M}ajorana involutions},
      series={Cambridge Tracts in Mathematics},
   publisher={Cambridge University Press, Cambridge},
        date={2009},
      volume={176},
        ISBN={978-0-521-88994-0},
         url={https://doi.org/10.1017/CBO9780511576812},
      review={\MR{2503090}},
}

\bib{Lam99}{book}{
      author={Lam, Tsit-Yuen},
       title={Lectures on modules and rings},
      series={Graduate Texts in Mathematics},
   publisher={Springer-Verlag, New York},
        date={1999},
      volume={189},
        ISBN={0-387-98428-3},
         url={https://doi.org/10.1007/978-1-4612-0525-8},
      review={\MR{1653294}},
}

\bib{Ree98}{article}{
      author={Reeder, Mark},
       title={Zero weight spaces and the {S}pringer correspondence},
        date={1998},
        ISSN={0019-3577},
     journal={Indag. Math. (N.S.)},
      volume={9},
      number={3},
       pages={431\ndash 441},
         url={https://doi.org/10.1016/S0019-3577(98)80010-4},
      review={\MR{1692153}},
}

\bib{Stei68}{book}{
      author={Steinberg, Robert},
       title={Lectures on {C}hevalley groups},
   publisher={Yale University, New Haven, Conn.},
        date={1968},
        note={Notes prepared by John Faulkner and Robert Wilson},
      review={\MR{0466335}},
}

\end{biblist}
\end{bibdiv}

\end{document}